\documentclass[preprint,12pt]{elsarticle}

\usepackage{amssymb}
\usepackage{amsmath}
\usepackage{tikz}%for nice drawings and plots
\usetikzlibrary{patterns}
\usetikzlibrary{plotmarks}
\usetikzlibrary{arrows.meta, decorations.pathreplacing}
\usepackage{comment}
\usepackage{wrapfig}
\usetikzlibrary{arrows}
\usepackage{pgfplots}
\usetikzlibrary{patterns}
\usepgfplotslibrary{external}
\usepackage{xcolor}
\usepackage{ulem}
\usepackage{graphicx,epstopdf}
\usepackage[table]{xcolor}
\usepackage[caption=false]{subfig}
\usepackage{amssymb} 
\usepackage{amsthm}
\usepackage{algorithm}
\usepackage{algpseudocode}
\theoremstyle{remark}

\usepackage{nicefrac}

\newtheorem{theorem}{Theorem}

\theoremstyle{definition}

\theoremstyle{remark}
\newtheorem{remark}{Remark}

\usepackage[colorlinks=true,linkcolor=blue,citecolor=blue,urlcolor=blue]{hyperref}

\journal{Journal of Computational and Applied Mathematics}

\begin{document}

\begin{frontmatter}



\title{On the role of relaxation and acceleration in the non-overlapping Schwarz alternating method for coupling}

\author[aff1]{Giulia Sambataro \corref{cor1}}
%\ead{giulia.sambataro@inria.fr}
\cortext[cor1]{Corresponding author. Email: giulia.sambataro@inria.fr}
\author[aff2]{ Irina Tezaur}
\affiliation[aff1]{organization={Université de Strasbourg, CNRS, Inria, IRMA},%Department and Organization
            addressline={7 rue René-Descartes}, 
            city={Strasbourg Cedex},
            postcode={67084}, 
            country={France}}
\affiliation[aff2]{organization={Sandia National Laboratories},%Department and Organization
            addressline={7011 East Ave}, 
            city={Livermore},
            postcode={94550}, 
            state={CA},
            country={United States}}
            
%% Abstract
\begin{abstract}
 The purpose of this paper is to study the influence of relaxation and acceleration techniques on the convergence behavior of the non-overlapping Schwarz algorithm with alternating Dirichlet-Neumann transmission conditions in the context of domain decomposition- (DD-) based coupling.  % -- acting on transmission conditions at the interface boundary -- . 
 %We focus, in particular, on the Dirichlet-Neumann Schwarz scheme. 
 After demonstrating that the multiplicative Schwarz scheme can be formulated as a fixed-point iteration, we explore, both theoretically and numerically, two promising techniques for speeding up the method: (i) Aitken acceleration and (ii) Anderson acceleration.   In the process, we derive a robust and efficient adaptive variant of Anderson acceleration, termed ``Anderson with memory adaptation''.  We compare the proposed acceleration strategies to the well-known classical relaxed Dirichlet-Neumann Schwarz alternating method.   
 %After introducing the classical relaxation update of the interface data, we analyze Aitken and Anderson accelerations techniques, originally proposed in the field of series acceleration and fixed-point iteration acceleration.
     %We present both theoretical insights and numerical evidence illustrating how different choices of relaxation parameters and acceleration schemes modify the mathematical properties of the transmission operator at the interface, thereby influencing the convergence and performance of the Schwarz algorithm. 
     Our results suggest that, while Aitken-accelerated Schwarz is the best approach in terms of efficiency and robustness when considering two sub-domain DDs, Anderson-accelerated Schwarz is the method of choice in larger multi-domain settings.  
\end{abstract}


%\item We analyze the impact of relaxation and acceleration on non-overlapping Schwarz
%\item We develop an adaptive variant of Anderson acceleration to improve Schwarz robustness
%\item Theoretical analyses are performed on a one-dimensional study case
%\item Numerical studies explore the method’s performance when coupling up to 5 subdomains

%% Keywords

\end{frontmatter}



\section{Introduction}
\label{sec:intro}
Domain decomposition (DD) methods have become a fundamental tool for the numerical solution of large-scale partial differential equations (PDEs)\footnote{
A nice overview of DD methods, including both overlapping and non-overlapping approaches, is provided by Nataf in \cite{dolean2015introduction}, which offers an accessible yet thorough introduction to these methods, with a particular emphasis on their applications to PDEs.}.    
A popular DD-based approach is the Schwarz alternating method, which is particularly attractive for parallel computing, as it: (i) requires only sub-domain-local solves, (ii) performs local communications between neighboring sub-domains to exchange boundary condition (BC) information at sub-domain interfaces, and (iii) can be implemented in a minimally-intrusive way into existing high-performance computing codes.
While the Schwarz alternating method is most commonly known in the linear solver literature, where it is often used as a preconditioner that can accelerate the convergence of Krylov iterative solvers \cite{gander2008schwarz}, several recent works \cite{mota2017schwarz, Mota:2022,  
barnett2022schwarz,
wentland2025role, 
rodriguez2025transmission,
tezaur2026hybridcouplingoperatorinference} have shown that the method can also be used as a concurrent multiscale coupling method capable of stitching together regions with different mesh resolutions,
different element types, different time integration schemes (e.g., implicit and explicit \cite{Mota:2022, tezaur2026hybridcouplingoperatorinference}) and different models (e.g., finite element and data-driven models \cite{barnett2022schwarz, 
wentland2025role,
moore2024domaindecompositionbasedcouplingoperator, rodriguez2025transmission, 
tezaur2026hybridcouplingoperatorinference}), all
without introducing any artifacts exhibited by alternative coupling methods.  

The Schwarz alternating method can be applied in the context of both overlapping and non-overlapping DDs, each with their own inherent advantages and disadvantages: whereas the overlapping Schwarz variant has more favorable convergence properties, the non-overlapping Schwarz scheme is more flexible, making it particularly attractive for multi-physics and multi-material applications because it allows completely independent meshing of the coupled sub-domains \cite{Mota:2022, tezaur2026hybridcouplingoperatorinference}.  
A number of authors, including Schwarz himself  \cite{Schwarz:1870}, have studied the theoretical properties of the overlapping
Schwarz alternating method \cite{Sobolev:2006,Mikhlin:1951,lions1988schwarz}, demonstrating convergence with Dirichlet-Dirichlet transmission conditions provided that the overlap region is non-empty. Extensions to the case of non-overlapping DDs were made concurrently by Lions \cite{Lions:1990} and Zanolli \textit{et al.} \cite{zanolli1987domain, funaro1988iterative}, who showed that either Robin-Robin or alternating Dirichlet-Neumann transmission conditions can yield a provably convergent algorithm.
%Robin-Robin transmission conditions for a non-overlapping DD were provided by Lions in \cite{Lions:1990}. %Lions proposed to use Robin conditions to obtain a convergent algorithm, and he added an example of analytically chosen Robin coefficients, opening the way to the first Optimized Schwarz methods. This method offers great advantage for convergence, but is relies on analytical choice of Robin parameters and its use is a bit limited to academic problems. % to verify better!!
%Concurrently, a provably-convergent non-overlapping Schwarz method based on relaxed alternating Dirichlet-Neumann transmission conditions was derived by Zanolli \textit{et al.} in \cite{zanolli1987domain, funaro1988iterative}. 
These foundational contributions paved the way for later developments, e.g., the so-called non-overlapping Schwarz waveform relaxation method, particularly relevant for time-dependent problems, in which a Dirichlet-Dirichlet iteration in both sub-domains is followed by a Neumann-Neumann iteration \cite{kwok2014neumann, mandalwaveform2017, ganderwaveform2014}.  
%, a Dirichlet-Dirichlet iteration with relaxation is first performed in both sub-domains, followed by a Neumann-Neumann iteration. This Schwarz variant has been demonstrated to offer convergence speedups compared to a classical non-overlapping Schwarz method, as discussed in \cite{mandalwaveform2017, ganderwaveform2014} 

%including other non-overlapping variants of the method: the method is particularly attractive for multi-physics and multi-material applications because it allows completely independent meshing of the coupled sub-domains.\cite{Mota:2022, tezaur2026hybridcouplingoperatorinference}.  

The present work is concerned with the non-overlapping variant of the Schwarz alternating method as a mechanism to perform flexible concurrent DD-based coupling in the context of PDE-based modeling and simulation (mod/sim) of scenarios where overlapping DD-based methods are not feasible.  A well-known disadvantage of non-overlapping Schwarz is the fact that it is slower and less robust than overlapping Schwarz, often requiring more iterations to reach convergence or not converging at all.  This is especially problematic for PDEs with highly heterogeneous coefficients, discretized using fine meshes, and/or posed on long and thin sub-domains \cite{gander2015optimized}.
%furthermore,
%more iterations are typically required to achieve convergence than the overlapping Schwarz variant \irinanote{Should add a reference here too}. 

To mitigate the non-overlapping Schwarz convergence issues described above, various  authors have
%Hence, subsequent research %, including that of Deng and Russier \cite{deng2003nonoverlapping} and Lui \cite{lui2001accelerated},  
%has 
explored acceleration strategies and algorithmic improvements for non-overlapping alternating Schwarz-based DD methods.  
%, which are particularly relevant in practical computations where overlapping sub-domains are not feasible.  
The so-called optimized Schwarz methods define a well-known class of accelerated Schwarz techniques based on the Robin-Robin formulation \cite{Lions:1990}, in which 
optimal values of the parameters pre-multiplying the Dirichlet and Neumann parts of the Robin BCs are derived \cite{gander2008schwarz}.  The main disadvantage 
of these approaches is that the optimal parameter values are generally problem-specific and require lengthy analytical derivations.
%: optimal parameters for Dirichlet and Neumann parts are often derived analytically and are problem-specific. For a detailed review, see .  
%in which optimal values of the parameters multiplying the Dirichlet and Neumann parts of the Robin transmission conditions are derived.  These derivations are often analytical and hence problem-specific.  For a detailed overview of existing optimized Schwarz methods and their theoretical foundations, the interested reader is referred to \cite{gander2008schwarz}.   
%Optimized Schwarz methods represent a more recent development in this area, aiming to improve convergence rates by using optimized transmission conditions at the interfaces between sub-domains. Gander provides an in-depth discussion of these techniques and their theoretical foundations in \cite{gander2008schwarz}, which has become a standard reference in the field.
Within the alternating Dirichlet-Neumann Schwarz framework, it has been shown that convergence can be accelerated through the addition of a 
relaxation step to the Dirichlet sub-problem \cite{toselli2004domain}.  Although several authors,   
%Within the relaxed alternating Dirichlet-Neumann Schwarz framework, various authors, 
e.g., %More recent studies, including 
\cite{cote2005comparison} and \cite{kwok2014neumann}, have investigated automatic strategies for choosing the relaxation parameter, the method is not particularly robust 
with respect to the choice of this parameter; moreover, the optimal value of this parameter is often very problem- and DD-specific.
%convergence can be extremely 
%sensitive to the value of this parameter and choosing the optimal value for a given problem is unclear.

The present work explores and analyzes both numerically and theoretically two alternative approaches for accelerating the Dirichlet-Neumann form of the non-overlapping Schwarz alternating method: Aitken \cite{deparis2004numerical} and Anderson \cite{walker2011anderson} acceleration.  
%order to apply these techniques, it is necessary to first re-write the classical Schwarz algorithm as a fixed-point iteration.  
We choose to focus on the Dirichlet-Neumann (rather than the Robin-Robin) variant of Schwarz, as Dirichlet and Neumann conditions are more readily available in physics-based PDE codes; for example, all solid mechanics codes come equipped with Neumann (traction) BCs, whereas Robin BCs are generally not used within this domain, as they do not have a physical interpretation.  
Since Aitken and Anderson acceleration are designed for speeding up solvers based on fixed-point iterations, applying these methods to Schwarz requires rewriting the classical Schwarz algorithm as a fixed-point iteration.
To the authors' knowledge, the earliest reference that considers an Aitken-accelerated Schwarz method is \cite{Garbey:2002}; however, like most references on this topic, e.g., \cite{Berenguer:2022}, attention is restricted to the overlapping DD case. 
Prior attempts to speed up Schwarz using Anderson acceleration, e.g., \cite{walker2011anderson}, have also focused on the overlapping variant of this method.  
Within the non-overlapping DD setting, Aitken acceleration 
has been shown as a valid alternative for accelerating the Schwarz algorithm in the context of fluid-structure interaction problems in \cite{deparis2006domain}.  
We are not aware of references that attempt to apply Anderson acceleration to non-overlapping Schwarz.  We are also not aware of publications that perform 
a systematic comparison of classical relaxation, Aitken acceleration and Anderson acceleration as mechanisms to speed up the non-overlapping Schwarz algorithm.

%We present both theoretical insights and numerical evidence illustrating how different choices of relaxation parameters and acceleration schemes modify the mathematical properties of the transmission operator at the interface, thereby influencing the convergence and performance of the Schwarz algorithm. 

Toward this effect, the main novel contributions of this work are as follows: (i) the first (to our knowledge) adaptation of Anderson acceleration to the non-overlapping Dirichlet-Neumann Schwarz alternating method and in particular,  the development of an adaptive variant, termed ``Anderson with memory adaptation'', which gives rise to a more robust and efficient Schwarz scheme (see Section \ref{sec:non_ovl}); (ii) some theoretical analyses of  Aitken-accelerated and Anderson-accelerated non-overlapping Dirichlet-Neumann Schwarz alternating method applied to a one-dimensional (1D)
%specific 
study case (see Section \ref{sec:non_ovl}); %, including derivations that demonstrate how the relaxation/acceleration schemes considered modify the mathematical properties of the transmission operator at the interface; %, and calculations of the methods' the convergence rates; 
and (iii) numerical studies comparing Aitken-accelerated and Anderson-accelerated Schwarz to Schwarz with classical relaxation, which verify numerically the theoretical properties derived herein on problems with DDs consisting of as many as five sub-domains (see Section \ref{sec:num_results}). %\irinanote{TODO check that Giulia agrees with enumerated contributions}.  The remainder of this paper is structured as follows. In Section \ref{sec:non_ovl}, we introduce the mathematical notation, present the classical alternating Dirichlet-Neumann Schwarz method in the non-overlapping case, and  summarize the different relaxation/acceleration techniques evaluated.  Section \ref{sec:num_results} presents a collection of theoretical and numerical results, which demonstrate the impact of acceleration on the convergence of the Schwarz algorithm. A few concluding remarks are provided in Section \ref{sec:conclusions}.
Conclusions and some avenues for future work are discussed in Section \ref{sec:conclusions}, after the main exposition.

\section{Non-overlapping Schwarz for DD-based coupling}
\label{sec:non_ovl}
\subsection{Problem formulation}
% Explain why non-overlapping domain decomposition
We begin by introducing our non-overlapping DD notation. We consider a domain $\Omega \subset \mathbb{R}^d$ with Lipschitz boundary $\partial \Omega$ and $x \in \bar{\Omega}$ a space variable; in the present work, attention is restricted to the case when $d=2$ (Figure \ref{fig:non_ovl}(a)). 
%We also consider a Hilbert space 
Let $\mathcal{X}= [H_0^1(\Omega)]^d$ denote a Hilbert space endowed with the norm $\|v\|_{\mathcal{X}}:=\|v\|_{H^1(\Omega)}$.  
Consider a generic elliptic PDE with variational form 
%We introduce the notation for stationary problems described by elliptic PDEs: we assume that the global problem can be written in variational form as follows:
\begin{equation}
\mathcal{R}(u,v)=0 \quad \forall v \in \mathcal{X},
\label{eq:glo_pb}
\end{equation}
where $\mathcal{R}(\cdot, \cdot)$ is a generic nonlinear operator.  
We assume that the problem \eqref{eq:glo_pb} is well-posed and has a unique solution.
 Let \( \mathbf{n} \) denote the outward unit normal on \( \partial \Omega \).

Suppose we are given a non-overlapping decomposition of $\Omega$ into $N_{\rm{dd}}$ open sub-domains $\{\Omega_i\}_{i=1}^{N_{\rm{dd}}}$ such that
\[
\Omega = \bigcup_{i=1}^{N_{\text{dd}}} \overline{\Omega}_i, \quad \Omega_i\cap \Omega_j= \emptyset \quad \text{for } i \ne j.
\] 
Without loss of generality, we will assume in the majority of our presentation that $N_{\rm{dd}} = 2$, that is, the domain $\Omega$ is decomposed into two non-overlapping sub-domains, as shown in Figure \ref{fig:non_ovl}(b)\footnote{Note that the $N_{\rm{dd}}>2$ case is explored numerically in Section \ref{par:multi_domain}.}.  %; although, in the numerical section, also $N_{\rm{dd}}\geq2$ partitions are taken into account.
The interface that separates the two sub-domains is denoted by $\Gamma$, so that  $\Gamma := \bar{\Omega}_1 \cap \bar{\Omega}_2$.
We assume that the boundaries $\Gamma$, $\partial \Omega_1 \setminus \Gamma$ and $\partial \Omega_2 \setminus \Gamma$ all have a non-vanishing $(d-1)$-dimensional measure.
The vectors $\mathbf{n}_i$, $i=1,2$, are the unit outward normals of $\Omega_i$ on $\Gamma$, so that $\mathbf{n}_1=-\mathbf{n}_2$. The normal derivative is denoted by $\frac{\partial{u}}{\partial \mathbf{n}}=\nabla u \cdot \mathbf{n} \in H^{-1/2}(\Gamma)$.
We denote the external boundaries, also referred to as the system boundaries, by $\partial \Omega_i^{\rm{e}}=\partial \Omega_i \cap \partial \Omega$  for $i=1, \ldots, N_{\rm{dd}}$.
%We define the external boundaries $\Omega^{\rm{e}}_i$ as $\partial \Omega_i^{\rm{e}}=\partial \Omega_i \cap \partial \Omega$ for each component $i=1, \ldots, N_{\rm{dd}}$; on these boundary, Dirichlet or Neumann conditions are derived from the global domain (a possible representation is depicted in Figure \ref{fig:non_ovl}(a), together with a non overlapping decomposition in Figure \ref{fig:non_ovl}(b).
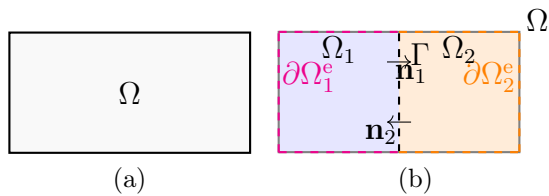
\begin{figure}[h!]
\centering
\subfloat[]{
\begin{tikzpicture}[scale=0.8]
   \fill[gray!5] (-2,-1) rectangle (2,1); % Global domain Omega
    
    % Domain boundary
    \draw[thick] (-2,-1) rectangle (2,1);

    % Labels for sub-domains
    \node at (0,0) {$\Omega$};
\end{tikzpicture}
}
\subfloat[]{
\begin{tikzpicture}[scale=0.8]
% Draw the left and right sub-domains with reduced height
    \fill[blue!10] (-2,-1) rectangle (0,1); % Omega_1
    \fill[orange!15] (0,-1) rectangle (2,1);   % Omega_2

    % Interface boundary (vertical dashed line)
    \draw[thick, dashed] (0,-1) -- (0,1) node[below right] {$\Gamma$};
    % Domain boundary
    \draw[thick, gray] (-2,-1) rectangle (2,1);

    % Labels
    \draw[thick, dashed, magenta] (-2,-1) -- (-2,1);
    \draw[thick, dashed, magenta] (-2,1) -- (0,1);
     \draw[thick, dashed, magenta] (-2,-1) -- (0,-1);
     \node at (-1.5,0.25) {\textcolor{magenta}{$\partial \Omega_1^{\rm{e}}$}};

    \draw[thick, dashed, orange] (2,-1) -- (2,1);
    \draw[thick, dashed, orange] (0,1) -- (2,1);
     \draw[thick, dashed, orange] (0,-1) -- (2,-1);
     \node at (1.5,0.25) {\textcolor{orange}{$\partial \Omega_2^{\rm{e}}$}};
     
    \node at (-1,0.7) {$\Omega_1$};
    \node at (1,0.7) {$\Omega_2$};
    \node at (2.3,1.2) {$\Omega$};
    \node at (0.2,0.3) {$\mathbf{n}_1$};
    \node at (-0.3,-0.7) {$\mathbf{n}_2$}; 
    % Interaction arrows across Gamma
    \draw[->] (-0.2,0.5) -- (0.2,0.5);
    \draw[->] (0.2,-0.5) -- (-0.2,-0.5);
\end{tikzpicture}
}
\caption{(a) Global domain. (b) Sketch of a non-overlapping decomposition of the domain $\Omega$ into $\Omega_1$ and $\Omega_2$ with interface $\Gamma$.}
\label{fig:non_ovl}
\end{figure}
%We introduce 

After introducing the following spaces
\begin{equation} \label{eq:spaces} 
\mathcal{X}_i=\{v|_{\Omega_i}, v \in \mathcal{X}\}= [H^1(\Omega_i)]^d , \; 
\mathcal{X}_{i,0}=\{v \in \mathcal{X}_i \, | \, v =0 \, \text{at} \, \Gamma_{i,\rm{dir}}\}, 
\end{equation}
(where $\Gamma_{i,\rm{dir}} \subset \partial \Omega_i$ is the local boundary associated with Dirichlet boundary conditions for subdomain $\Omega_i$),
the local sub-problems corresponding to the global problem \eqref{eq:glo_pb} can be written in variational form as follows:
\begin{equation}
\mathcal{R}_i(u_i,v)=0 \quad \forall v \in \mathcal{X}_{i,0},
\label{eq:loc_pb}
\end{equation}
where $\mathcal{R}_i:\mathcal{X}_i\rightarrow \mathbb{R}$ is the restriction of the global operator $\mathcal{R}$ to $\Omega_i$.
We equip the local Hilbert spaces $\mathcal{X}_i$ with the norms $\|v\|_{\mathcal{X}_i}=\|v\|_{H^1(\Omega_i)}$ for $i=1, \ldots, N_{\rm dd}$.  %we also adopt $\|g\|_{\mathcal{W}}=\|g\|_{L^2(\Gamma)}$ 
Herein, we will consider finite element (FE) discretizations of the global and local domains.  Within this context, let 
%We consider finite element (FE) discretizations of the global and local domains:  we denote by 
$\mathcal{V}\subset \mathcal{X}$ and $\mathcal{V}_i\subset \mathcal{X}_i$  denote the discrete global and local FE spaces, respectively, 
and assume that each sub-domain 
%For each sub-domain we introduce 
%is discretized %with a $\mathbb{P}_k$ discretization 
discretization is based on 
%with 
nodal FE bases $\{\boldsymbol{\varphi}^{j}_{i}\}_{j=1}^{N_i}$.  The discrete finite element spaces %such that the FE spaces %$\mathcal{V}_i$ 
can 
now 
be expressed as $\mathcal{V}_i=\text{span}\{\boldsymbol{\varphi}_i^j\}_{j=1}^{N_i}$, where $N_i$ is the number of nodes in the $i^{th}$ sub-domain. 

One advantage of the Schwarz alternating method for coupling is that it allows the sub-domains $\Omega_i$ to be discretized by different non-conformal meshes
%We consider the general case of finite element non-conforming discretizations 
(c.f., \cite{mota2025fundamentally,zappon2024reduced}). 
%; as a result, each sub-domain
%: each domain 
%$\Omega_i$ can have its own local representation of the interface $\Gamma$.  
Consequently, each sub-domain can maintain a distinct discrete representation of the shared interface. We therefore introduce the notation $\Gamma_{ij}$ to denote the discrete representation of $\Gamma$ that conforms to the local mesh of $\Omega_i$.

Let $\mathcal{U}_i=\{v|_{\Gamma_{ij}}: v \in \mathcal{V}_i\} \subset{H^{1/2}(\Gamma_{ij})}$  denote the discrete trace, defined as the restriction of the FE space $\mathcal{V}_i$ to the local interface $\Gamma_{ij}$.
%In $\Gamma_{ij}$, we construct a one-dimensional FE approximation and we denote by $N^{\rm{c}}$ the degrees of freedom of the discretized space $\mathcal{U}_i\subset{H^{1/2}(\Gamma_{ij})}$; we have $\mathcal{U}_i=\{v|_{\Gamma_{ij}}: v \in \mathcal{V}_i\} $ as the span of nodal basis functions restricted to $\Gamma_{ij}$.
We denote the standard discrete trace operator onto this space by $\gamma_i: \mathcal{V}_i \to \mathcal{U}_i$ for $i=1,\ldots, N_{\rm{dd}}$. 
To adapt the notation to non-conforming interfaces, we also introduce the composition operator $\gamma_j^i: \mathcal{V}_j \rightarrow \mathcal{U}_i$ such that  $\gamma^i_j(\mathbf{u}_j)=\Pi_{ij}\circ \gamma_j (\mathbf{u}_j)$, where the projection
operator $\Pi_{ij}:\mathcal{U}_{j}\rightarrow \mathcal{U}_i$ is used to compute the trace of the solution $\mathbf{u}_j$ on the interface $\Gamma_{ij}$ of sub-domain $\Omega_i$.  
Finally, we introduce the discrete Neumann operator $\Psi_i : \mathcal{V}_i \rightarrow \mathcal{U}_i^{\star}$, which evaluates the discrete normal derivative on the interface $\Gamma_{ij}$. Acting as the discrete counterpart to the continuous normal trace $\partial_{\mathbf{n}_i} u_i \in H^{-1/2}(\Gamma_{ij})$, the operator $\Psi_i$ is defined in a weak sense and takes values in $\mathcal{U}_i^{\star}$, the formal dual space of $\mathcal{U}_i$.

\subsection{Dirichlet-Neumann Schwarz scheme}
\label{sec:classic_DirNeu}
Herein, we overview the classical Dirichlet-Neumann version of the non-overlapping Schwarz alternating method.  
We describe the algorithm for the simple case $N_{\rm{dd}}=2$, as shown in Figure \ref{fig:non_ovl}(b). The solution to the global problem \eqref{eq:glo_pb} on $\Omega$ is obtained via an iterative process in which a
sequence of sub-domain-local problems is solved, with information propagating between sub-domains through transmission BCs.
Algorithm \ref{alg:DN_relax} summarizes this procedure. 
Let the maximum number of iterations $\texttt{maxit} \in \mathbb{N}_0$ be given, as well as tolerances $\epsilon_{\rm{abs}} >0$ and $\epsilon_{\rm{rel}} >0$.
We define the discrete residual operator $\mathcal{R}_i^{h}:\mathcal{V}_i\times\mathcal{V}_i\rightarrow \mathbb{R}$
as the restriction of the continuous residual operator $\mathcal{R}_i$ \eqref{eq:loc_pb} to the discrete spaces $\mathcal{V}_i$. For simplicity of notation, we will drop the superscript $h$ from this point forward, so that $\mathcal{R}_i(\mathbf{u}_i^{(k)},\mathbf{v})$ denotes the discrete residual computed at Schwarz iteration $k$ (c.f., \eqref{eq:loc_pb}).  
The Schwarz iteration process requires an  initial value of the interface trace at $\Gamma_{21}$, i.e., $\gamma_2(\mathbf{u}_2) \in \mathcal{U}_2$, which we denote as  $\mathbf{g}^{(1)}$, where the superscript points out the fact that these data are used at the first iteration $k=1$. 

\begin{remark}
As clear from \eqref{eq:spaces}, the space $\mathcal{X}_{i,0}$ incorporates homogeneous Dirichlet BCs on the external boundary $\partial \Omega_{i}^{\rm{e}}$, meaning 
%.  The corresponding (possibly non-homogeneous) Dirichlet BCs are then 
they can be imposed strongly by defining the trial space $\mathcal{X}_i$.  Neumann BCs are accounted for through the residual operator $\mathcal{R}_i$. Thus, in Algorithm \ref{alg:DN_relax}, only interface BCs are explicitly written. 
\end{remark}

\begin{algorithm}[h]
\caption{Dirichlet--Neumann scheme with relaxation.}
\label{alg:DN_relax}
\begin{algorithmic}[1] % The number tells where line numbering should start
\Require Initial guess \( \mathbf{g}^{(1)} \), \( \texttt{maxit} \), \( \epsilon_{\text{rel}} \), $\epsilon_{\text{abs}}$, \( \rho \in (0,1] \), \( k = 0 \)
%\Ensure Outputs: \( \mathbf{u}_1, \mathbf{u}_2 \)
\Repeat
 \State Increment \( k \leftarrow k + 1 \)
   \State \textbf{Dirichlet step in \( \Omega_1 \):} solve \label{alg_line:dirichlet}
    \[
    \begin{cases}
    \mathcal{R}_1(\mathbf{u}_1^{(k)}, \mathbf{v}) = 0 \quad \forall \mathbf{v} \in \mathcal{X}_{1,0} & \text{in } \Omega_1 \\
    \gamma_1(\mathbf{u}_1^{(k)}) = \Pi_{12}\mathbf{g}^{(k)} & \text{on } \Gamma_{12}
    \end{cases}
    \]
  \State \textbf{Neumann step in \( \Omega_2 \):} solve \label{alg_line:neumann}
    \[
    \begin{cases}
    \mathcal{R}_2(\mathbf{u}_2^{(k)}, \mathbf{v}) = 0 
    \quad \forall \mathbf{v} \in \mathcal{X}_{2,0} &
    \text{in } \Omega_2 \\
   \Psi_2(\mathbf{u}_2^{(k)}) =- \Pi_{21} \Psi_1(\mathbf{u}_1^{(k)}) & \text{on} \, \Gamma_{21} 
    \end{cases}
    \]
    \State \textbf{Classical relaxation}: Update
    $\mathbf{g}^{(k+1)} = (1 - \rho) \mathbf{g}^{(k)} + \rho \,  \gamma_2(\mathbf{u}_2^{(k)})$ \label{line:relax}
   \Until {$e^{(k)}_{\rm{abs}}>\epsilon_{\rm{abs}}$ \textbf{and} $e^{(k)}_{\rm{rel}}>\epsilon_{\rm{rel}}$ \textbf{and} \( k < \texttt{maxit} \)} \label{line:conv_criterion}
\State \textbf{Output:} $\mathbf{u}_1^{(k)}, \mathbf{u}_2^{(k)}$
\end{algorithmic}
\end{algorithm}

%\irinanote{I don't know if the ``Ensure'' line in the algorithms is necessary and if it is intended to give the outputs. You give the outputs always at the end of the algorithm.  I therefore removed the ``Ensure'' lines.}

The interaction between sub-domains is encoded in the compatibility constraint on the continuity of the solution on interface $\Gamma_{12}$, and on the continuity of normal fluxes on $\Gamma_{21}$.
Since the Dirichlet-Neumann iteration procedure gives an infinite sequence, it is necessary to define convergence criteria for the algorithm.  The convergence criteria employed herein are given in Algorithm \ref{alg:DN_relax}, line $6$. Here, $\epsilon_{\rm{abs}} $ and $\epsilon_{\rm{rel}}$ are pre-specified absolute and relative tolerances, respectively, and 
%We consider the following interface absolute $e_{\rm{abs}}^{(k)}<\epsilon_{\rm{abs}}$ and relative errors $e^{(k)}_{\rm{rel}}<\epsilon_{\rm{rel}}$ as stopping criteria for the pre-specified tolerances $\epsilon_{\rm{abs}},\epsilon_{\rm{rel}}$:
\begin{equation}
\begin{aligned}
e^{(k)}_{\text{abs}} := 
\sqrt{
\left\| \gamma_1(\mathbf{u}_1^{(k)})- \gamma_1(\mathbf{u}_1^{(k-1)}) \right\|_2^2 
+ 
\left\| \gamma_2(\mathbf{u}_2^{(k)}) - \gamma_2(\mathbf{u}_2^{(k-1)}) \right\|_2^2
},\\[2pt]
e^{(k)}_{\text{rel}} := 
\sqrt{
\nicefrac{ \left\| \gamma_1(\mathbf{u}_1^{(k)}) - \gamma_1(\mathbf{u}_1^{(k-1)}) \right\|_2^2 }{ \left\| \gamma_1(\mathbf{u}_1^{(k)}) \right\|_2^2 }
+ 
\nicefrac{ \left\| \gamma_2(\mathbf{u}_2^{(k)}) - \gamma_2(\mathbf{u}_2^{(k-1)}) \right\|_2^2 }{ \left\| \gamma_2(\mathbf{u}_2^{(k)}) \right\|_2^2 }
}.
    \label{eq:error_def}
\end{aligned}
\end{equation}
We choose to adopt stopping criteria based solely on interface quantities, as the Dirichlet–Neumann iteration is governed by the satisfaction of compatibility conditions on $\Gamma_{12}$ and $\Gamma_{21}$.  We note that alternative convergence criteria based on full sub-domain solution norms, such as those in 
\cite{mota2017schwarz, Mota:2022}, are also possible; however, 
%: for this reason, we adopt stopping criteria based solely on interface quantities, which provide a sufficient measure of convergence of the coupling iteration. 
it can be shown that control of the error/traces on the interfaces implies control of the error in the full subdomains, hence interface-based stopping criteria are equivalent to sub-domain-based ones up to constants depending on the problem and norms used \cite{QuarteroniValli1999, ToselliWidlund2005}.  
%Once the interface traces converge, the corresponding sub-domain solutions follow by stability of the underlying boundary value problems \irinanote{What do you mean by ``stability'' here?  I don't think it is the right term to use.}.
%  \irinanote{Is there a reference that shows equivalence regardless whether you use interface of full sub-domain convergence criteria?}
%Alternative criteria based on full sub-domain norms, such as those in 
%\cite{mota2017schwarz, Mota:2022}, the authors impose a stronger but not required notion of convergence involving the solution norm at each entire sub-domain.
In all of our numerical experiments, the proposed interface-based criteria were sufficient to reliably detect convergence.

At this point, it is worthwhile to point the reader's attention to   line \ref{line:relax} in Algorithm \ref{alg:DN_relax}, which introduces a common Schwarz acceleration technique referred to herein as ``classical relaxation''.  It has been shown in several references, e.g., \cite{zanolli1987domain, funaro1988iterative}, that adding relaxation of this form to the Dirichlet step of the Schwarz iteration procedure can speed up, and, in some cases, enable, convergence.  
The value $\rho \in (0,1]$ in Algorithm \ref{alg:DN_relax}, line \ref{line:relax}, known as the ``relaxation parameter'', is a constant pre-specified value, and one of the tuning knobs in the algorithm.   
%In order to ensure and speed up the convergence of the
%algorithm, a classical relaxation \cite{zanolli1987domain, funaro1988iterative} of the Dirichlet data at the interfaces is proposed at line \ref{line:relax} at the end of each cycle:
%we denote as $\rho \in (0,1]$ the relaxation parameter for the update of the transmission BCs. 
%We emphasize that, in the classical relaxation approach, $\rho$ is a constant, pre-specified value.  
The scheme in Algorithm \ref{alg:DN_relax} has been shown to converge in the case of coercive elliptic problems, as well as for problems described by skew-symmetric bilinear forms where the skew-symmetric part is small enough, c.f., \cite{marini1989relaxation,quarteroni1999domain}.

\begin{remark} \label{remark:additive}
   The approach described in Algorithm \ref{alg:DN_relax} solves the Dirichlet and Neumann problems (lines \ref{alg_line:dirichlet} and \ref{alg_line:neumann}, respectively) sequentially; that is, the $\Omega_2$ solve at iteration $k$ cannot proceed until $\mathbf{u}_1^{(k)}$ has been calculated.
    %The scheme in Algorithm \ref{alg:DN_relax} computes the solution at the Neumann step by using the updated solution $\mathbf{u}_1^{(k)}$ in the first sub-problem: 
    This procedure is commonly referred to in the literature as ``multiplicative Schwarz''. An alternative option,  referred to as ``additive Schwarz'', relies on the previously computed solution $\mathbf{u}_1^{(k-1)}$ when solving the Neumann problem in line \ref{alg_line:neumann} of Algorithm \ref{alg:DN_relax}, making the algorithm embarrassingly parallel. It can be shown that the multiplicative and additive Schwarz variants converge to the same solution, though the latter method often requires more iterations to reach convergence \cite{ToselliWidlund2005, SmithBjorstadGropp1996}. Herein, we focus on the multiplicative Schwarz scheme, %due to 
    %its more favorable convergence properties, together with our ability 
    %the fact that it is possible to 
    as it is possible to 
    formulate various relaxed and accelerated versions of this method as a fixed-point iteration (see Section \ref{sec:fixed_point} and Remark \ref{remark:additive_fp}), which enables certain theoretical analyses (Section \ref{sec:acceleration}).  
    %We opt to focus on the multiplicative variant, as it not only exhibits a faster convergence rate compared to the additive one (\cite{carlberg2019network}) but is also more intuitive.
\end{remark}

\subsection{Reformulation of multiplicative Schwarz as a fixed-point iteration method}
\label{sec:fixed_point}
The Dirichlet and Neumann BCs in Algorithm \ref{alg:DN_relax}
 can be seen as the outputs of suitable Dirichlet-to-Neumann (DtN)  and Neumann-to-Dirichlet  (NtD) maps, which we define as
 $\Lambda: \mathcal{U}_1 \rightarrow \mathcal{U}_1^{\star}$,
$ \Upsilon: \mathcal{U}_2^{\star} \rightarrow \mathcal{U}_2$
 respectively, where 
 \begin{subequations}
     \begin{align}
 \Lambda: \Pi_{12}\gamma_2(\mathbf{u}_2) \mapsto \Psi_1(\mathbf{u}_1),\\
 \Upsilon: -\Pi _{21}\Psi_1(\mathbf{u}_1) \mapsto \gamma_2(\mathbf{u}_2).
 \label{eq:DtN_NtD}
 \end{align}
The DtN and NtD operators are understood as solution operators corresponding to their respective sub-domain problems. After introducing these maps, we can rewrite 
 \[
 \gamma_2(\mathbf{u}_2^{(k)})=\Upsilon(-\Pi_{21}\Psi_1(\mathbf{u}_1^{(k)}))=\Upsilon(-\Pi_{21}(\Lambda(\Pi_{12}\mathbf{g}^{(k)}))),
 \]
 so that, from the relaxation formula in line \ref{line:relax} of Algorithm \ref{alg:DN_relax}, we obtain
 \begin{equation}
     \mathbf{g}^{(k+1)}=(1-\rho)\, \mathbf{g}^{(k)}\, + \, \rho \Upsilon\Big(-\Pi_{21}\big(\Lambda(\Pi_{12}\mathbf{g}^{(k)})\big)\Big),
     \label{eq:fixed_point_relax}
 \end{equation}
    a fixed-point equation on $\mathbf{g}$.
   Now, we define the operator $T: \mathcal{U}_2 \rightarrow \mathcal{U}_2$ such that 
    \begin{equation}
   T(\mathbf{g})=\Upsilon\Big(-\Pi_{21}\big(\Lambda(\Pi_{12}\mathbf{g})\big)\Big).
   \label{eq:fixed_point_operator}
   \end{equation}
   Importantly, equation \eqref{eq:fixed_point_relax} can be rewritten as
   \begin{equation}  
   \mathbf{g}^{(k+1)}=T^{\rm{r}}_{\rho}(\mathbf{g}^{(k)})
    \label{eq:fp_general}
   \end{equation}
   for a suitable operator $T^{\rm{r}}_{\rho}: \mathcal{U}_2 \rightarrow \mathcal{U}_2$.  This generalizes the standard method into a classical one-step fixed-point iteration for any suitable choice of $\rho$.
    \end{subequations}
  
   Thanks to the reformulation in \eqref{eq:fixed_point_relax}, we can restate Algorithm \ref{alg:DN_relax} in a more concise form of a fixed-point iteration algorithm, as done in Algorithm \ref{alg:DN_relax_fixed_point}.  
   As observed in \cite{discacciati2024nonintrusive}, the formulation in \eqref{eq:fixed_point_relax} does not involve the values of the solutions inside the sub-domains $\Omega_i$, as all the operators are defined on the interface $\Gamma$. Hence, it is possible to apply the Schwarz alternating method using only interface data, and recover the full sub-domain solutions $\mathbf{u}_1$ and $\mathbf{u}_2$ at the end of the iteration procedure (Algorithm \ref{alg:DN_relax_fixed_point}).  The absolute and relative errors appearing in line \ref{line:conv_criterion} of Algorithm \ref{alg:DN_relax_fixed_point} can be computed using formula \eqref{eq:error_def} with $\gamma_1(\mathbf{u}_1^{(k)})=\Pi_{12}\mathbf{g}^{(k)}$ and $\gamma_2(\mathbf{u}_2^{(k)})=T(\mathbf{g}^{(k)})$.
   
   %However, the application of Schwarz to mechanical systems and possibly to multi-physics systems motivates the reconstruction of the solutions in these sub-domains, which can be recovered after the fixed-point problem has converged. Furthermore, we believe that the final comparison of the local solutions (found by Schwarz algorithm) with the global one represents an important tool to assess the numerical validity of the method.  Thus, we write $\mathbf{u}_1$ and $\mathbf{u}_2$ as outputs of Algorithm \ref{alg:DN_relax_fixed_point} at line \ref{alg_line:fp_output} once $\mathbf{g}$ is computed at convergence.\\
 \begin{algorithm}[h]
\begin{algorithmic}[1] % The number tells where line numbering should start
\Require Initial guess \( \mathbf{g}^{(1)} \), \( \texttt{maxit} \), $\epsilon_{\text{rel}} $, $\epsilon_{\text{abs}}$,  \( \rho \), \( k = 0 \)
%\Ensure Outputs: \( \mathbf{u}_1, \mathbf{u}_2 \)
\Repeat
 \State Increment \( k \leftarrow k + 1 \)
 \State Update $\mathbf{g}^{(k+1)}=(1-\rho) \mathbf{g}^{(k)}+\rho T(\mathbf{g}^{(k)})$ \label{line:fp-iter}
  \Until{$e^{(k)}_{\rm{abs}}>\epsilon_{\rm{abs}}$ \textbf{and} $e^{(k)}_{\rm{rel}}>\epsilon_{\rm{rel}}$ \textbf{and} \( k < \texttt{maxit} \)} \label{alg_line:criterion}
\State \textbf{Output:} $\mathbf{u}_1^{(k)}, \mathbf{u}_2^{(k)}$ \label{alg_line:fp_output}
\end{algorithmic}
  \caption{Dirichlet-Neumann scheme (with relaxation) as a fixed-point iteration.}
  \label{alg:DN_relax_fixed_point}
  \end{algorithm}
% The errors critieria in Algorithm \ref{alg:DN_relax_fixed_point} computed at each Schwarz iteration $k$ are based on the formula in \eqref{eq:error_def}: to compute that, it is sufficient to replace $\gamma_1(\mathbf{u}_1^{(k)})=\Pi_{12}\mathbf{g}^{(k)}$ and $\gamma_2(\mathbf{u}_2^{(k)})=T(\mathbf{g}^{(k)})$.\\
%The expression in  \eqref{eq:fixed_point_relax} has been found for the classical relaxation of the datum $\mathbf{g}$. One may wonder what other reformulations can be obtained with different relaxation/acceleration formulas.  In the next sub-section, we comment about the reformulation of Schwarz in terms of the DtN and NdT maps in the case of different types of updates.

\begin{remark} \label{remark:additive_fp}
    In the case of the additive variant of the Schwarz alternating method (see Remark \ref{remark:additive}), achieving the one-step fixed-point iteration \eqref{eq:fp_general} is more nuanced.  % for the additive variant (see Remark \ref{remark:additive}). 
    Because the local sub-problems are solved simultaneously, the transmission conditions evaluated on $\Gamma_{ij}$ depend on the solutions from the previous iteration step, in particular the Neumann transmission condition would be dependent on $\Psi_1(\mathbf{u}^{(k-1)})$ on $\Gamma_{21}$. Thus, we would obtain, in place of \eqref{eq:fixed_point_relax}, a two-step iterative scheme of the form
    \begin{equation}
    \mathbf{g}^{(k+1)}=(1-\rho) \mathbf{g}^{(k)}+\rho T(\mathbf{g}^{(k-1)}).
    \label{eq:two_step_update}
    \end{equation}
    Although \eqref{eq:two_step_update} is not a fixed-point iteration in the classical sense, it can be rigorously recast as a generalized one-step fixed-point iteration by augmenting the state space into $\bar{\mathbf{g}}=[\mathbf{g}^{(k)}, \mathbf{g}^{(k-1)}]^T$. The additive iteration can then be written as $\bar{\mathbf{g}}^{(k+1)} = \bar{T}(\bar{\mathbf{g}}^{(k)})$, where the augmented operator $\bar{T}$ is defined as 
    \begin{equation} 
    \bar{T} \begin{pmatrix} \mathbf{v}_1, \ \mathbf{v}_2 \end{pmatrix} := \begin{pmatrix} (1-\rho) \mathbf{v}_1 + \rho T(\mathbf{v}_2), \quad \mathbf{v}_1\end{pmatrix}.
    \label{eq:augmented_operator} 
    \end{equation} 
    It is straightforward to verify that any fixed-point $[\mathbf{v}_1^\star, \mathbf{v}_2^\star]^T$ of this extended operator $\bar{T}$ necessarily satisfies $\mathbf{v}_1^\star = \mathbf{v}_2^\star$, which immediately implies $\mathbf{v}_1^\star = T(\mathbf{v}_1^\star)$.
\end{remark}

\subsection{Schwarz acceleration techniques} 
\label{sec:acceleration}
\subsubsection{Aitken acceleration} \label{sec:aitken}
The classical relaxation introduced in Algorithm \ref{alg:DN_relax}, line \ref{line:relax} 
assumes that $\rho$ is a fixed pre-specified parameter.  The ``optimal'' value of $\rho$
is generally unknown \textit{a priori}, and poor choices of this parameter can lead to convergence issues, 
as discussed in \cite{deparis2004numerical} and \cite{rodriguez2025transmission}.
%does not provide an explicit way to compute $\rho$ and the convergence may suffer from non-optimal choices of this algorithmic parameter.
%The choice of this parameter is crucial for the convergence of the method: in \cite{deparis2006domain}, it has been discussed that if $\rho$ is constant, it must be a function of the domain; otherwise, the algorithm does not converge; also, for small values, the convergence is slowed. 
As demonstrated in \cite{deparis2006domain} for the specific case of a two sub-domain fluid-structure
interaction problem, it is possible to derive a relaxation scheme in which $\rho$ is updated dynamically throughout the Schwarz iteration process by applying Aitken acceleration to the fixed-point version of the Schwarz alternating method (Algorithm \ref{alg:DN_relax_fixed_point}).  
%To remedy this issue, an adaptive choice of the algorithmic parameter $\rho$ at each Schwarz iteration
%represents a valid alternative for the update of the interface datum, in order to ensure or accelerate the convergence.
%Aitken acceleration
%A dynamic relaxation technique based on the generalization to the vector case of Aitken acceleration has been proposed (see \cite{deparis2006domain}) for $2$ domains fluid-structure interaction applications. 
%This technique leads to an automatic choice of the relaxation parameter $\rho^{(k)}$, which is wisely choosen at each Schwarz iteration $k=1, \ldots, \texttt{maxit}$. 
The scheme is defined by first introducing the interface jump $\mathcal{E}:\mathcal{U}_2 \times \mathcal{U}_2 \rightarrow \mathbb{R}^{N^{\rm c} d}$ (where $N^{\rm{c}}$ is the number of interface nodes of $\Gamma_{21}$)
\begin{equation}
\mathcal{E}\left(\gamma_1^1(\mathbf{u}_1), \gamma_2(\mathbf{u}_2)\right)=\gamma_2(\mathbf{u}_2)-\gamma_1^2(\mathbf{u}_1),
\label{eq:int_jump}
\end{equation}
together with the following variables: 
%We can denote as $\mathbf{d}^{(k)}$ and $\boldsymbol{\delta}^{(k)}$ the following expressions:
\[ \mathbf{d}^{(k)}=\left( \gamma^2_1(\mathbf{u}_{1}^{(k)}) - \gamma^2_1(\mathbf{u}_{1}^{(k-1)}) \right),\;
\boldsymbol{\delta}^{(k)}=\mathcal{E}\left(\gamma_1^1(\mathbf{u}^{(k)}_1), \gamma_2(\mathbf{u}^{(k)}_2)\right)-\mathcal{E}\left(\gamma_1^1(\mathbf{u}^{(k-1)}_1), \gamma_2(\mathbf{u}^{(k-1)}_2)\right).\]
Now, to apply Aitken acceleration within the Dirichlet-Neumann Schwarz scheme, line \ref{line:relax} in Algorithm \ref{alg:DN_relax} is replaced by 
%The classical relaxation update expressed at line \ref{line:relax} would be replaced by the following:
 \begin{equation}
 \begin{aligned}
    &\mathbf{g}^{(k+1)} = \mathbf{g}^{(k)} + \rho^{(k)} \mathcal{E}\left(\gamma_1^1(\mathbf{u}^{(k)}_1), \gamma_2(\mathbf{u}^{(k)}_2)\right), \quad \text{where}
   &\rho^{(k)} = -\frac{
   \boldsymbol{\delta}^{(k)} \cdot \mathbf{d}^{(k)}
    }{
    || \boldsymbol{\delta}^{(k)} ||_2^2
    }
    \end{aligned}
    \label{eq:aitken}
    \end{equation}
    Here, $\rho^{(k)}$ denotes the value of the relaxation parameter at the $k^{th}$ Schwarz iteration, with $\rho^{(1)}$ given as an input.
    The value of $\rho^{(k)}$  in \eqref{eq:aitken} minimizes   
 %In \cite{deparis2004numerical}, the authors refer to the choice of $\rho^{(k)}$ according to the formula in \eqref{eq:aitken} as the minimizer of 
\begin{equation}
\mathcal{J}\big(\rho\big)=\Big\|\left(\gamma_1^2(\mathbf{u}_1^{(k)})-\gamma_1^2(\mathbf{u}_1^{(k-1)})\right)+\rho \left(
\gamma_2(\mathbf{u}_2^{(k)})-\gamma_1^2(\mathbf{u}_1^{(k)})-\gamma_2(\mathbf{u}_2^{(k-1)})+\gamma_1^2(\mathbf{u}_1^{(k-1)})
\right)\Big\|^2_2.
\label{eq:aitken_of}
\end{equation}
at each Schwarz iteration $k =1, ...,\texttt{maxit}$ (see  \ref{app:aitken}).  

Our Aitken-accelerated Schwarz scheme is summarized in Algorithm \ref{alg:aitken}.  The reader can observe that we have introduced a parameter $N_0 \in \mathbb{N}$ such that, for $1 < k < N_0$, $\rho^{(k)} = \rho^{(1)}$.  As far as we know, the introduction of $N_0$ into the Aitken acceleration approach has not been considered in past literature.  %\irinanote{Also, was $N_0$ also introduced in \cite{deparis2006domain}?  If so, I will note that.}  In our numerical experiments (Section \ref{sec:num_results}), we explore the performance and robustness of Aitken-accelerated Schwarz to both $\rho^{(1)}$ and $N_0$. 
%, we mathematically justify the employment of this expression and the derivation of its minimizer $\rho$ as in \eqref{eq:aitken}.

%In the numerical tests, we investigate the choice of both the relaxation and acceleration parameters. To this scope, it is convenient to introduce a further hyper-parameter in Aitken acceleration: we denote by $N_0$ the number of iterations such that $\rho^{(k)}=\rho^{(jN_0)}$ where $j=\Big\lfloor \frac{k}{N_0}\Big\rfloor$. As we discuss in the numerical results section, the choice of $N_0$ plays a role in the multi-domain setting: we numerically observed that for $k<N_{\rm{dd}}$, classical relaxations have to be used before switching to Aitken formula, since Aitken acceleration parameters values badly influence Schwarz convergence at these early stage iterations,  while in the two-domain tests, in order to limit the number of tuning knobs in our algorithm we can just set the minimum feasible value which is $N_0=2$. 

\begin{remark}
    It is possible to rewrite the Aitken formula \eqref{eq:aitken} using the previously-defined DtN and NtD maps, as done for the classical relaxation formula in \eqref{eq:fixed_point_relax}:
  %Doing so gives
    \begin{equation}
    \begin{aligned}
    &\mathbf{g}^{(k+1)}=\mathbf{g}^{(k)}+\rho^{(k)} \, \Big[\,T(\mathbf{g}^{(k)})\, - \, \mathbf{g}^{(k)} \Big]  \quad \text{with}
    &\rho^{(k)}=\rho^{(k)}\Big(\mathbf{g}^{(k)}, \mathbf{g}^{(k-1)}\Big).
    \end{aligned}
        \label{eq:fake_fixed_point_aitken}
    \end{equation}
    In \eqref{eq:fake_fixed_point_aitken}, 
    %where, in the second line, 
    we have explicitly expressed the dependence of the parameter $\rho^{(k)}$ on the current and previous interface solutions.
Unlike \eqref{eq:fixed_point_relax}, the iteration in \eqref{eq:fake_fixed_point_aitken}  is not a single fixed-point equation in the classical sense, as  the dynamic parameter $\rho^{(k)}$ is a function of the interface data at iteration $k-1$. A similar discussion to Remark \ref{remark:additive_fp} can be replicated in this case.
     \end{remark}
\begin{algorithm}[h]
\begin{algorithmic}[1] 
\Require Initial guess \( \mathbf{g}^{(1)} \), \( \texttt{maxit} \),$\epsilon_{\text{rel}} $, $\epsilon_{\text{abs}}$, \( \rho^{(1)} \), \( k = 0 \), $N_0$
%\Ensure Outputs: \( \mathbf{u}_1, \mathbf{u}_2 \)
\Repeat
 \State Increment \( k \leftarrow k + 1 \)
%\State Increment \( k \leftarrow k + 1 \)
\If{$k < N_0$}
    \State $\rho^{(k)} = \rho^{(1)}$
\Else
    \State $\rho^{(k)}$ dynamically updated from \eqref{eq:aitken}
\EndIf
\State Update $\mathbf{g}^{(k+1)}=(1-\rho^{(k)}) \mathbf{g}^{(k)}+\rho^{(k)} \left[T(\mathbf{g}^{(k)})-\mathbf{g}^{(k)}\right]$ 
  \Until{$e^{(k)}_{\rm{abs}}>\epsilon_{\rm{abs}}$ \textbf{and} $e^{(k)}_{\rm{rel}}>\epsilon_{\rm{rel}}$ \textbf{and} \( k < \texttt{maxit} \)} \label{alg_line:general_criterion}
\State \textbf{Output:} $\mathbf{u}_1^{(k)}, \mathbf{u}_2^{(k)}$ \label{alg_line:general_fp_output}
\end{algorithmic}
  \caption{Dirichlet-Neumann scheme with Aitken acceleration }
  \label{alg:aitken}
  \end{algorithm}

     \begin{remark} \label{remark:aitken_bounds}
          The Aitken formula for $\rho^{(k)}$ \eqref{eq:aitken} does not guarantee explicit upper or lower bounds for the relaxation parameter. To ensure consistency with the classical relaxation framework (Algorithm \ref{alg:DN_relax}) and to avoid nonphysical values, we enforce in our numerical implementation of Aitken-accelerated Schwarz the following safeguards: (i) we set $\rho^{(k)}=1$ whenever Aitken’s formula \eqref{eq:aitken} gives $\rho^{(k)}>1$, and (ii) we revert to $\rho^{(k)}=\rho^{(1)}$ whenever a negative value is obtained.  %\irinanote{Have others done this, or is it new?}
     \end{remark}

\subsubsection{Analysis of the Aitken-accelerated Schwarz scheme} 
\label{sec:aitken_conv_rate}

Having presented the general Aitken-accelerated Schwarz scheme, we now perform some theoretical analysis to uncover its convergence rate for the specific case of $d=1$. We recall that a $p$-order method converges to its fixed-point $g^{\star}$ if \[ \exists c>0: 
\frac{|g^{(k+1)}-g^{\star}|}{|g^{(k)}-g^{\star}|^{p}} \leq c, \; \forall k \geq k_0
\]
where $k_0$ is a suitable integer and, in the case $p=1$, we need to strictly require $0<c<1$.
\begin{theorem}\label{thm:aitken_conv}
Suppose $d=1$. Let $g^{(k+1)}=T(g^{(k)})$ be the fixed-point iteration of an unrelaxed Schwarz scheme (line \ref{line:fp-iter} in Algorithm \ref{alg:DN_relax_fixed_point} with $\rho =1$). Assuming that the unrelaxed Schwarz method converges linearly ($p=1$), the Aitken-accelerated Schwarz algorithm (Algorithm \ref{alg:aitken}) converges to $g^{\star}$ quadratically ($p=2$).  
\end{theorem}
\begin{proof}
Let $e_{k}=g^{(k)}-g^{\star}$  denote the convergence error at iteration $k$.
We have that $\lim_{k\rightarrow \infty} \frac{|\mathbf{e}_{k+1}|}{|\mathbf{e}_{k}|}=|\lambda|$ with the convergence factor  $0<|\lambda|<1$ in order  for $g^{(k)}$ to converge to $g^{\star}$. 
Now, the error can be written as
\begin{subequations}
\begin{align}
\label{eq:linear_conv_fp}
e_{k+1}=\lambda e_{k}+O(|e_{k}|^2),
\end{align}
assuming that the fixed-point operator $T$ is sufficiently smooth (specifically, at least locally of class $C^2$ near $g^{\star}$).
In the scalar case, the quantities appearing in Aitken formula simplify to:  
\begin{align}
d^{(k)}&=g^{(k)}-g^{(k-1)}=e_k-e_{k-1}=(\lambda-1)e_{k-1}+O(|e_{k-1}|^2), \label{eq:d}\\
\delta^{(k)}&=\mathcal{E}^{(k)}-\mathcal{E}^{(k-1)}=ad^{(k)}=a(\lambda-1)e_{k-1}+O(|e_{k-1}|^2),
\end{align}
where we assume that the interface residual is linear in the interface unknown, i.e., $\mathcal{E}^{(k)}=a e_k$ with $a\in \mathbb{R}_0$.  
It follows that Aitken parameter $\rho^{(k)}$ reduces to 
\begin{equation} \label{eq:rhok}
\rho^{(k)}=-\frac{1}{a}+O(|e_{k-1}|).
\end{equation}
From the interface update formula with Aitken acceleration (c.f., \eqref{eq:aitken}), $
%\[
g^{(k+1)}=g^{(k)}+\rho^{(k)}\mathcal{E}^{(k)}$.
%\]
Now, using \eqref{eq:linear_conv_fp} and \eqref{eq:d}, we have that % at iteration $k$ and \eqref{eq:d}, 
\begin{align}
    e_{k+1}&=e_k+\rho^{(k)}ae_k\\
    &= \lambda e_{k-1}+O(|e_{k-1}|^2)+\rho^{(k)}[a\lambda e_{k-1}+O(|e_{k-1}|^2)] \label{eq:aitken_p2_0}\\
    &=O(|e_{k-1}|^2), \label{eq:aitken_p2}
\end{align} 
where the last equality  is derived by replacing in it the expression on $\rho^{(k)}$ in \eqref{eq:rhok}. We have thus demonstrated that the linear term in the error drops out, so that the Aitken error scales only with the quadratic term $|e_{k-1}|^2$. The Aitken convergence rate is thus equal to two, since 
\[
\frac{|e_{k+1}|}{|e_{k}|^2}\leq C \frac{|e_{k-1}|^2}{|e_k|^2}=C\Big(\frac{e_{k-1}}{e_k}\Big)^2\leq C \left(\frac{1}{|\lambda|-\epsilon}\right)^2<\infty
\]
where $\epsilon = \mathcal{O}(|e_{k-1}|^2)$.  
%the comment by irina was not correct, i corrected the value of epsilon, even though I do not think I need to explicitly expression for epsilon...
The first inequality for $C>0$ is derived from the result in \eqref{eq:aitken_p2} and, in the last inequality, we used the assumption that unrelaxed Schwarz has a linear convergence rate, i.e. 
\begin{align*}
    \frac{|e_k|}{|e_{k-1}|}=|\lambda|+O(|e_{k-1}|^2),
    &\implies \quad \text{for $k$ large enough, $\exists \epsilon >0$:} \quad |\lambda|-\epsilon \leq \frac{|e_k|}{|e_{k-1}|}\leq |\lambda|+\epsilon \\
     &\implies \frac{1}{|\lambda|+\epsilon} \leq \frac{|e_{k-1}|}{|e_{k}|}\leq \frac{1}{|\lambda|-\epsilon}.
\end{align*}
\end{subequations}
%\hspace{5in} $\square$
\end{proof}

\begin{remark}
Convergence of the Aitken-accelerated sequence actually just requires  the \textit{local} linear convergence of the base fixed-point iteration, i.e., in the proof of Theorem \ref{thm:aitken_conv}, we actually just needed $0 < |\lambda| <1 $  in a neighborhood of $g^{\star}$. Indeed, once the fixed-point iterates enter the linear regime near $g^{\star}$, Aitken acceleration guarantees quadratic convergence to $g^{\star}$, irrespective of any weaker global convergence properties of the fixed-point map.
%It does not claim that Aitken can produce convergence from arbitrary starting points or from a completely non-convergent global method.
% THIS IS FALSE! According to \cite{quarteroni2006numerical}, Aitken’s method can be linearly convergent ($p=1$) %and may converge 
%even when the underlying fixed-point method is not contractive; however, this result is not proven in this reference. 
\end{remark}
\begin{remark}
When $d>1$, there is no general result guaranteeing quadratic convergence of the Aitken-accelerated vector sequence without additional structure (a dominant eigenvector or a 1D invariant subspace). 
The acceleration can at best damp or cancel the dominant 1D component, but the full vector error may not be reduced quadratically in general. 
As we describe in Section \ref{sec:num_results} for a  two-dimensional (2D) study case, a super-linear/quadratic convergence can be achieved by Aikten-accelerated Schwarz for some values of $N_{\text{dd}}$.
\end{remark}

\begin{comment} NOT so important
\begin{remark}
We also note that, in the literature, Aitken acceleration is frequently linked with Steffensen's method (see \cite{johnson1968steffensen}). Steffensen's method guarantees quadratic convergence but, unlike Aitken’s approach, it requires two function evaluations per iteration rather than one.  
\end{remark}
\end{comment}

\subsubsection{Anderson acceleration}   \label{sec:anderson}   
% Anderson acceleration

An alternate method that has been shown to speed up the convergence of fixed-point iterations is Anderson acceleration \cite{walker2011anderson}.  Here, we adapt this approach to the specific case of the non-overlapping Schwarz alternating method, which can be written as a fixed-point iteration, as shown above.
To the best of our knowledge, our theoretical and numerical assessment of Anderson-accelerated non-overlapping Schwarz is one of the novel contributions of this work.  

The basic non-overlapping Schwarz Dirichlet-Neumann scheme with Anderson acceleration is given in Algorithm \ref{alg:DN_anderson}. The reader can observe that a new algorithmic parameter is introduced, $m_{\rm{and}} \in \mathbb{N}_0$, which represents the maximum history of the previous steps that we want to include in the update of $\mathbf{g}^{(k)}$.  %If $m_k$ denotes the number of history steps utilized in Schwarz iteration $k$, t
The basic version of the algorithm sets   $m_k=\min(k, m_{\rm{and}})$ (Algorithm \ref{alg:DN_anderson}, line 4).   We introduce in line \ref{alg_line:fixed_point_res} of Algorithm \ref{alg:DN_anderson} the history residual matrix $\mathbf{F}_k \in \mathbb{N}^{N^{\rm{c}} d\times (m_{k}+1)}$, which stores in its columns the last $m_{k}+1$ fixed-point residuals $\mathbf{f}_i=T(\mathbf{g}^{(i)})-\mathbf{g}^{(i)}$ with $i=k-m_k, \ldots, k$, and defines the history for Schwarz iteration $k$.  The Dirichlet and Neumann problems in lines \ref{alg_line:dirichlet} and \ref{alg_line:neumann} of the classical Schwarz algorithm,  Algorithm \ref{alg:DN_relax}, are encoded in step \ref{alg_line:fixed_point_res} of Algorithm \ref{alg:DN_anderson} in the evaluation of the operator $T(\mathbf{g})$.  The reader can observe that Algorithm \ref{alg:DN_anderson} includes a pre-specified relaxation parameter $\rho \in (0,1]$, similar to the classical relaxed Dirichlet-Neumann scheme (Algorithm \ref{alg:DN_relax}).  Some analysis of relaxed Anderson acceleration can be found in \cite{toth2015convergence}.  While it is common to simplify the Anderson acceleration scheme by setting $\rho =1$, herein we choose explore the impact of this parameter on Schwarz convergence and robustness (see Section \ref{sec:num_results}).  Our results suggest that Anderson-accelerated Schwarz is far less sensitive to $\rho$ than Schwarz with classical relaxation, often exhibiting convergence even in the unrelaxed ($\rho =1$) case.    
%which form the history for the Schwarz iteration $k$.

%Anderson acceleration can be found in the literature for the acceleration of the convergence rate of fixed-point iterations \cite{walker2011anderson}. Its applicability in the context of Schwarz algorithm and the comparison with Aitken acceleration both from a theoretical and numerical point of view is also a contribution of the present work.
%The procedure requires the introduction of an algorithmic parameter $m_{\rm{and}} \in \mathbb{N}_0$ representing the maximum history of the previous steps we want to include in the update of $\mathbf{g}^{(k)}$. 
%At each Schwarz iteration we set $m_k=\min(k, m_{\rm{and}})$ to select the previous updates on $\mathbf{g}$ and we use a history residuals matrix $\mathbf{F}_k \in \mathbb{N}^{N^{\rm{c}}\times m_{k}+1}$
%storing in its columns the last $m_{k}+1$ fixed-point residuals $\mathbf{f}_i=T(\mathbf{g}^{(i)})-\mathbf{g}^{(i)}$ with $i=k-m_k, \ldots, k$ which form the history for the Schwarz iteration $k$.
%In Algorithm \ref{alg:DN_anderson} this procedure is summarized. \\
\begin{algorithm}[h]
\begin{algorithmic}[1] % The number tells where line numbering should start
\Require Initial guess \( \mathbf{g}^{(0)} \), \( \texttt{maxit} \), $\epsilon_{\text{rel}} $, $\epsilon_{\text{abs}}$,  \( k = 0 \), $m_{\rm{and}}$, $\rho$
%\Ensure Outputs: \( \mathbf{u}_1, \mathbf{u}_2 \)
\State Compute $\mathbf{g}^{(1)}=T(\mathbf{g}^{(0)})$
\Repeat
 \State Increment \( k \leftarrow k + 1 \)
 \State Define $m_k=\min(k, m_{\rm{and}})$ 

\State Set the residual matrix $\mathbf{F}_{k}=[\mathbf{f}_{k-m_k}, \ldots, \mathbf{f}_{k}]$ where $\mathbf{f}_i=T(\mathbf{g}^{(i)})-\mathbf{g}^{(i)}$. \label{alg_line:fixed_point_res}
    \State Determine $\boldsymbol{\alpha}^{(k)} = [\alpha^{(k)}_0, \ldots, \alpha^{(k)}_{m_k}] ^T$ that solve \begin{equation}
        \min_{\boldsymbol{\alpha}=[\alpha_0, \ldots, \alpha_{m_k}] ^T} \|\mathbf{F}_k \boldsymbol{\alpha}\|_2 \quad\text{s.t.} \quad \sum_{i=0}^{m_k} \alpha_i=1 \label{eq:anderson_constrained}
    \end{equation}
    \State Set \begin{equation}
     \mathbf{g}^{(k+1)}=(1-\rho)\sum_{j=0}^{m_k}\alpha_j^{(k)}\mathbf{g}^{(k-m_k+j)}+\rho \sum_{j=0}^{m_k} \alpha_j^{(k)}\, T\big(\mathbf{g}^{(k-m_k+j)}\big)
       \label{eq:anderson_update}
    \end{equation}
   \Until{$e^{(k)}_{\rm{abs}}>\epsilon_{\rm{abs}}$ \textbf{and} $e^{(k)}_{\rm{rel}}>\epsilon_{\rm{rel}}$ \textbf{and} \( k < \texttt{maxit} \)} \label{alg_line:criterion}
\State \textbf{Output:} $\mathbf{u}_1^{(k)}, \mathbf{u}_2^{(k)}$
\end{algorithmic}
\caption{Dirichlet--Neumann scheme with Anderson acceleration (and no memory adaptation) as a fixed-point iteration.}
\label{alg:DN_anderson}
\end{algorithm}

%The Dirichlet and Neumann problems at steps \ref{alg_line:dirichlet} and \ref{alg_line:neumann} of the classical algorithm,  Algorithm \ref{alg:DN_relax}, are encoded in step \ref{alg_line:fixed_point_res} in Algorithm \ref{alg:DN_anderson} in the evaluation of the operator $T(\mathbf{g})$.

%As observed in \cite{walker2011anderson}, any norm can be used in the constrained optimization problem in \eqref{eq:anderson_constrained}.  We opt for the $\ell^2$ norm. Also, the update formula in \eqref{eq:anderson_update} can be written in a relaxed version as follows:
%\begin{equation}
 %\mathbf{g}^{(k+1)}=(1-\rho)\sum_{j=0}^{m_k}\alpha_j^{(k)}\mathbf{g}^{(k-m_k+j)}+\rho \sum_{j=0}^{m_k} \alpha_j^{(k)}\, T\big(\mathbf{g}^{(k-m_k+j)}\big)
% \label{eq:anderson_relaxed}
%\end{equation}
%by employing a relaxation hyper-parameter $\rho \in [0,1]$. The analysis of the method for a general fixed-point problem in \cite{toth2015convergence} is restricted to the case $\rho^{(k)}=\rho^{(1)}=\rho \quad \forall k\geq 1$; although, an adaptive choice of $\rho^{(k)}$ in the context of Anderson acceleration is expected to improve Schwarz's convergence properties; this aspect is an ongoing work.

\subsubsection{Comparisons to and connections with Aitken acceleration} \label{sec:aitken_anderson_compare}

Having introduced the Anderson-accelerated Schwarz scheme, we now make some comparisons and connections between this scheme and Aitken-accelerated Schwarz.  

%\begin{remark}
First, when $m_{\text{and}} = 1$ and $\rho = 1$ (no relaxation), equation \eqref{eq:anderson_update} takes the form 
    %    As done for classical relaxation and Aitken acceleration, we rewrite Anderson formula \eqref{eq:anderson_update} by employing the DtN and NtD maps as done for the classical relaxation formula \ref{line:relax} in Algorithm \ref{alg:DN_relax}. One obtains, for the case $m=1$ (and thus $m_k=1 \quad \forall k\geq1$):
    \begin{equation}
    \mathbf{g}^{(k+1)}=\alpha_0^{(k)} T(\mathbf{g}^{(k-1)}) \, + \, \alpha_1^{(k)}T(\mathbf{g}^{(k)}),
        \label{eq:fake_fixed_point_anderson}
    \end{equation}
    which represents an update based on the two previous iterations, as in Aitken acceleration.
    %as in Aitken for the particular case of $m=1$. In this computation, we considered the no-relaxation case $\rho=1 \quad \forall k\geq 1$.  
    For $m_{\text{and}}>1$, the formula in \eqref{eq:fake_fixed_point_anderson} is enriched with $m_{k}+1$ terms of the form $g^{(k-(m_k+1))}, \ldots, g^{(k)}$. 
    %\end{remark}

Returning to the more general case of Anderson acceleration with $m_{\text{and}} > 1$, we show that Anderson and Aitken acceleration can be interpreted as multi-secant and single-secant quasi-Newton methods, respectively.  These interpretations have implications for the expected  convergence of the Schwarz alternating method accelerated with these schemes.

%\begin{remark}
%\label{remark:secant_eq}
  As shown in \cite{walker2011anderson}, the original constrained Anderson least-squares problem \eqref{eq:anderson_constrained} can be recast as an unconstrained residual minimization involving residual differences. Using the constraint in \eqref{eq:anderson_constrained}, the last Anderson coefficient takes the form $\alpha_{m_k}=1-\sum_{i=0}^{m_k-1}\alpha_i$.  Substituting this value into the objective function in \eqref{eq:anderson_constrained} gives rise to the following equivalent form
    \begin{equation}
        \min_{\boldsymbol{\gamma}=[\gamma_0, \ldots, \gamma_{m_k-1}]^T} \|\mathbf{f}_k-\mathcal{F}^{(k)}\boldsymbol{\gamma}\|_2^2, 
        \label{eq:anderson_unconstrained}
    \end{equation}
for coefficients vector $\boldsymbol{\gamma} \in \mathbb{R}^{m_k}$  with $\mathcal{F}^{(k)}=\left(\Delta \mathbf{f}_{k-m_k}, \ldots, \Delta \mathbf{f}_{k-1}\right)$ and $\Delta \mathbf{f}_k:=\mathbf{f}_k-\mathbf{f}_{k-1}$.
If we define 
%We can define 
$\mathbf{U}^{(k)}=\left(\Delta \mathbf{g}^{(k-m_k)}, \ldots, \Delta \mathbf{g}^{(k-1)}\right)$ with $\Delta \mathbf{g}^{(k)}=\mathbf{g}^{(k+1)}-\mathbf{g}^{(k)}$, 
%taking into account Anderson $m_k$ history of interface solutions. Now, 
the minimization problem \eqref{eq:anderson_unconstrained} can be interpreted as %makes explicit the re-intepretation of Anderson acceleration as 
a multi-secant quasi-Newton method with associated multi-secant equation 
%. The associated \emph{multi-secant} equation is the following:
    \begin{equation} \label{eq:multi-secant}
        \mathbf{U}^{(k)}=\mathbf{G}^{(k)}\mathcal{F}^{(k)}. 
    \end{equation}
    Here, $\mathbf{G}^{(k)}$ is an approximation of the inverse of the Jacobian of $\mathbf{f}_k=T(\mathbf{g}^{(k)})-\mathbf{g}^{(k)}$. 
In contrast, it is straightforward to see that Aitken acceleration is associated with a single-secant quasi-Newton update. %In the vector case, 
As shown in Section \ref{sec:aitken}, the Aitken acceleration parameter \(\rho^{(k)}\) is obtained by minimizing the functional  \eqref{eq:aitken_of} involving only the most recent solution pair of solutions.  This least-squares minimization problem corresponds to a secant equation of the general form $\Delta \mathbf{g}^{(k)}=\mathbf{G}^{(k)}\Delta \mathbf{f}_k$, which is written explicitly in \eqref{eq:aitken}.  
The reader can observe from this equation that the inverse Jacobian $\mathbf{G}^{(k)}$ is being approximated by a scalar multiple of the identity. %We refer to Appendix \ref{app:aitken} for the derivation of Aitken formula \eqref{eq:aitken}.\\
It is well-known that a multi-secant method approximates the Jacobian inverse more accurately than the single secant technique; hence, Anderson-accelerated Schwarz is expected to be more robust, especially when applied to large-scale multi-dimensional problems decomposed into large numbers of subdomains.  This claim is explored numerically in Section \ref{sec:num_results}.  

%, above all for large scale or multi-domains/multi-dimensional problems, at 
Finally, we compare Aitken- and Anderson-accelerated Schwarz from the perspective of implementation and usability.  It is clear that, for $m_{\text{and}} > 1$, Anderson-accelerated
Schwarz has a larger memory footprint and gives rise to a more involved implementation.  While we would expect for Anderson-accelerated Schwarz to have a larger CPU time per iteration due to the need to solve the minimization problem \eqref{eq:anderson_constrained} at each Schwarz iteration $k$, our numerical results demonstrate that this is actually not the case in practice, at least in our implementation (see Section \ref{sec:num_results}).  
%From a usability perspective, Aitken-accelerated Schwarz is preferred as it has fewer tuning knobs

%The price to pay is a larger memory footprint relative to Aitken-accelerated Schwarz. 
%Although, from a practical implementation and usability perspective, we have to notice that Aitken acceleration method offers the advantage of  basically no required tuning parameters; on the contrary, Anderson requires the tuning of $(\rho, m_{\rm{and}})$ and it also requires the repeated numerical solution of a constrained optimization problem \eqref{eq:anderson_constrained}.\\
%We remind to numerical section \ref{sec:num_results} \ref{par:multi_domain} for the numerical comparison of these two accelerations, in particular to \ref{par:multi_domain} for multi-domains applications.
%\end{remark}

\subsubsection{Analysis of the Anderson-accelerated Schwarz scheme} \label{sec:anderson_anal}

In the following theorem, we analyze the Anderson optimization coefficients $\alpha_j$ in terms of the spectral properties of the fixed-point operator $T$ for the specific case of $d=1$. 

\begin{theorem}
\label{theo:anderson_coeff}
The Anderson coefficients  $\boldsymbol{\alpha} \in \mathbb{R}^{m_k+1}$ obtained by solving the constrained minimization problem \eqref{eq:anderson_constrained} in the scalar limit follow a geometric distribution. Specifically, they take the form
    \begin{equation}
        \alpha_j=\frac{|\ell_{\rm{max}}|^j}{\sum_{i=0}^{m_k}|\ell_{\rm{max}}|^i}  \; \text{for} \; j=0, \ldots, m_k,
    \end{equation}
    where $|\ell_{\rm{max}}|=\max \{ |\lambda|: \lambda \in \sigma(T)\}$, with $\sigma(T)$ denoting the spectrum of  the fixed-point operator $T$, assumed to be linear. %represents the magnitude of its dominant eigenvalue.
    If $|\ell_{\rm{max}}|=1$, the coefficients reduce to the uniform weights $\alpha_j=\frac{1}{m_k+1}$.  
\end{theorem}
\begin{proof}
   \begin{subequations}
     Consider the fixed-point residuals $f_i=T(g^{(i)})-g^{(i)}=T(g^{(i)})-g^{\star}+g^{\star}-g^{(i)}=T(e_i)-e_i$. It follows that $f_k=T(e_k)-e_k$ and  $f_{k+1}=T(e_{k+1})-e_{k+1}$.
     By the hypothesis of asymptotic regime, we have that $T(e_{k+1})\approx \ell_{\rm{max}}T(e_k)$. Indeed, 
     since $g^{(k+1)}=T(g^{(k)})$ we have that $g^{(k+1)}-g^{\star}=T(g^{(k)})-g^{\star}$, which implies
        $ e_{k+1}=T(g^{(k)}-g^{\star})=T(e_k)$ since $T$ is linear.
        As for the residuals, $f_{k+1}\approx \ell_{\rm{max}}T(e_k)-\ell_{\rm{max}}e_k =\ell_{\rm{max}} f_k$;  thus $f_{k+1}\approx \ell_{\rm{max}}f_k$.
        
     Using the last $m_k+1$ Anderson history residuals $f_{m-m_k}, \ldots, f_k$, we can write
     \begin{equation}
         |f_{k-j}|\approx |\ell_{\rm{max}}|^{-j}|f_k| \quad \text{for} \; j=0, \ldots,m_k
     \end{equation}
    Substituting these approximations into the Anderson objective function at iteration $k$ \eqref{eq:anderson_constrained}, where 
    $F_k \approx [ \ell_{\rm{max}}^{-m_k}f_k, \ldots, \ell_{\rm{max}}^{-1} f_k ]$, 
    we can recast \eqref{eq:anderson_constrained}
    %Anderson constrained problem 
 as a minimization problem for the weighted sum of squares 
 \begin{equation}
     \min_{\boldsymbol{\alpha} \in \mathbb{R}^{m_k}} \mathcal{J}_{\rm{and}}(\boldsymbol{\alpha})= \sum_{j=0}^{m_k} \frac{1}{|\ell_{\rm{max}}|^{j}} \alpha_j^2  \; \text{s.t.} \; \sum_{j=0}^{m_k} \alpha_j=1, 
     \label{eq:anderson_weighted_var_min}
 \end{equation}
 Problem \eqref{eq:anderson_weighted_var_min} represents a weighted variance minimization problem over the historical contributions. The associated Lagrangian is 
 $\mathcal{L}(\boldsymbol{\alpha}, \mu)=\sum_{j=0}^{m_k} |\ell_{\rm{max}}|^{-j}\alpha_j^2 -\mu \left(\sum_{j=0}^{m_k}\alpha_j-1\right)$. The optimality conditions give 
 \begin{equation}
 \alpha_j=\frac{\mu}{2}|\ell_{\rm{max}}|^{j} \quad \text{for each} \quad  j=0, \ldots, m_k.
     \label{eq:alpha_and_proportional_to_ellmax}
 \end{equation}
Substituting \eqref{eq:alpha_and_proportional_to_ellmax} into the sum-to-one constraint yields
$\frac{\mu}{2}=\frac{1}{\sum_{i=0}^{m_k}|\ell_{\rm{max}}|^i}$. By substituting this expression for $\mu/2$  into \eqref{eq:alpha_and_proportional_to_ellmax}, one obtains the result $\alpha_j=\frac{|\ell_{\rm{max}}|^j}{\sum_{i=0}^{m_k}|\ell_{\rm{max}}|^i}$.  Finally, we observe that, if $|\ell_{\rm{max}}|=1$, then \eqref{eq:anderson_weighted_var_min}  is satisfied by the uniform weights solutions $\alpha_j$ s.t. $\alpha_j=\frac{1^j}{\sum_{i=0}^{m_k} 1^i} =\frac{1}{m_k+1}$. 
   \end{subequations} 
 \end{proof}

\begin{remark}
\label{remark:anderson_sign}
    While, in the general Anderson acceleration algorithm, the coefficients $\alpha_j$ can take negative values to orthogonalize the residuals, our derivation shows that, under the assumption of scalar case and dominant-eigenvalue decay, the coefficients asymptotically approach a positive geometric distribution.
\end{remark}

\subsubsection{Anderson with memory adaptation}  \label{sec:anderson_memory}
As remarked in \cite{walker2011anderson}, if $m_{\rm{and}}$ is small, the secant information used by the method may be too limited to provide decidedly fast convergence, whereas if $m_{\rm{and}}$ is too large, the least-squares optimization problem for the $\boldsymbol{\alpha}$ coefficients \eqref{eq:anderson_constrained}  may be poorly conditioned. %Indeed, 
%the 
The most appropriate choice of $m_{\rm{and}}$ (and thus of $m_k$) is often problem- and DD-dependent, as we will show in %.  As will be shown in 
Section \ref{sec:results_anderson_memory}.  %, the basic Anderson-accelerated Schwarz scheme (Algorithm \ref{alg:DN_anderson}) can be sensitive to the parameter $m_{\text{and}}$.  
To mitigate this difficulty, we propose an adaptive variant of the method, which we term ``Anderson with memory adaptation''.  
%To remedy this problem in the case of large history windows--- for this example, in the case of $m_{\rm{and}}>3$--- we propose an adaptive version of Anderson. 
In this version of the algorithm, the value of $m_k$ is adapted on-the-fly using the following formula:  
\begin{equation}
m_k:= \begin{cases}
    \min(\bar{m}, m_{\rm{and}}), \;  \text{for} \, |e_{\rm{rel}}^{(k)}-e_{\rm{rel}}^{(k-1)}| <\epsilon_{\rm{and}} \,\text{and} \, k>2, \\
    \min(k, m_{\rm{and}}), \;  \text{otherwise}. \\
\end{cases}
\label{eq:anderson_with_memory_criterion}
\end{equation}
The criterion in \eqref{eq:anderson_with_memory_criterion} is based on the relative interface error at the current and previous Schwarz iterations and a prescribed threshold $\epsilon_{\rm{and}}$.   
%where $\bar{k}$ is a pre-selected parameter. 
The pre-selected parameter $\bar{m}$ can be chosen as a small integer, since, for small interface errors, the loss of relevant information is negligible. In addition, a small $\bar{m}$ discards outdated secant information from earlier iterations and keeps the constrained problem small.
While the idea of adapting $m_k$ on-the-fly is mentioned in \cite{walker2011anderson} (where the authors discuss the idea of dropping columns based on the condition number of the constrained minimization problem), the specific formula \eqref{eq:anderson_with_memory_criterion} for how to do this is new, to the best of our knowledge.

 \section{Numerical results}
\label{sec:num_results}
\subsection{One-dimensional Laplace equation}
\label{app:laplace}
We consider here a simple study case based on the 1D Laplace equation, with the goal of comparing the proposed relaxation and acceleration techniques for the non-overlapping Schwarz alternating method both numerically and theoretically.  Letting $\Omega = (0,1) \in \mathbb{R}$, the boundary value problem of interest is 
\begin{subequations}
\begin{equation}
\begin{cases}
     -u^{\prime\prime}(x)=0 \quad \text{in} \, (0,1)\\
     u(0)=0, \quad u(1)=1.
\end{cases}
\label{eq:laplace}
\end{equation}
It is possible to derive analytically the exact solution to \eqref{eq:laplace}: $u(x)=x$. 

In applying the non-overlapping Schwarz alternating method, we will partition $\Omega$ into two sub-domains, with $\bar{x}\in (0,1)$ denoting the interface point.  As in the previous section, we denote by $g^{(k)}$ the  interface data for the Dirichlet-Neumann Schwarz algorithm at iteration $k >0$. The sub-solutions $u_1^{(k)}$ and $u_2^{(k)}$ are obtained by solving the following Dirichlet and Neumann problems in $\Omega_1=(0,\bar{x})$ and $\Omega_2=(\bar{x},1)$:
\begin{equation}
\begin{cases}
    -{u_1^{(k)}}^{\prime \prime}(x)=0 \quad \text{in} \, (0,\bar{x})\\
    u_1^{(k)}(0)=0\\
    u_1^{(k)}(\bar{x})=g^{(k)} 
\end{cases}, 
\begin{cases}
       -{u_2^{(k)}}^{\prime \prime}(x)=0 \quad \text{in} \, (\bar{x},1)\\
    u_2^{(k)}(1)=1\\
    {u_2^{(k)}}^{\prime}(\bar{x})={u_1^{(k)}}^{\prime}(\bar{x})
\end{cases}
\quad \text{and} \quad g^{(k+1)}=\gamma(u_2^{(k)}).
\label{eq:laplace_dd}
\end{equation}
In our numerical study, we apply the Dirichlet-Neumann Schwarz iteration procedure \eqref{eq:laplace_dd} 
with classical relaxation, as well as with Aitken and Anderson acceleration.  
%using a Dirichlet-Neumann Schwarz scheme both (i) without relaxation and (ii) with classical relaxation for select values of the relaxation parameter $\rho$.  
We discretize the problem in both $\Omega_1$ and $\Omega_2$ using a finite difference discretization with $20$ uniformly-spaced points. %\irinanote{Need to ask Giulia about this} 
As will become apparent later, the interface location $\bar{x}$ is varied in our analyses.  Errors are computed with respect to a reference solution in $\Omega:=\Omega_1 \cup \Omega_2$, and the following algorithmic parameters are utilized: $g^{(1)}=0.3$, $\texttt{maxiter}=50$ and $\epsilon_{\text{abs}}=\epsilon_{\text{rel}}=10^{-8}$.

Before presenting our numerical results, we introduce and prove the following theorem, which is referred to later on, when our results are analyzed.  

\begin{theorem} \label{thm:laplace_optimal_rho}
    For the simple Laplace equation study case \eqref{eq:laplace} solved using the Aitken-accelerated Dirichlet-Neumann Schwarz scheme \eqref{eq:laplace_dd} applied to two sub-domains with an interface at $x = \bar{x}$, the optimal Aitken parameter $\rho^{(k)}$  satisfies
    \begin{equation}
\rho^{(k)}=\bar{x}.
\label{eq:aitken_laplace}
\end{equation}
\end{theorem}
\begin{proof}
    It is straightforward to show that 
\begin{equation}
\gamma(u_2^{(k)})=1-\frac{g^{(k)}}{\bar{x}}(1-\bar{x}),
\label{eq:u2_laplace}
\end{equation}
so that the interface residual jump is $\mathcal{E}(\gamma(u_1), \gamma(u_2))=\gamma(u_2)-g^{(k)}=1-\frac{g^{(k)}}{\bar{x}}$.  Recall the Aitken formula for $\rho^{(k)}$  \eqref{eq:aitken}, which simplifies in the 1D setting to % for $k\geq 1$ relies on an adaptive choice of the parameter 
$\rho^{(k)}=-\frac{\delta^{(k)}d^{(k)}}{|\delta^{(k)}|^2}$
with $\delta^{(k)}=\mathcal{E}^{(k)}-\mathcal{E}^{(k-1)}=
1-\frac{g^{(k)}}{\bar{x}}-[\gamma(u_2^{(k-1)})-\gamma(u_1^{(k-1)})]=\frac{g^{(k-1)}-g^{(k)}}{\bar{x}}
$ 
and $d^{(k)}=\gamma(u_1^{(k)})-\gamma(u_1^{(k-1)})=g^{(k)}-g^{(k-1)}$.
It follows that, for the study case of interest, the optimal Aitken parameter is given by \eqref{eq:aitken_laplace}
This result is consistent with the previous related study 
%This result is in accordance with the study of the convergence range for iterative DD methods 
in \cite{funaro1988iterative}. 
\end{proof}

\begin{table}[h!]
\centering
\resizebox{0.8\textwidth}{!}{%
\begin{tabular}{c|cccccc|c|c}
\hline
\multicolumn{7}{c}{\textbf{Classical relaxation $\rho$ values}} & \textbf{Aitken} ($N_0=2$) &\textbf{Anderson}  \\ 
\hline
$\bar{x}$& $0.1$& $0.2$& $0.5$& $ 0.7$& $0.8$& $1$ & $\rho^{(1)}=1$ & $\rho=1$\\
\hline
0.1 & \cellcolor{green!20}$2$& \cellcolor{red!10}$\texttt{maxit}$&  \cellcolor{red!10}$\texttt{maxit}$& \cellcolor{red!10}$\texttt{maxit}$& \cellcolor{red!10}$\texttt{maxit}$& \cellcolor{yellow!20}$\texttt{maxit}$ & $3$ & $3$\\
\hline
0.2 & $24$ & \cellcolor{green!20}$2$& $25$ & \cellcolor{red!10}$\texttt{maxit}$& \cellcolor{red!10}$\texttt{maxit}$& \cellcolor{yellow!20}$\texttt{maxit}$ & $3$ & $3$ \\
\hline
0.5 & $50$ & $33$ & \cellcolor{green!20}$2$& $20$ & $35$ & \cellcolor{yellow!20}$\texttt{maxit}$ & $3$ & $3$\\
\hline
0.7 & $\texttt{maxit}$ & $\texttt{maxit}$ & $15$ & \cellcolor{green!20}$2$& $11$ & \cellcolor{yellow!20}$23$ & $3$ & $3$\\
\hline
0.8 & $\texttt{maxit}$ & $\texttt{maxit}$& $19$ & $10$ & \cellcolor{green!20}$2$& \cellcolor{yellow!20}$14$ & $3$ & $3$\\
\hline
\end{tabular}
}
\caption{1D Laplace equation.  Number of Schwarz iterations required to reach convergence  tolerances $\epsilon_{\text{abs}} = \epsilon_{\text{rel}}=10^{-8}$ for different values of $\rho$ and $\bar{x}$ when employing classical relaxation, Aitken acceleration and Anderson acceleration. }
\label{tab:laplace_iters}
\end{table}
\begin{table}[h!]
\centering
\resizebox{1\textwidth}{!}{%
\begin{tabular}{c|cccccc|c|c}
\hline
\multicolumn{7}{c}{\textbf{Classical relaxation $\rho$ values}} & \textbf{Aitken} ($N_0=2$) & \textbf{Anderson}  \textbf{Aitken} ($N_0=2$)\\  
\hline
$\bar{x}$& $0.1$& $0.2$& $0.5$& $ 0.7$& $0.8$& $1$ & $\rho^{(1)}=1$ & $\rho =1$\\
\hline
0.1 & \cellcolor{green!20}$5.6 \times 10^{-17}$& \cellcolor{red!10}$4 \times 10^{-1}$&  \cellcolor{red!10}$10^{29}$& \cellcolor{red!10}$10^{38}$& \cellcolor{red!10}$10^{41}$& \cellcolor{yellow!20}$10^{47}$ & $2.22 \times 10^{-16}$& $3.5 \times 10^{-13}$ \\
\hline
0.2 & $6.0 \times 10^{-9}$ & \cellcolor{green!20}$2.8 \times 19^{-17}$& $8.9 \times 10^{-9}$ & \cellcolor{red!10}$10^{19}$& \cellcolor{red!10}$10^{22}$& \cellcolor{yellow!20}$10^{29}$ & $2.22 \times 10^{-16}$& $5.2 \times 10^{-12}$ \\
\hline
0.5 & $7.1 \times 10^{-7}$ & $6.4 \times 10^{-9}$ & \cellcolor{green!20}$0$& $7.7 \times 10^{-9}$ & $9.2 \times 10^{-9}$& \cellcolor{yellow!20}$4.0 \times 10^{-1}$ & 0 & $3.13 \times 10^{-11}$ \\
\hline
0.7 & $3.0 \times 10^{-5}$ & $7.9 \times 10^{-9}$ & $6.9 \times 10^{-9}$ & \cellcolor{green!20}$0$& $10^{-9}$ & \cellcolor{yellow!20}$4.6 \times 10^{-9}$ & 0 & $3.44 \times 10^{-11}$ \\
\hline
0.8 & $9.0 \times 10^{-5}$ & $9.4 \times 10^{-8}$& $6.7 \times 10^{-9}$ & $3.2 \times 10^{-9}$ & \cellcolor{green!20}$0$& \cellcolor{yellow!20}$9.3 \times 10^{-9}$ & 0 &  $3.61 \times 10^{-11}$ \\
\hline
\end{tabular}
}
\caption{1D Laplace equation.  Error $\epsilon_{\text{abs}}$ obtained at the last Schwarz iteration for different values of $\rho$ and $\bar{x}$ when employing classical relaxation, Aitken acceleration and Anderson acceleration.}
\label{tab:laplace_errors}
\end{table}

\subsubsection{Schwarz with classical relaxation} \label{sec:results_laplace_classical}

First, we present results for the non-overlapping Schwarz alternating method applied to \eqref{eq:laplace} accelerated using classical relaxation  (Algorithm \ref{alg:DN_relax}).  
 % on 
%the interface error $|g^{(k+1)}-g^{(k)}|$ (reported in Table \ref{tab:laplace_errors}). 
%\irinanote{I wanted to double check something.  You are varying $\bar{x}$ in the analyses.  As you vary it, are you always using 20 mesh points in each sub-domain?  That would mean the sub-domains are not discretized with the same resolution and the mesh points in the sub-domains will not be coincident with the reference solution mesh points, unless $\bar{x} = 0.5$.  If you have different resolutions in the sub-domains, this brings in another variable, namely the impact of different resolutions in different sub-domains.}
Columns 2--7 of Table \ref{tab:laplace_iters} report the number of Schwarz iterations required to reach convergence for several different values of the $\rho$ parameter.  Similarly, columns 2--7 of Table \ref{tab:laplace_errors} report the errors obtained at the last Schwarz iteration for different values of $\rho$.  The reader can observe that the best convergence and minimal error is reached for the choice of $\rho=\bar{x}$ for all the values of $\bar{x} \in (0,1)$ tested (see the squares marked in green in Tables \ref{tab:laplace_iters} and \ref{tab:laplace_errors}).  This is consistent with the result in Theorem \ref{thm:laplace_optimal_rho}.  
In this case, the method converges
in two iterations, since the stopping criterion is based on the difference between successive iterates, $e^{(k)}_{\text{abs}}$. With the optimal relaxation parameter, the interface iteration yields the exact interface value after the first iteration, so that the true error is already zero, i.e. $e_1=|g^{(1)}-g^{\star}|=|g^{(1)}-\bar{x}|$, as expected from our theoretical analysis.  Figure \ref{fig:iter_solutions}(a) shows the solutions for Schwarz iterations $k=1,2$ when applying the optimal value of $\rho$ in the Schwarz algorithm with classical relaxation: the solution at iteration two coincides with the exact solution.

We now attempt to explain theoretically the non-convergence cases in Tables \ref{tab:laplace_iters} and \ref{tab:laplace_errors}, marked with a pink color.  
%The value in \eqref{eq:aitken_laplace} is optimal: indeed, 
If we go back to the classical relaxation formula
\begin{equation*}
g^{(k+1)}=(1-\rho) g^{(k)}+\rho \gamma(u_2^{(k)})
\end{equation*}
and replace the last term with the value obtained for the trace of the solution to the Neumann problem \eqref{eq:u2_laplace}, we obtain 
\begin{equation}
g^{(k+1)}=\Big[1-\frac{\rho}{\bar{x}}\Big]g^{(k)}+\rho.
\label{eq:relaxation_laplace}
\end{equation}
Now, we know from the exact solution of \eqref{eq:laplace} that $g^{\star}=\bar{x}$, where $g^{\star}$ is the solution to our fixed-point iteration problem.
%Now, we define the fixed-point $g^{\star}=\alpha$ known from the exact solution of \eqref{eq:laplace}. % the true sol is u(x)=x
The fixed-point iteration error  at iteration $k$ can thus be written as $e_{k}=|g^{(k)}-g^{\star}|=|g^{(k)}-\bar{x}|$. By using \eqref{eq:relaxation_laplace}, the error at iteration $k+1$ now takes the form 
\[
e_{k+1}=\left|1-\frac{\rho}{\alpha}\right| \, e_{k}, 
\] so that the fixed-point iteration error decays with rate 
\begin{equation} \label{eq:lambda}
\lambda(\rho)=\frac{e_{k+1}}{e_{k}}=\left|1-\frac{\rho}{\bar{x}}\right|
\end{equation}
For convergence, we would need the following constraint:  $|\lambda(\rho)|<1$.  In order for this to hold for \eqref{eq:lambda}, we must have that  $0< \rho <2\bar{x}<2$, 
since $\bar{x} \in (0,1)$. We observe, however, that for $\bar{x}$ very small, we have that $\frac{\bar{x}-\rho}{\bar{x}}\sim -\frac{\rho}{\bar{x}}$: in this case, the condition would be that
$-\frac{\rho}{\bar{x}}<1$, which implies that we must have $\rho<\bar{x}$ for convergence. This explains the divergence cases highlighted in pink in Tables \ref{tab:laplace_iters}-- \ref{tab:laplace_errors}, 
which correspond to $\alpha$ being very small and/or $\rho>\bar{x}$.  %, which corresponding to cells highlighted in pink in the tables.
The analysis above implies that 
%We also observe that 
an at least superlinear convergence is expected when $\lambda=0$.  This case corresponds to the condition that $\rho=\bar{x}$, which coincides with the optimal $\rho^{(k)}$ derived in Theorem \ref{thm:laplace_optimal_rho}.

Finally, we observe that,  for the case of no relaxation (i.e., $\rho=1$), which corresponds to the yellow-shaded columns in Tables \ref{tab:laplace_iters} and \ref{tab:laplace_errors}, the condition that the convergence factor $|\lambda(\rho)| < 1$ would require that 
$\left|\frac{\bar{x}-1}{\bar{x}}\right|<1,$ which implies that $\bar{x}>\frac{1}{2}$, using the assumption that $\bar{x} \in (0,1)$. This theoretical result is confirmed by our numerical experiments: for $\bar{x}<\frac{1}{2}$ (the first three yellow rows), we observe a much slower convergence and, in some cases, actual divergence of the Schwarz alternating method.

\subsubsection{Schwarz with Aitken acceleration}
%\textcolor{red}{
Next, we solve problem \eqref{eq:laplace} using Aitken-accelerated Schwarz  with $\rho^{(1)}=1$.  We find that, for the same tested values of $\bar{x}$, just three iterations are needed to reach convergence (see the second-to-last columns of Tables \ref{tab:laplace_iters} and \ref{tab:laplace_errors}).  In this case, the true error is zero after the second iteration, since, at $k=1$, we do not apply Aitken yet. In Figure \ref{fig:iter_solutions}(b) the solutions at iteration $k=1,2,3$ are depicted.
Not only is the optimal value $\rho^* = \bar{x}$ obtained within the algorithm upon convergence, the Aitken acceleration algorithm detects this value for all $k$, as well as for 
all values of $\bar{x}$ tested. 
 This numerical result is consistent with the theoretical result in \eqref{eq:aitken_laplace}.   
 
\subsubsection{Schwarz with Anderson acceleration}

To complete this brief analysis, we lastly consider Anderson-accelerated Schwarz for the specific values of  $m_{\rm{and}}=1$ and $\rho = 1$ (no relaxation) in the Anderson acceleration algorithm (c.f., \eqref{eq:anderson_update}). Since $m_{\rm{and}}=1$, Anderson with memory adaptation (Section \ref{sec:anderson_memory}) is not relevant.  The reader can observe from examining the last columns of Tables \ref{tab:laplace_iters} and \ref{tab:laplace_errors} that the Anderson-accelerated Schwarz converges in just three iterations to a solution whose error is close to machine precision.  

We now demonstrate that this numerical result is expected theoretically.  By computing the sub-solutions to \eqref{eq:laplace_dd} in $\Omega_1$ and $\Omega_2$ we obtain that
$T(g^{(k)})=\left(1-\frac{1}{\bar{x}}\right)g^{(k)}+1$.
The Anderson update in the unrelaxed case takes the form (c.f., \eqref{eq:anderson_update}) 
\begin{equation}
    g_{\rm{and}}^{(k+1)}=\alpha_0 T(g^{(k-1)})+\alpha_1 T(g^{(k)}),
\end{equation}
where we explicitly indicated with the subscript $\rm{and}$ that the new datum $g$ is found using the Anderson acceleration algorithm.  
The coefficients $\alpha_0$ and $\alpha_1$ solve the constrained minimization problem (c.f., \eqref{eq:anderson_constrained})
\[
\min_{\tilde{\alpha}_0, \tilde{\alpha}_1} |\tilde{\alpha}_0 f_{k-1}+\tilde{\alpha}_1 f_k| \quad \text{s.t. } \tilde{\alpha}_0+\tilde{\alpha}_1=1, 
\]
with $f_{k-1}=1-\frac{g^{(k-1)}}{\bar{x}}$ and $f_k=1-\frac{g^{(k)}}{\bar{x}}$.
The reader can easily check that the solutions to the minimization problem are 
\[
    \alpha_0=\frac{g^{(k)}-\bar{x}}{g^{(k)}-g^{(k-1)}},\quad
    \alpha_1=\frac{\bar{x}-g^{(k-1)}}{g^{(k)}-g^{(k-1)}}, 
\]
and thus we obtain the following update for the new iteration ($k=2$):
\begin{equation}
   g_{\rm{mix}}:=\alpha_0 g^{(k-1)}+\alpha_1 g^{(k)}=\bar{x}=g^{\star}.
   \label{eq:anderson_gmix}
\end{equation}
Since $T$ is an affine map in this simple study case, we can write 
\begin{equation*}
g^{(k+1)}_{\rm{and}}=\alpha_0 T(g^{(k-1)})+\alpha_1 T(g^{(k)})=T\left(\alpha_0 g^{(k-1)}+\alpha_1 g^{(k)}\right)= T(g_{\rm{mix}})=T(g^{\star})=g^{\star}
\end{equation*}
by using \eqref{eq:anderson_gmix} and the definition of the operator $T$. 

The above analysis shows that, for this simple study case, the Anderson-accelerated Schwarz method (Algorithm \ref{alg:DN_anderson}) should produce the exact fixed-point solution in just one acceleration step, after the two history-steps required to compute $g^{(k-1)}$ and $g^{(k)}$. This result is numerically verified  in the last columns of Tables \ref{tab:laplace_iters} and \ref{tab:laplace_errors}.  

%In Table \ref{tab:laplace_iters}, we report in the number of required Schwarz iterations for convergence.  We observe that Anderson converges in $1$ acceleration step, but requires $2$ startup iterations to compute residual history (see Algorithm \ref{alg:DN_anderson}); this is why $3$ iterations are computed in the last column of Table \ref{tab:laplace_iters}. The associated errors with respect to the analytical solution are reported in Table \ref{tab:laplace_errors}.  
Finally, in Figure \ref{fig:iter_solutions}(c) the Anderson-accelerated Schwarz solution is depicted for different Schwarz iterations. %\textcolor{red}{
Discarding iteration $k=1$, for which the Anderson history matrix $F_k$ is empty, we can numerically compute the solutions to the constrained minimization problem \eqref{eq:anderson_constrained}. 
The resulting optimal values of $\alpha_0$ and $\alpha_1$ are shown as histograms super-imposed on top of the solution plots in Figure \ref{fig:iter_solutions}(c).
%, which also show the computed solutions at each Schwarz iteration until convergence.  
%: the solutions values are reported in the subplots in Figure \ref{fig:anderson_relax_iter_solutions}. 
The reader can observe that the asymptotic value of $\boldsymbol{\alpha}^{\star}$ is $[0,1]^T$.  Since the Anderson residual is 0 after the first two iterations which create the history required to apply the approach, %: for this simple test case, after the two first iterations for history construction, Anderson residual is $0$, so that 
the constrained problem \eqref{eq:anderson_constrained}
is solved by zeroing the old non-negligible residual (by setting $\alpha_0=0$) and choosing $\alpha_1=1$ to satisfy the constraint. 
%In the same figure, Anderson solution is compared with the solution obtained by classical relaxation with $\rho$ set to the optimal (Aitken) choice for the case $\bar{x}=0.7$. 
\begin{figure}[h]
\centering
\subfloat[Schwarz solutions with classical relaxation and the optimal choice of $\rho$ for iteration $k= 1$ (left) and $k=2$ (right)]{
\includegraphics[scale=0.13]{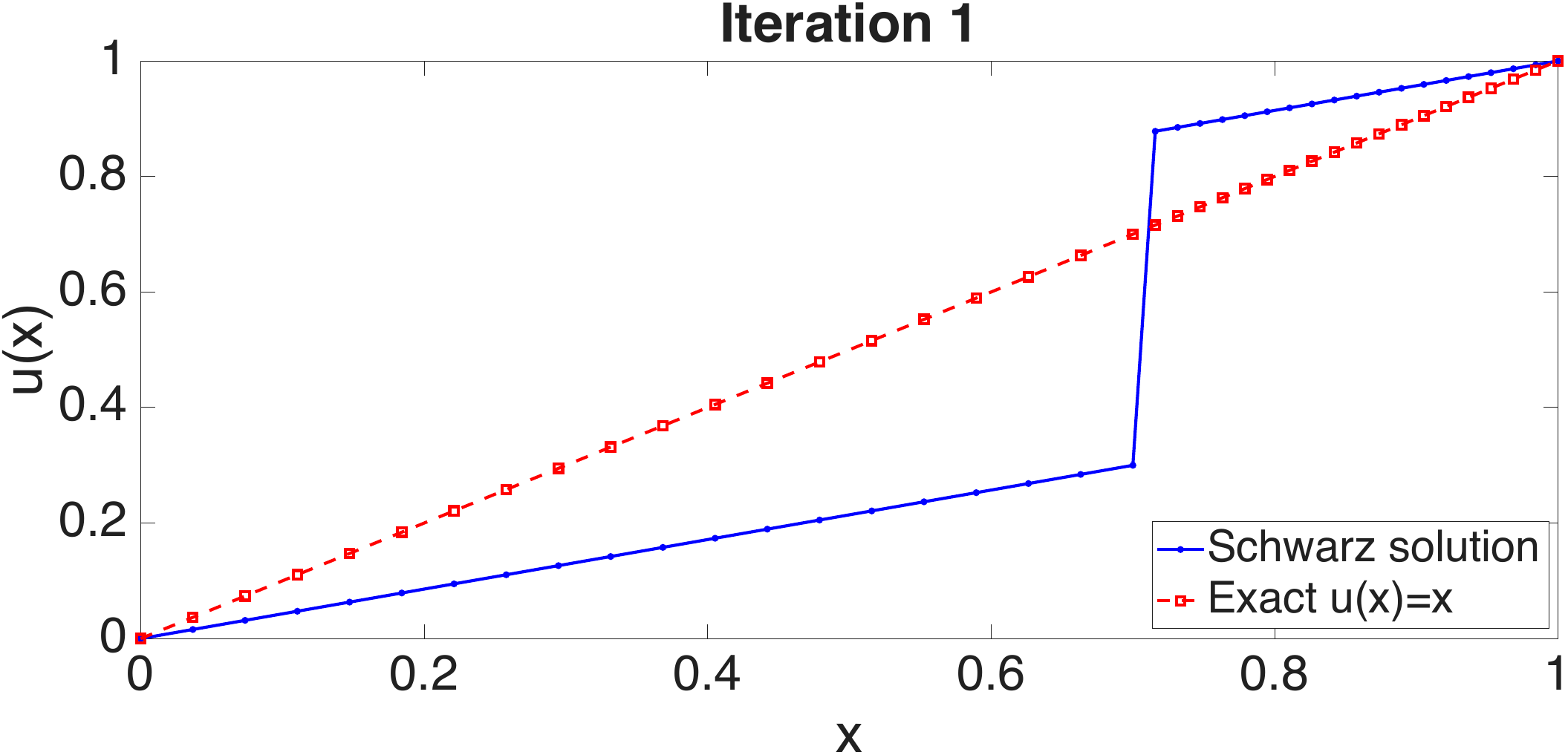}
\includegraphics[scale=0.13]{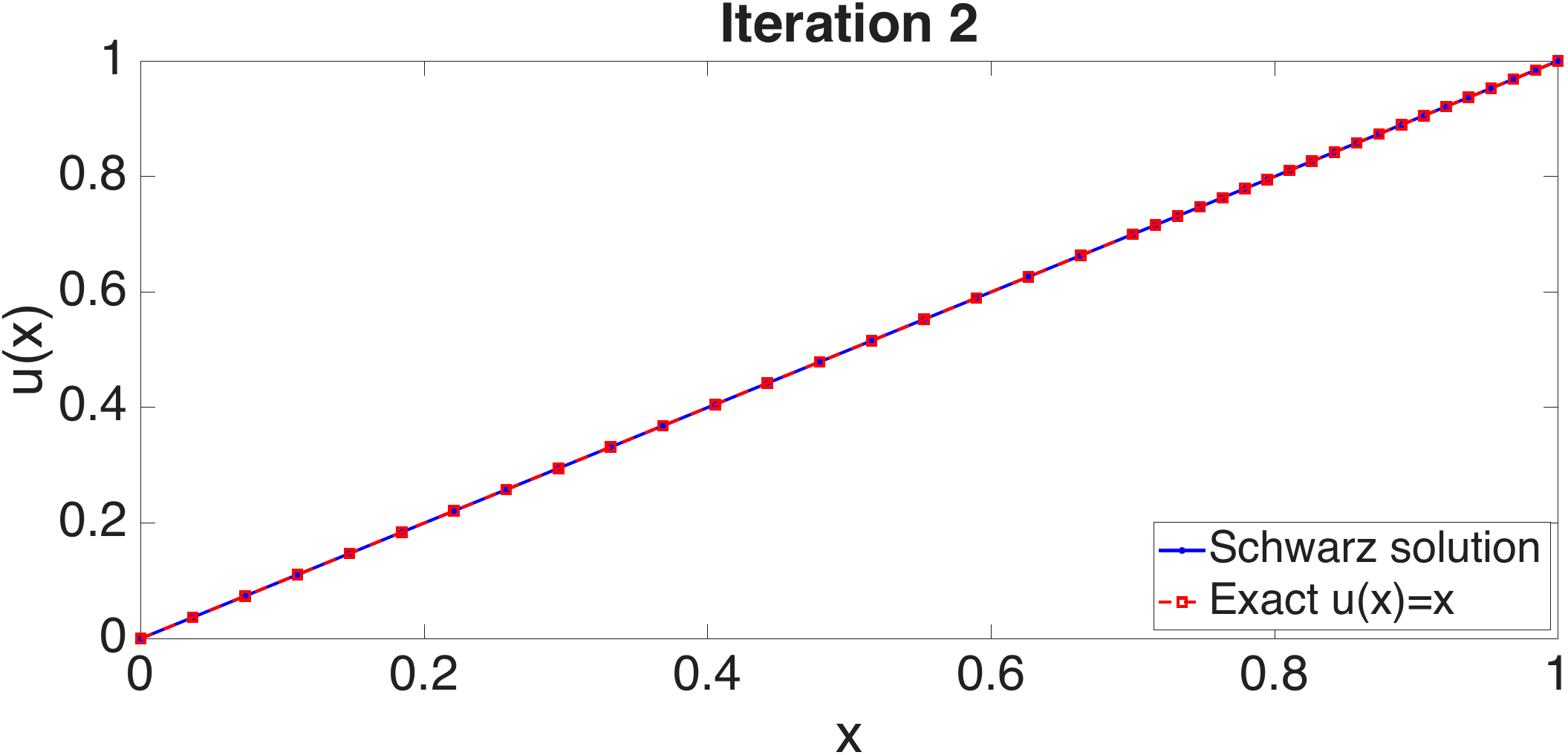}
}
\quad
\subfloat[Aitken-accelerated Schwarz solutions for iterations $k=1$ (left), $k=2$ (middle) and $k=3 $ (right)]
{
\includegraphics[scale=0.16]{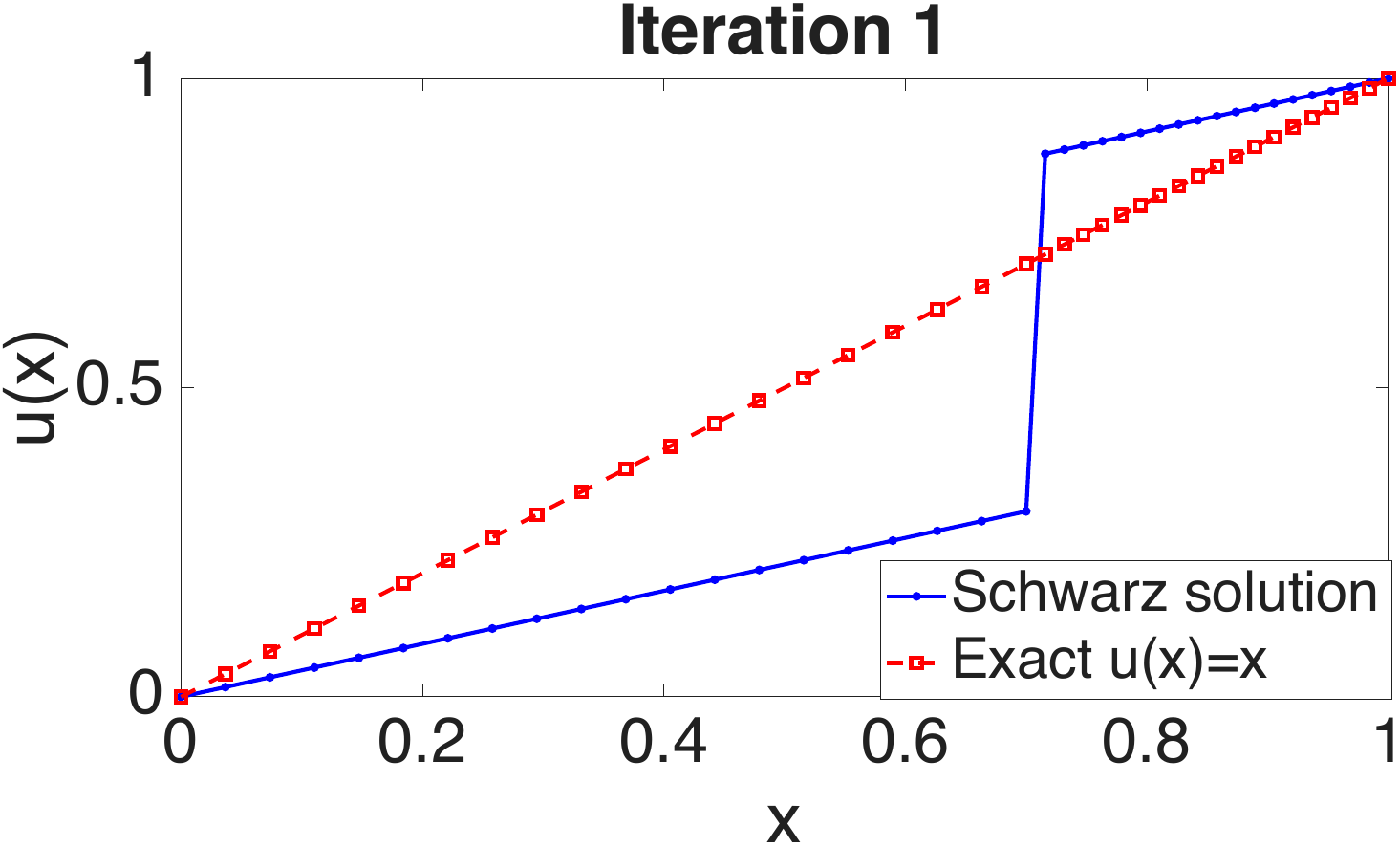}
\includegraphics[scale=0.16]{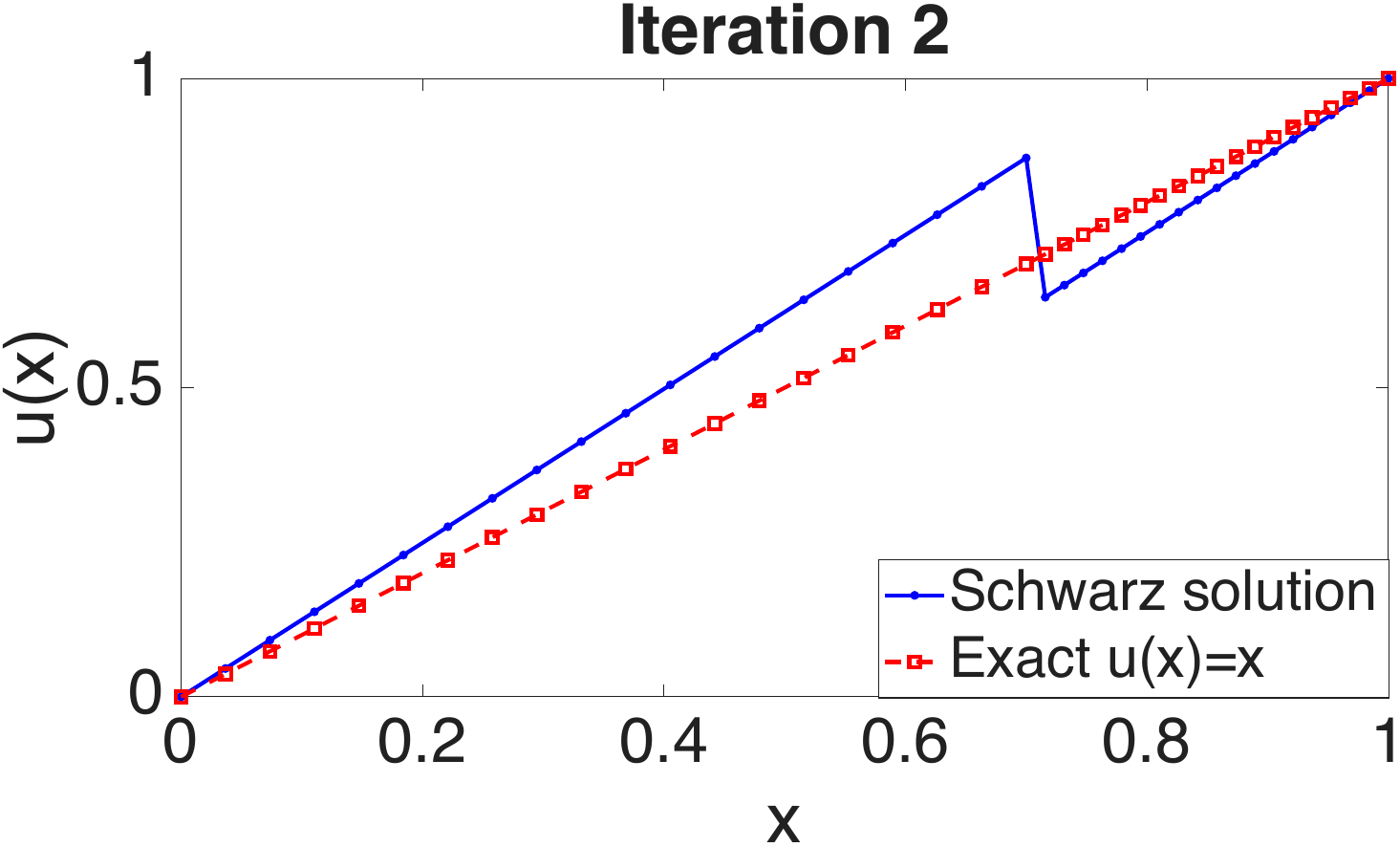}
\includegraphics[scale=0.16]{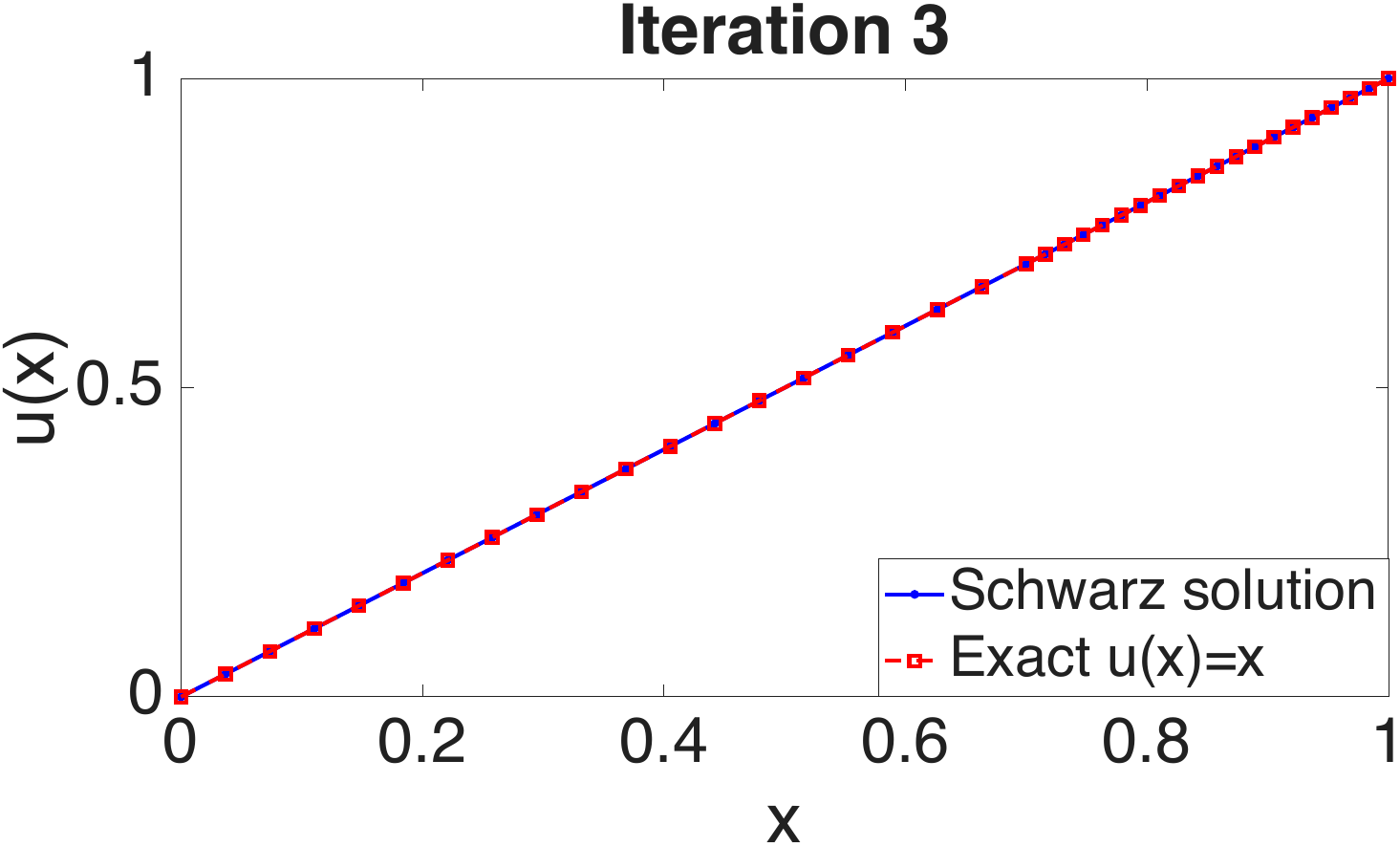}
}
\quad
\subfloat[Anderson-accelerated Schwarz solutions for iterations $k=1$ (left), $k=2$ (middle) and $k=3 $ (right)]{
\includegraphics[scale=0.12]{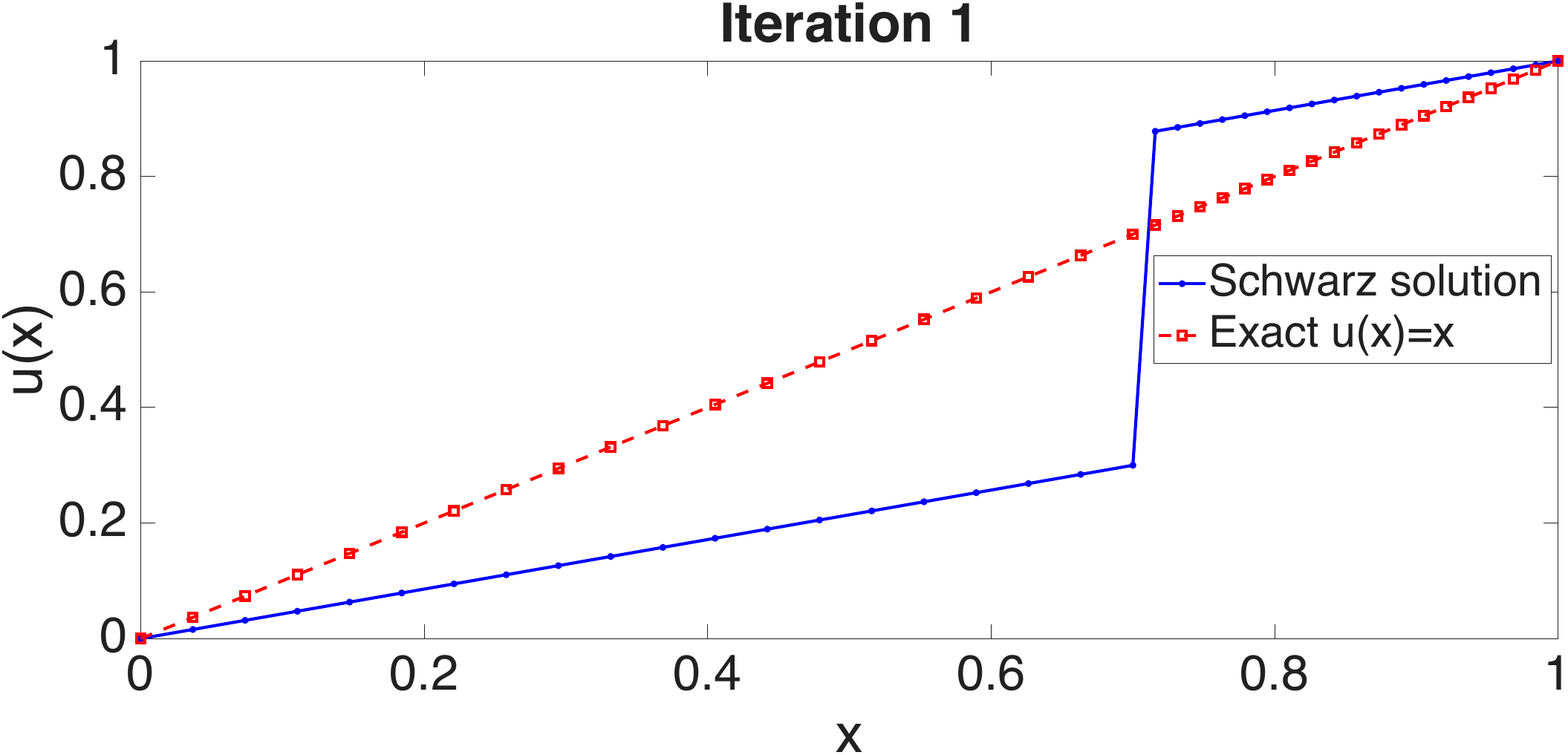}
\includegraphics[scale=0.12]{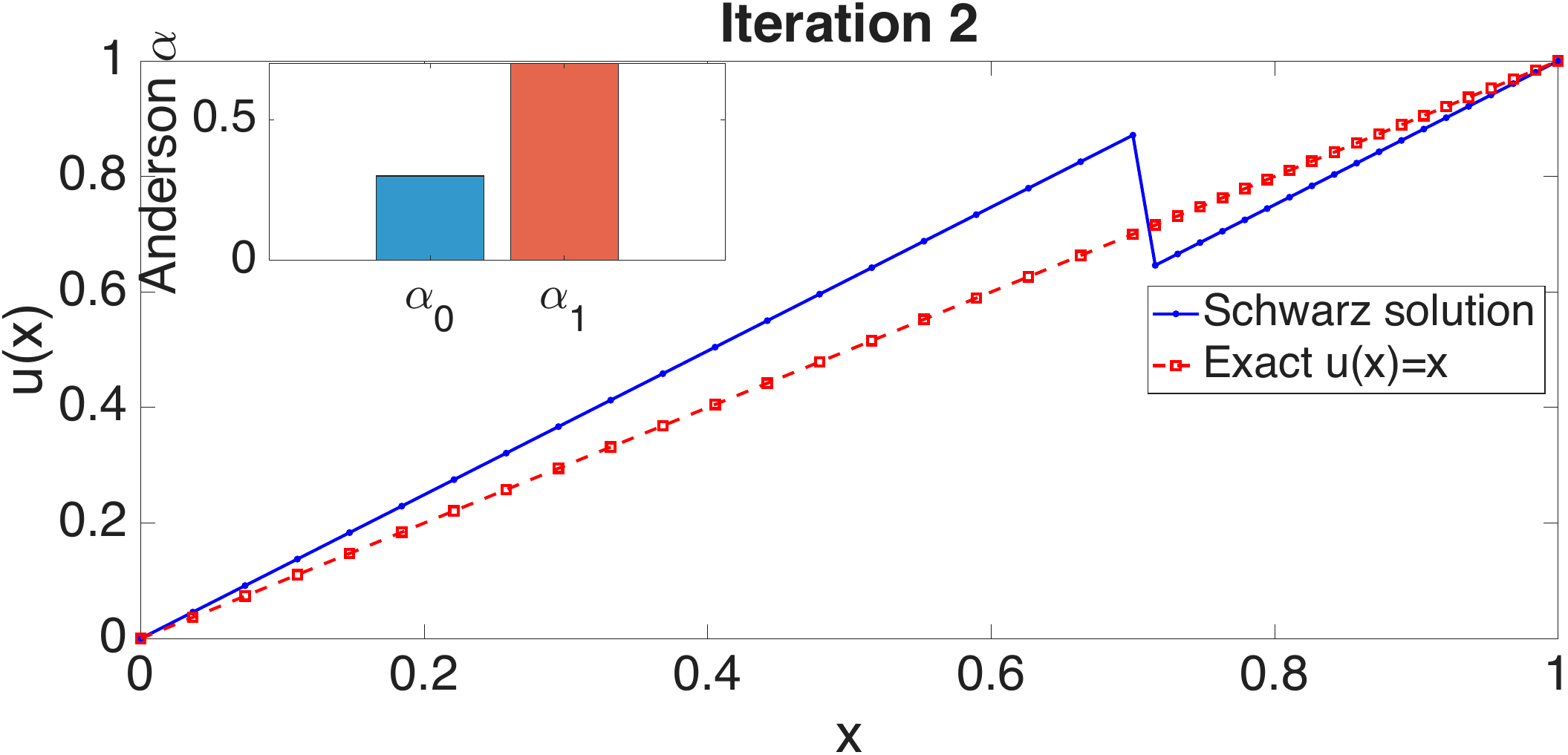}
\includegraphics[scale=0.12]{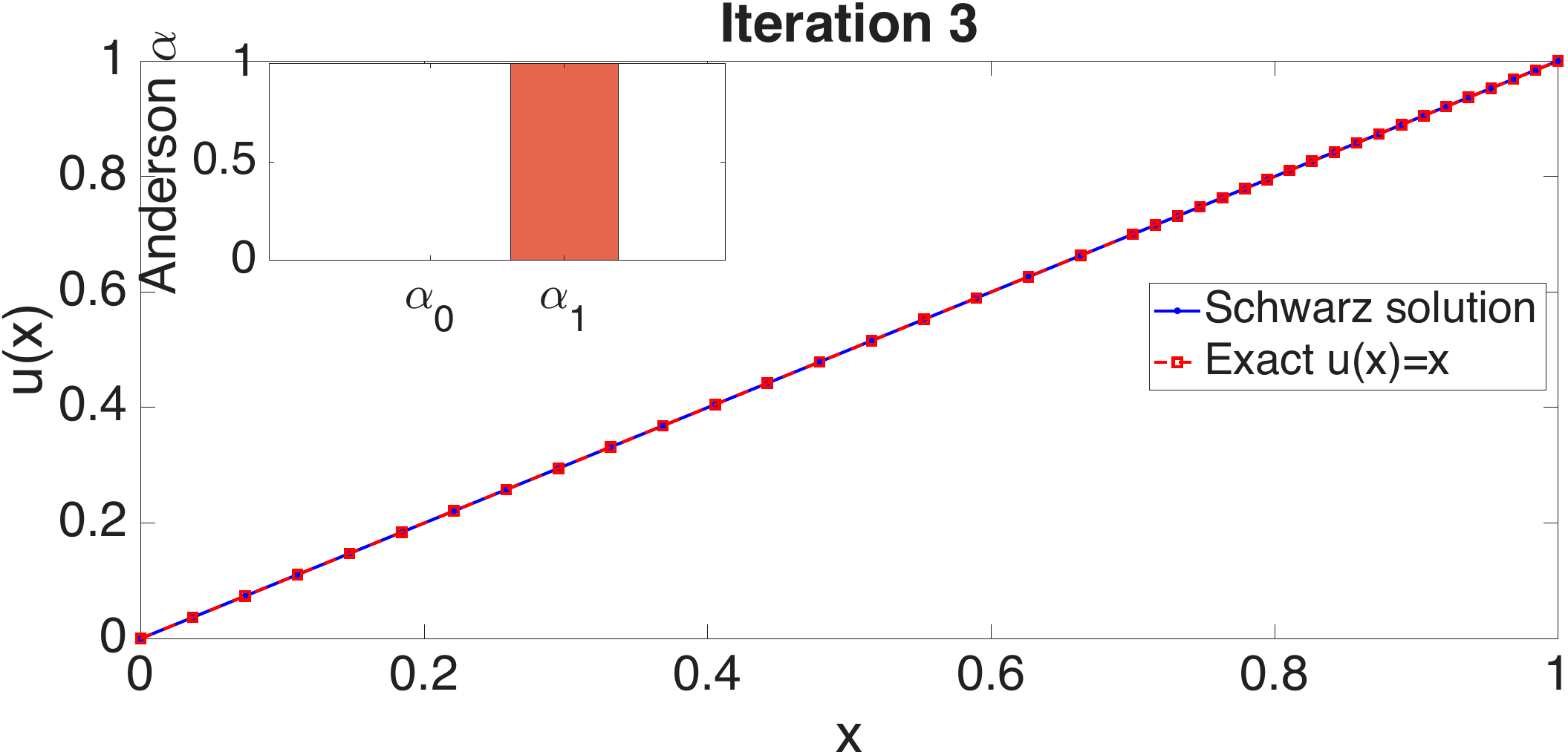}
}
\caption{1D Laplace equation.  (a) Classical relaxation-accelerated Schwarz solutions with the optimal choice of $\rho=\bar{x}$. (b) Aitken accelerated Schwarz solutions with $\rho^{(1)}=1$.  (c) Unrelaxed Anderson-accelerated Schwarz solutions.   The middle and right sub-plots in (c) also show the values of $\alpha_0$ and $\alpha_1$ in histogram form for $k=2$ and $k=3$, respectively.  All subplots are for the specific case of $\bar{x} = 0.7$.  }
\label{fig:iter_solutions}
\end{figure}
\end{subequations}

%\irinanote{There are some more cleanups/clarifications that can be made in this section.  I can do this in a subsequent iteration, maybe after the section is moved to the body of the paper.}

\subsection{Two-dimensional nonlinear elasticity}
For our second example, we numerically study the convergence of the Dirichlet-Neumann Schwarz alternating method with different relaxation/ acceleration strategies for a steady nonlinear elasticity problem.  
In particular, we compare the numerical performance of the relaxation and acceleration methods summarized in Section \ref{sec:acceleration} on a 2D simplification of the cuboid test proposed in \cite{mota2017schwarz}.
More specifically, we solve a nonlinear elasticity problem $-\nabla \cdot \mathbf{P}(\mathbf{F}(\mathbf{u})) = \mathbf{0}$ in $\Omega$ expressed in terms of the the deformation gradient  $\mathbf{F}(\mathbf{u}) = \mathbb{I} + \nabla \mathbf{u}$ and the displacement $\mathbf{u}$. The stress tensor is given by the first Piola-Kirchhoff stress tensor \( \mathbf{P} \):
\begin{equation} \label{eq:P}
\mathbf{P}(\mathbf{F}) = \lambda_2 \left( \mathbf{F} - \mathbf{F}^{-T} \right) + \lambda_1 \left( \log(\det(\mathbf{F})) \right) \mathbf{F}^{-T},
\end{equation}
implying a Neohookean-type material model \cite{Mota:2011}.  The Lamé constants \( \lambda_1 \) and \( \lambda_2 \) are defined in terms of the Young’s modulus \( E \) and Poisson’s ratio \( \nu \) as
$\lambda_1 = \frac{E \nu}{(1 + \nu)(1 - 2\nu)},
\lambda_2 = \frac{E}{2(1 + \nu)}.$  
We depict, in Figure \ref{fig:cuboid_glo}(a), the global computational domain together with the system boundaries, denoted by $\Sigma_i$ for $i=1,...,4$. 
The problem is characterized by the following Dirichlet BCs: $\mathbf{u}\cdot \mathbf{n}=0$ at boundaries $\Sigma_1$ and $ \Sigma_3$; $\mathbf{u}\cdot \mathbf{n}=1$ on boundary $\Sigma_2$ and % with $\Delta=1$. On the top boundary $\Sigma_4$ a Neumann boundary condition is imposed on the stress tensor:
   $ \mathbf{P}(\mathbf{F}(\mathbf{u})) \cdot\mathbf{n} = \mathbf{0}$ on the top boundary $\Sigma_4$.
We set the same material parameters values as in the three-dimensional (3D) example in \cite{mota2017schwarz}, in particular, we use $E = 1440$ Pa and $\nu = 0.25$.  
Figure \ref{fig:cuboid_glo}(c) shows the displacement magnitude of the global solution $\mathbf{u}$ in $\Omega$.  The reader can observe that the BCs applied have caused the domain to stretch by $1$ m in the $x$ direction.

\begin{figure}[h!]
\centering
\subfloat[Global domain $\Omega$]{
    \includegraphics[scale=0.17]{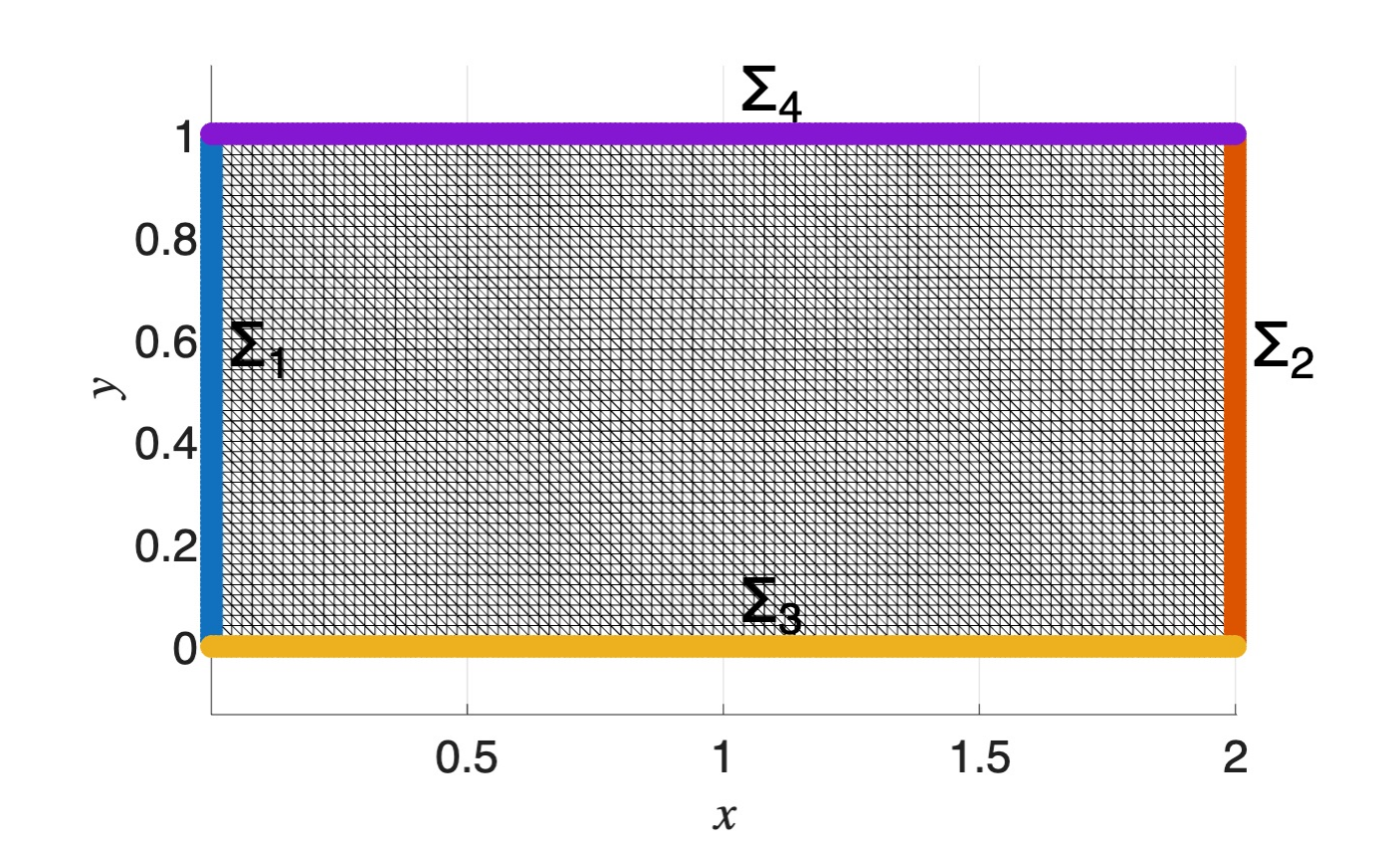}
    }
  \subfloat[Two sub-domain DD of $\Omega$]{
   \begin{tikzpicture}
  % include the figure
  \node[inner sep=0] (img) {\includegraphics[width=0.27\textwidth]{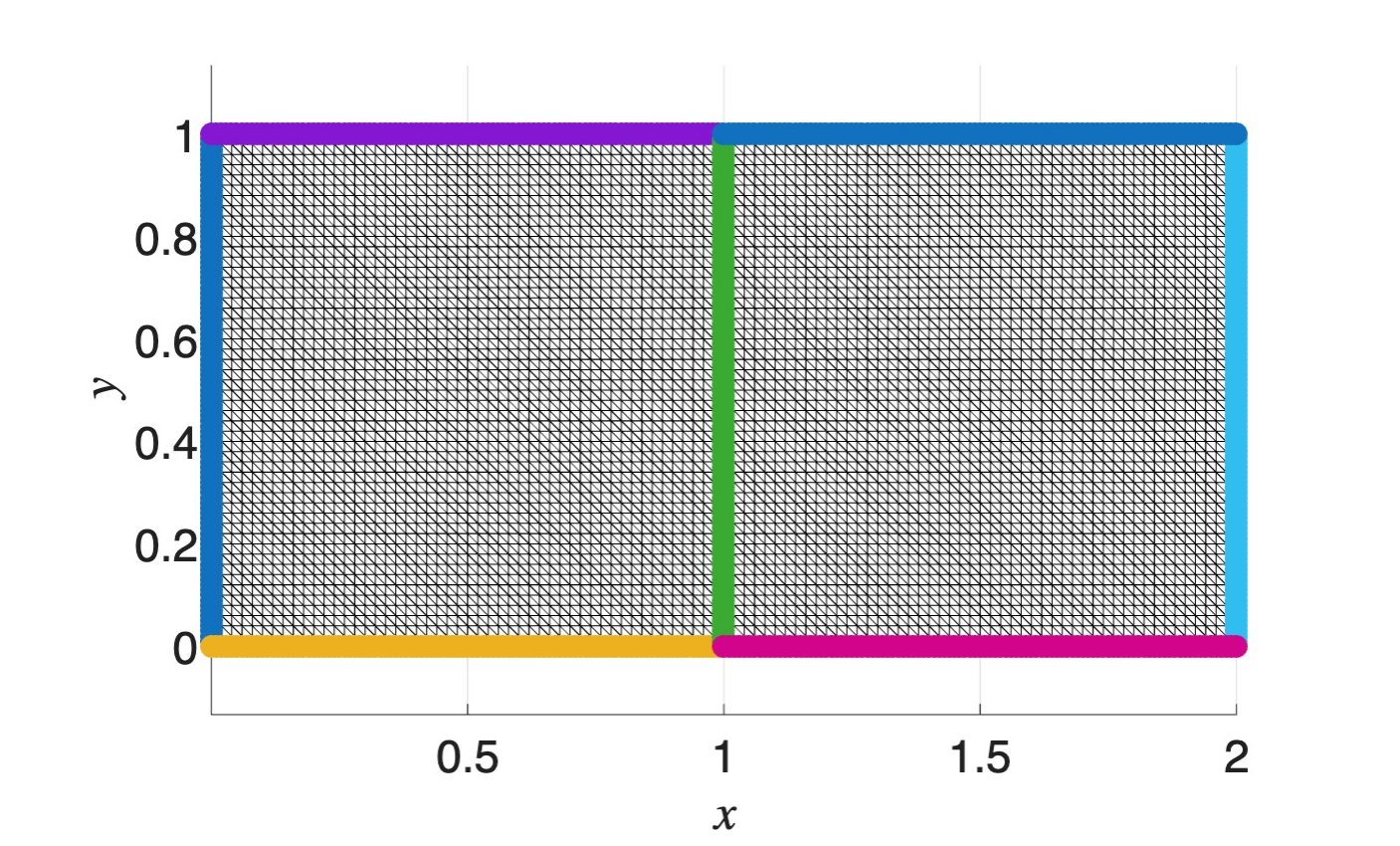}};

  % overlay text
  \node at (1.5,0.6) {$\Omega_1$};
  \node at (-1.3,0.6) {$\Omega_1$};

  % overlay arrows
  %\draw[->,thick] (0.5,0.2) -- (1.2,0.9);
  % Internal interface Gamma
  \node[black, right, font=\Large] at (0.18,0.5) {$\Gamma$};
  % Normal n_12 (pointing from Omega_1 to Omega_2)
%  \draw[->, thick] (0,0.15) -- ++(0.5,0)
%      node[midway, below] {$\mathbf{n}_{12}$};
\end{tikzpicture}
   }
    \subfloat[Global solution $|\mathbf{u}|$ ]{
    \includegraphics[scale=0.24]{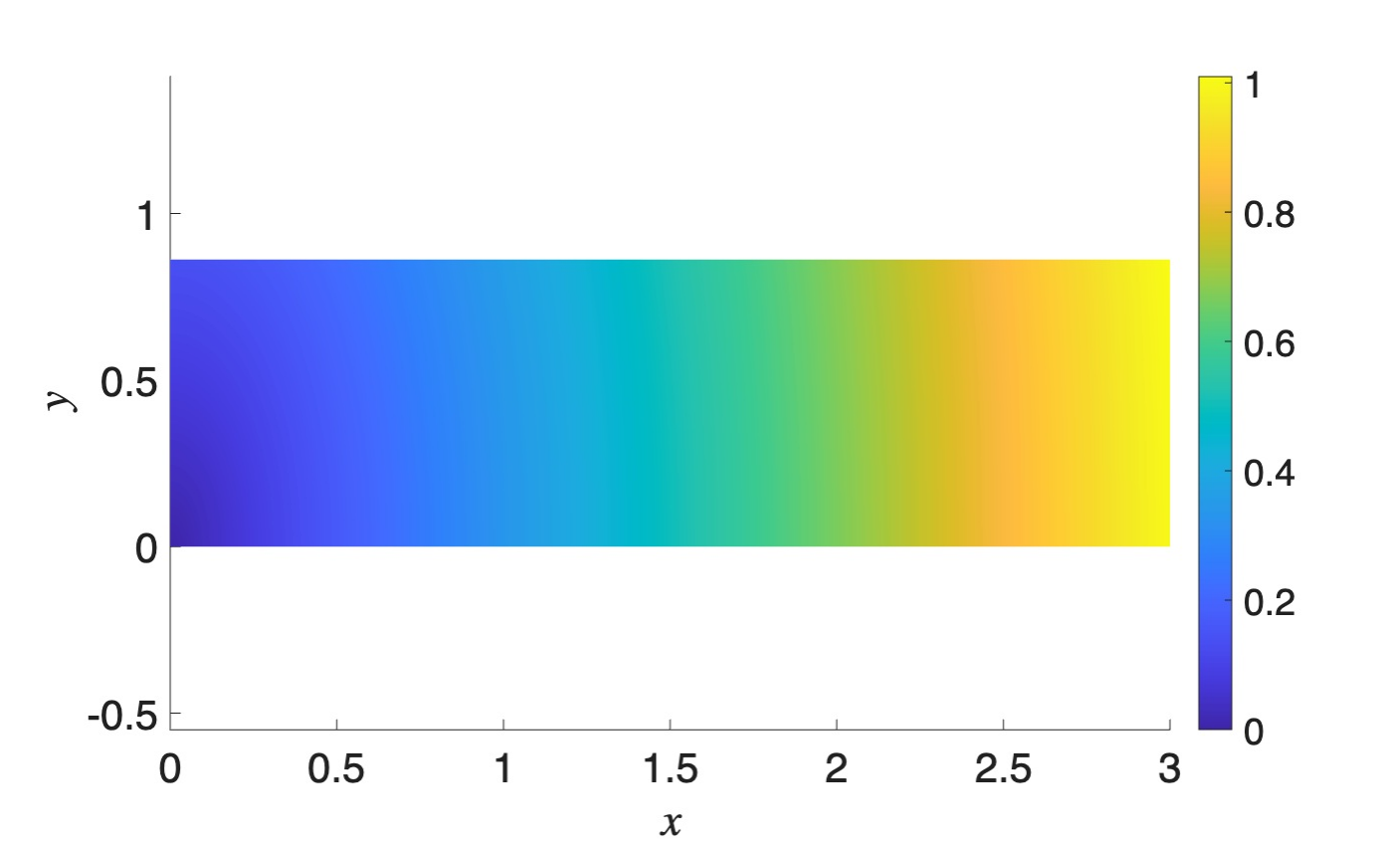}
    }
    \caption{2D nonlinear elasticity.  (a) The global domain.  (b) The global displacement magnitude solution.  (c) Non-overlapping partition with $N_{\text{dd}} = 2$.  }
    \label{fig:cuboid_glo}
    \end{figure}
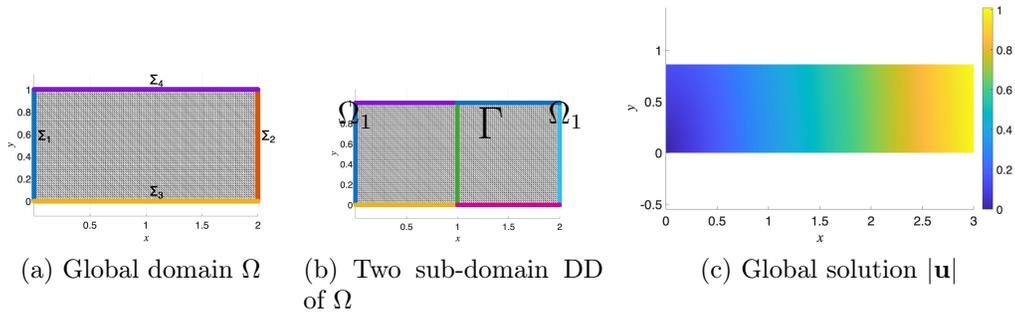

In studying the performance of the non-overlapping Schwarz alternating method, the global domain $\Omega$ is decomposed into two non-overlapping sub-domains discretized with uniform triangular meshes comprised of $10201$ nodes, as shown in Figure \ref{fig:cuboid_glo}(b). The characteristic measure of the finite element size $h$ is set to $0.1$ for both the global mesh and the sub-meshes.  While the Schwarz alternating method is capable of coupling sub-domains having different meshes and mesh resolutions, we leave this use case to a future work.  In applying the non-overlapping Schwarz alternating method, we specify the following algorithmic parameters: $\epsilon_{\rm{abs}} = \epsilon_{\rm{rel}}=10^{-6}$, $\texttt{maxit}=100$, and  $\mathbf{g}^{(1)}=\mathbf{0}$.  For the Newton nonlinear solver, we use a tolerance of  $10^{-8}$.

%\begin{remark}
%    The reader can observe that the meshes are conformal at the interface boundary $\Gamma$.  We 
%assume in this work that the domain partition is given; we note that it may be possible to optimize the convergence of the Schwarz alternating method through a clever choice of the domain decomposition, 
%but doing so is beyond the scope of this work.  
%\end{remark}

%he choice of an \textit{a priori} partition which could enhance the convergence properties of the domain decomposition algorithm is beyond the scope of this work.

%\textcolor{red}{
\begin{remark}
    For nonlinear PDEs such as the one considered here, each nonlinear sub-problem within the Schwarz algorithm requires an initial guess for the solution $\mathbf{u}_i^{(k)}$ at each Schwarz iteration $k$. 
    While performing numerical studies on the 2D nonlinear elasticity problem, we discovered that the Newton solver performance can be improved by: (i) solving a linear variant of the problem on-the-fly at each Schwarz iteration as a pre-processing step, and (ii) using the solutions to this linear problem as initial guesses for the Newton algorithm.  %be stabilized by solving a linear elastic version of the same problem with the same interface condition: this initial guess calculation is done as a simple pre-processing step before applying the Newton's solver to the targeted nonlinear problem.  
    %Empirically, we observe that using the solution of a linear elasticity solver with the same external boundary conditions as those described above and with interface boundary conditions induced by the previous iteration $\mathbf{u}_i^{(k-1)}$ is convenient to compute a good initial guess at iteration $k$ and to limit instabilities in the local Newton's solvers.
% This Schwarz initialization approach was utilized herein.  
\end{remark}

We define the following error metric to assess the accuracy of the DD methods %in each sub-domain $i=1, \ldots, N_{\rm{dd}}$
\begin{equation} \label{eq:max_norm_error}
e_{max}(i):=\max \left(\sqrt{(u_{i,x}-{u_x}|_{\Omega_i})^2+(u_{i,y}-{u_y}|_{\Omega_i})^2}\right). 
\end{equation}
Here, $\mathbf{u}_i:\Omega_i\rightarrow \mathbb{R}^2$ denotes the numerical sub-solution in $\Omega_i$, where $\mathbf{u}_i=[u_{i,x}, u_{i,y}]^T$, with $u_{i,x}$ and $u_{i,y}$ denoting the horizontal and vertical components of the displacement, respectively. We denote by $\mathbf{u}|_{\Omega_i}: \Omega_i \to \mathbb{R}^2$ the restriction of the global solution to sub-domain $\Omega_i$, where %---computed on $\Omega$ to the $i^{th}$ sub-domain, i.e. 
$\mathbf{u}|_{\Omega_i} = [u_x, u_y]^T=\mathbf{u}(\mathbf{x}) \; \forall \mathbf{x} \in \Omega_i$. %We note that the grid on $\Omega$ is conformal with the sub-domain meshes $\Omega_1$ and $\Omega_2$, making \eqref{eq:max_norm_error} straightforward to compute.  
While it is common to study the accuracy and convergence of numerical methods using relative, rather than absolute errors, the displacement solution to the 2D nonlinear elasticity problem is within the range $[0,1]$, so that an absolute error is a reasonable metric to utilize.  
%\irina{}{
%\begin{equation} \label{eq:max_norm_error}
%e_{max}(i):=\max \left(\sqrt{((\mathbf{u}_i)_x-{\mathbf{u}_x}|_{\Omega_i})^2+((\mathbf{u}_i)_y-{\mathbf{u}_y}|_{\Omega_i})^2}\right).
%\end{equation}
%where $\mathbf{u}_i:\Omega_i\rightarrow \mathbb{R}^2$ denote the numerical sub-solutions at $\Omega_i$ with $\mathbf{u}_i=[u_x, u_y]^T$, with $u_x$ and $u_y$ the horizontal and vertical components of the displacement, respectively. We denote by $\mathbf{u}|_{\Omega_i}: \Omega_i \to \mathbb{R}^2$ the restriction of the global solution---computed on $\Omega$ to the $i^{th}$ sub-domain, i.e. $\mathbf{u}|_{\Omega_i}=\mathbf{u}(\mathbf{x}) \; \forall \mathbf{x} \in \Omega_i$. We note that the grid on $\Omega$ is conformal with the sub-domain meshes $\Omega_1$ and $\Omega_2$.}

\subsubsection{Schwarz with classical relaxation}
We begin by studying the convergence of the non-overlapping Schwarz alternating method accelerated using the classical relaxation scheme (Algorithm \ref{alg:DN_relax}).  
In Figure \ref{fig:neohook_relax}, we depict the absolute and relative errors with respect to the Schwarz iteration $k$ and for different values of the relaxation parameter $\rho \in \{0.1,0.2,0.3,0.4,0.5\}$.  For $\rho > 0.5$, the algorithm did not converge. The reader can observe by examining Figure \ref{fig:neohook_relax} that an intermediate value of $\rho = 0.4$ gives rise to convergence in the fewest number of Schwarz iterations. 
\begin{figure}[h]
    \centering
    \includegraphics[width=0.4\linewidth]{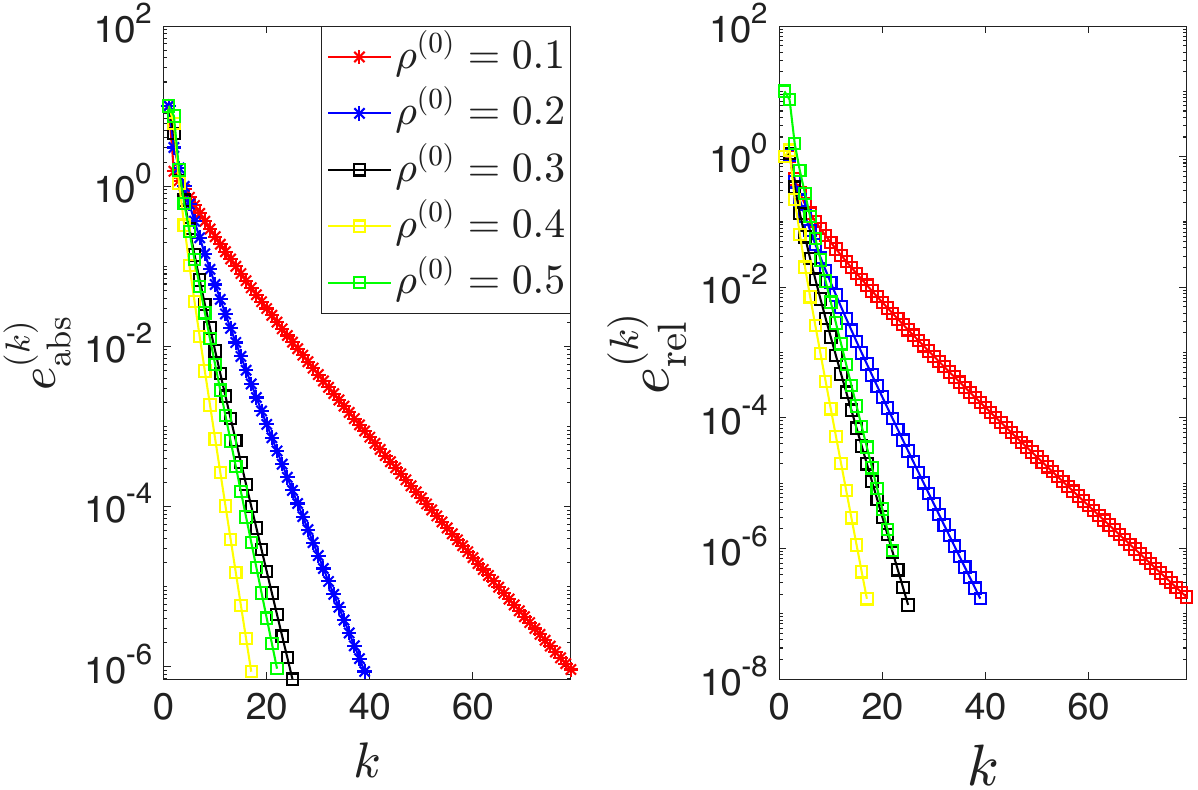}
    \caption{2D nonlinear elasticity with $N_{\text{dd}} = 2$.  Convergence of the Dirichlet-Neumann Schwarz alternating method with classical relaxation for different values of $\rho$. 
    The left panel is showing convergence of $e_{\text{abs}}^{(k)}$, whereas the right panel shows convergence of $e_{\text{rel}}^{(k)}$.
    }
    \label{fig:neohook_relax}
\end{figure}
\begin{comment}
\begin{table}[H]
\centering
\begin{tabular}{|c|c|c|c|}
\hline
$\rho$& $e_{max}(1)$& $e_{max}(2)$ &  $k$\\
\hline
$0.1$&  $1.23 \times 10^{-6}$&  $6.27 \times 10^{-7}$&$79$\\
$0.2$& $5.06 \times 10^{-7}$& $2.58 \times 
10^{-7}$&$39$\\
$0.3$&  $2.25 \times 10^{-7}$& $1.15\times 10^{-7}$&$25$\\
$0.4$&  $1.60 \times 10^{-7}$& $8.20 \times 10^{-8}$& $17$\\
$0.5$& $5.78 \times 10^{-8}$& $1.14 \times 10^{-7}$& $22$\\
\hline
\end{tabular}
\caption{Error values \eqref{eq:max_norm_error} for different values of $\rho$ with classical relaxation.}
\label{tab:relax}
\end{table}
\end{comment}

\subsubsection{Schwarz with Aitken acceleration}
% Aitken
Next, we study convergence of the non-overlapping Schwarz alternating method with Aitken acceleration.  
Figures \ref{fig:neohook_aitken}(a) and (b) show the Schwarz absolute and relative errors as a function of the Schwarz iteration $k$ for different values of  %with respect to Schwarz iterations $k$ are depicted in the case of Aitken acceleration for different values of the initial parameter 
$\rho^{(1)}$. We  only depict results for $\rho^{(1)} \in \{0.1,0.2,0.3,0.4,0.5\}$ since the algorithm did not converge for for $\rho^{(1)}>0.5$.
We show results for two different values of the hyperparameter $N_0$: $N_0 = 10$ (Figure \ref{fig:neohook_aitken}(c)) and $N_0 = 2$ (Figure \ref{fig:neohook_aitken}(d)).  We remind the reader that, to the best of our knowledge, the introduction of the $N_0$ parameter into the Aitken scheme is new to this work.
\begin{figure}[h]
        \centering
        \subfloat[$N_0=10$]{
        \includegraphics[width=0.4\linewidth]{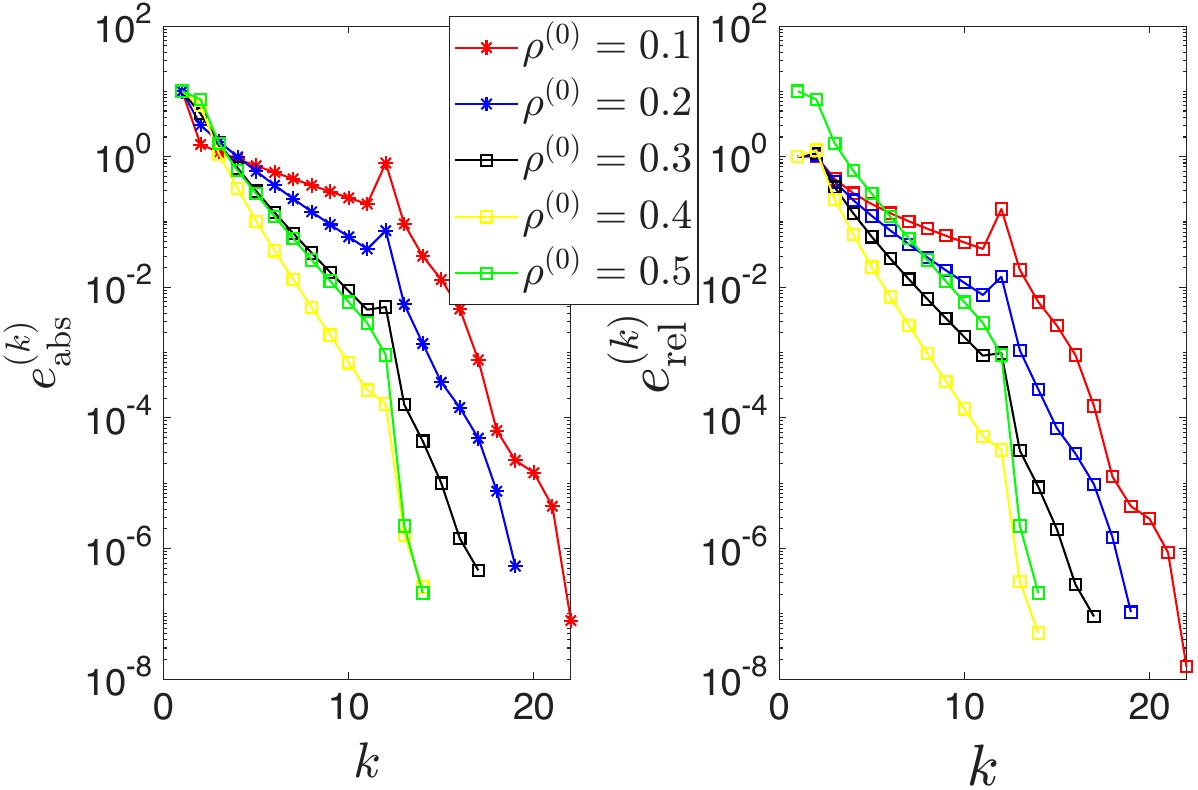}
         }
         \subfloat[$N_0=2$]{
          \includegraphics[width=0.4\linewidth]{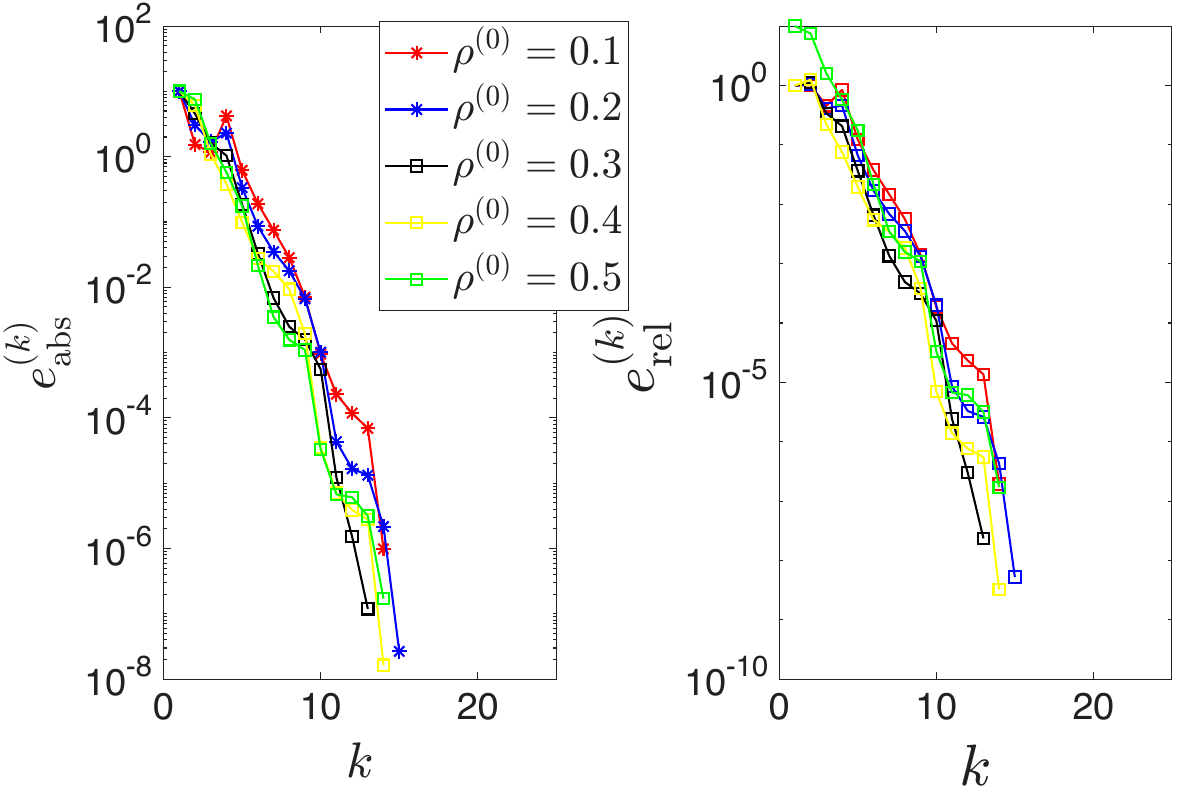}
         }
        \quad 
          \subfloat[$N_0=10$]{
       \includegraphics[width=0.32\linewidth]{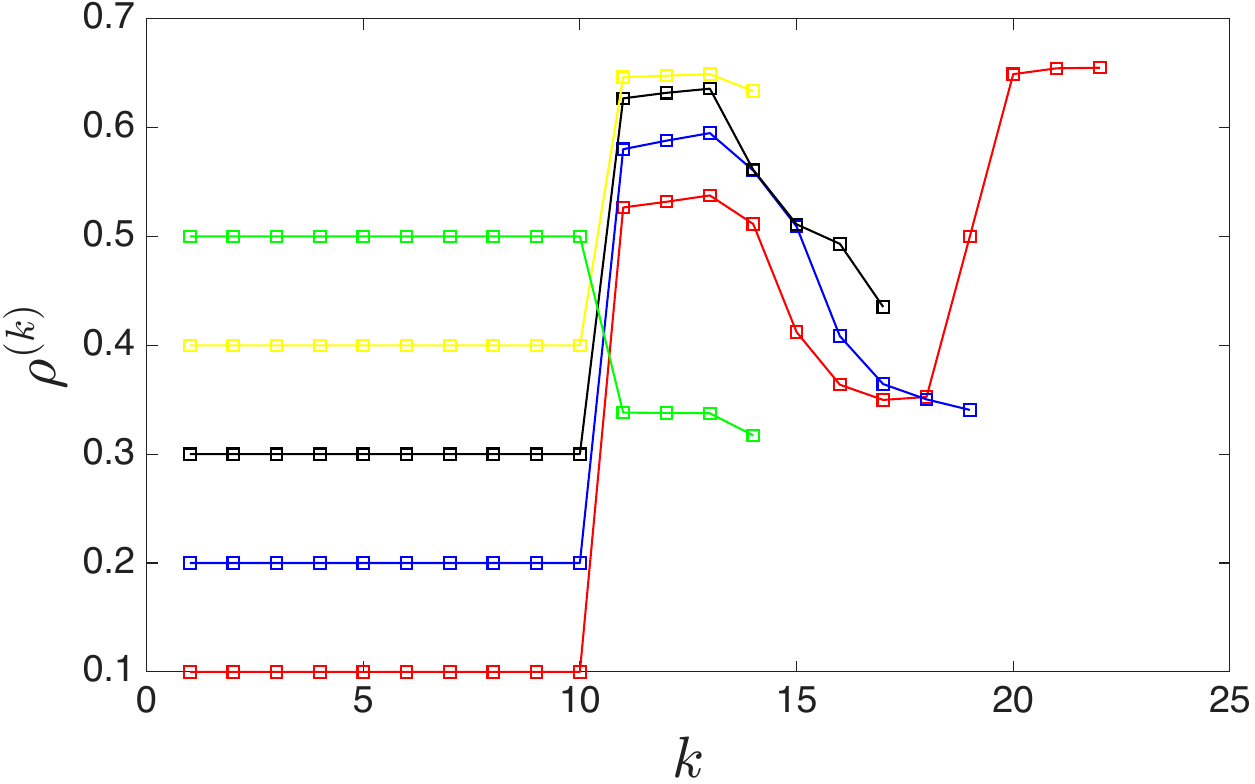}
       }
       \subfloat[$N_0=2$]{
       \includegraphics[width=0.32\linewidth]{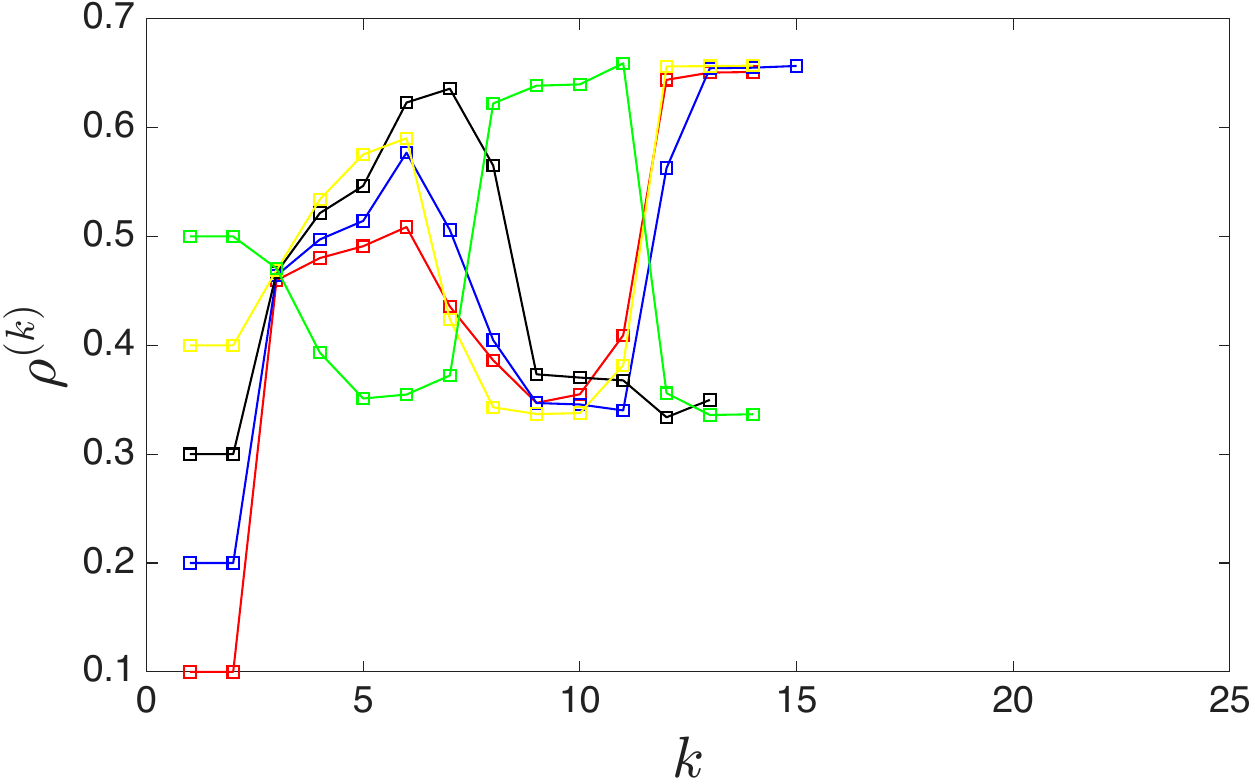}
       }
         \caption{2D nonlinear elasticity with $N_{\text{dd}} = 2$.  Sub-plots (a) and (b) depict the convergence of Aitken-accelerated Schwarz for five different choices of $\rho^{(1)}$ and two choices of $N_0$. Sub-plots (c) and (d) depict the values of the Aitken parameter $\rho^{(k)}$ as a function of the Schwarz iteration $k$, again for two choice of $N_0$. The parameter $N_0$  is set to $10$ in subfigures (a) and (c), and to $2$ in subfigures (b) and (d). }
          \label{fig:neohook_aitken}
    \end{figure}
Comparing these figures with Figures \ref{fig:neohook_relax}(a) and (b), we conclude that, as expected, the Schwarz algorithm with a dynamical update of the relaxation parameter converges in a smaller number of iterations.  We additionally observe that Aitken-accelerated Schwarz appears to be stable and robust with respect to different choices of the initial relaxation parameter $\rho^{(1)}$.  Figure \ref{fig:neohook_aitken} suggests that, for the 2D nonlinear elasticity problem with $N_{\text{dd}} = 2$, selecting a smaller value of $N_0$  results in a smaller number of Schwarz iterations for a fixed value of $\rho^{(1)}$, though the algorithm is not particularly sensitive to the value of this parameter.  Curiously, when we tried the smallest possible value, $N_0=1$, the method diverged. 

%\irinanote{I think $N_0 = 1$ doesn't make sense - need to discuss with Giulia.}%By setting $N_0$ to a possibly low value, one can take the $N_0$ tuning parameter out of the equation for the Aitken acceleration scheme.  This adds to the appeal of the method, as, without $N_0$, it contains only one tuning parameter, namely $\rho^{(1)}$.  
%As shown in Figures \ref{fig:neohook_relax}(a) and (b), the algorithm is not particularly sensitive to this parameter.  

Finally, in Figures \ref{fig:neohook_aitken}(c)--(d), we depict the values  of $\rho^{(k)}$ as a function of the Schwarz iteration $k$ for $N_0 = 10$ and $N_0 = 2$.
For $\rho^{(1)} < 0.5$, the optimal values of $\rho^{(k)}$ follow similar trends for $k > N_0$ and are generally larger than $\rho^{(1)}$.  
 %Additionally, one can see from Figure \ref{fig:neohook_aitken}(c) and (d) that the optimal values of $\rho^{(k)}$ for the Aitken acceleration method are within the range $(0,1]$, even though this is not guaranteed from Aitken formula \eqref{eq:aitken}.  \irinanote{But we are doing what is in Remark \ref{remark:aitken_bounds}, no?}

\subsubsection{Schwarz with classical Anderson acceleration}

We now apply Anderson-accelerated Schwarz (Algorithm \ref{alg:DN_anderson}) to our 2D nonlinear elasticity problem with $N_{\text{dd}}=2$.  Note that, in this subsection, we are \textit{not} considering the Anderson with memory adaptation approach described in Section \ref{sec:anderson_memory}. We use the same algorithmic settings as in the previous sections.  For the solution of the constrained minimization problem \eqref{eq:anderson_constrained}, we opt for the quadratic programming function $\texttt{quadrprog}$ in {\tt MATLAB}. 
We consider both the unrelaxed and relaxed versions of Anderson acceleration. \\
In Figure \ref{fig:neohook_anderson_rho05}, we depict the $\alpha_i^{(k)}$ coefficients calculated in the Anderson optimization procedure, which solves the constrained minimization problem \eqref{eq:anderson_constrained}. For a given $m_{\rm{and}}$, the asymptotic value of $\alpha_i^{(k)}$ is  $\approx \frac{1}{m_{\rm{and}}+1}$ for all $i=1, \ldots, m_{\rm{and}}+1$, a result consistent with Theorem \ref{theo:anderson_coeff}.  
We further notice that, while the $\alpha_i^{(k)}$ coefficients are not guaranteed to be between $0$ and $1$ (cf. Remark \ref{remark:anderson_sign} ), 
they generally stay within this range.
%for $m_{\rm{and}}=2$  but mostly seem to converge in this range nonetheless.
\begin{figure}[h]
\centering
\subfloat[$m_{\rm{and}}=1$]{
     \includegraphics[scale=0.32]{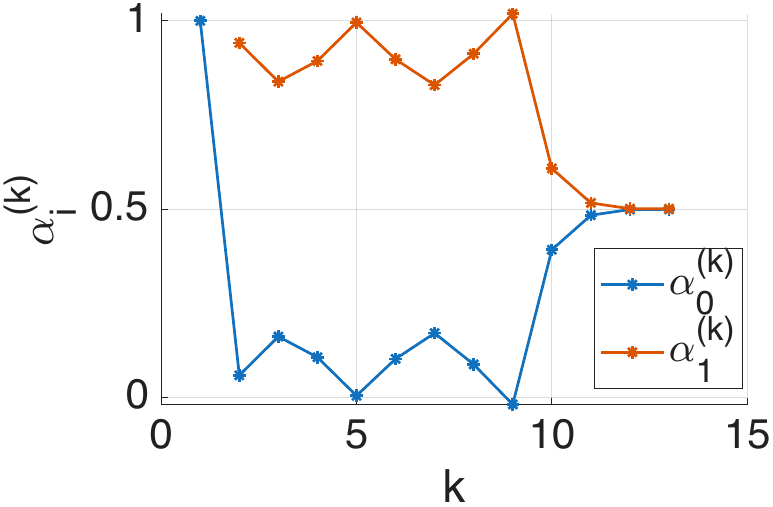}
     }
     \subfloat[$m_{\rm{and}}=2$]{
      \includegraphics[scale=0.32]{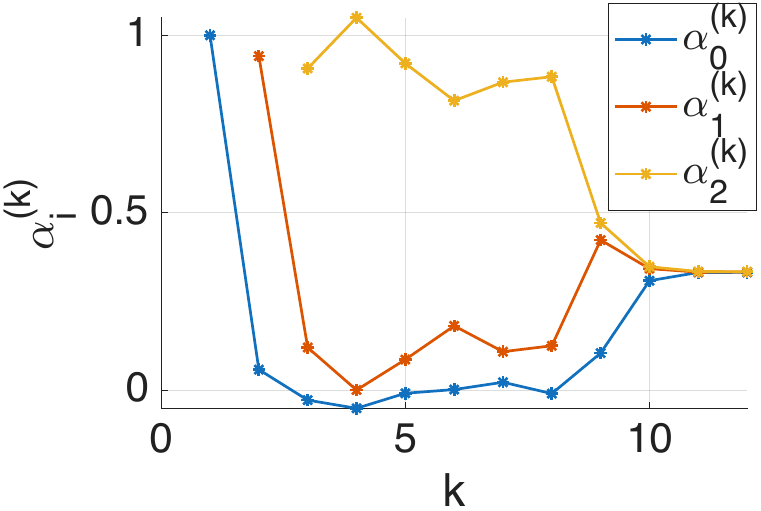}
     }
     \subfloat[$m_{\rm{and}}=3$]{
      \includegraphics[scale=0.32]{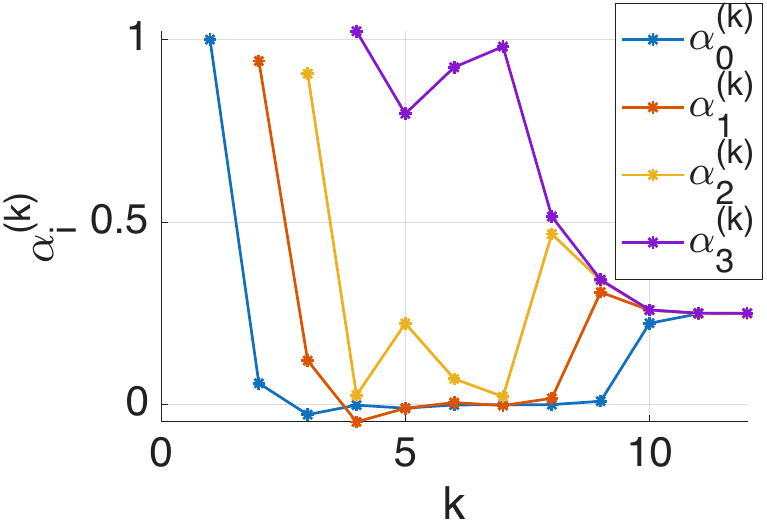}
     }
    \caption{2D nonlinear elasticity with $N_{\text{dd}} = 2$.  Anderson coefficients $\alpha_i^{(k)}$ for $\rho = 0.5$ and increasing values of $m_{\text{and}}\in \{1,2,3\}$.  }
    \label{fig:neohook_anderson_rho05}
\end{figure}

In Figures  \ref{fig:neohook_anderson_rho}(a) and (b), the norm of the residual $||\mathbf{f}_k||_2$ is depicted as a function of the Schwarz iteration number $k$. Unlike Schwarz with classical relaxation and Aitken-accelerated Schwarz, the reader can observe that Anderson-accelerated Schwarz is convergent 
for $\rho > 0.5$, including the $\rho =1$ (unrelaxed) case;
%also in the case of $\rho=1$ (i.e., no relaxation) for different choices of $m_{\rm{and}}$ as well as for $\rho=0.7$  (c.f., Figure \ref{fig:neohook_anderson_rho}(b)); 
however, there is some sensitivity to the $m_{\rm{and}}=1$ parameter: in the case of $m_{\rm{and}}=1$ and larger values of $\rho$, the convergence is much slower, requiring almost the maximum amount $\texttt{maxiter}$ of iterations. We additionally observe that oscillations in the residual norm (c.f., Figure \ref{fig:neohook_anderson_rho}(a)) may appear when utilizing unrelaxed Anderson-accelerated Schwarz. 

One can additionally infer from Figures  \ref{fig:neohook_anderson_rho}(a) and (b) that, for a fixed $\rho$, including more history information, i.e., using a larger value of $m_{\text{and}}$, can reduce the number of Schwarz iterations required for convergence; however this is true only up to a point.  As shown in Figure \ref{fig:neohook_anderson_rho}(c), employing substantially larger values of $m_{\rm{and}}$, e.g.,  $m_{\rm{and}}=20$ or higher while setting $m_{\rm{and}}=k$, actually makes convergence worse.  
%(c.f., Figure \ref{fig:neohook_anderson_rho}(c)) and the number of iterations required to reach convergence approaches $\texttt{maxit}$ for the choice $\rho=0.5$ (c.f., Figure \ref{fig:neohook_anderson_rho}(c)), empty marks curves.
This result is consistent with the literature: as remarked in \cite{walker2011anderson}, if $m_{\rm{and}}$ is small,  the secant information used by the method may be too limited to provide decidedly fast convergence, whereas if $m_{\rm{and}}$ is too large, the least-squares optimization problem for the $\mathbf{\alpha}$ coefficients \eqref{eq:anderson_constrained}  may be poorly conditioned. 
%The method does not break down upon stagnation before the solution has been found; one should expect near-stagnation at some step to result in  ill-conditioning of the least-squares problem \eqref{eq:anderson_constrained} (although we did not observe this in this numerical tests).  
While we did not observe stagnation in the Schwarz method's convergence %in our numerical tests in Figure   
Figure \ref{fig:neohook_anderson_rho} clearly shows that  values of   $m_{\text{and}}$ on the order of 20 or greater lead to slow convergence and hence large compute times.

\subsubsection{Schwarz accelerated using Anderson acceleration with memory adaptation} \label{sec:results_anderson_memory}

The results above motivated the construction of the ``Anderson with memory adaptation'' method described 
earlier in Section \ref{sec:anderson_memory}.  
%In order to remedy the problem described above, namely the fact that Anderson-accelerated Schwarz can be quite sensitive to $m_{\text{and}}$, 
%we propose an adaptive variant of the method, which we term ``Anderson with memory adaptation''.  
%To remedy this problem in the case of large history windows--- for this example, in the case of $m_{\rm{and}}>3$--- we propose an adaptive version of Anderson. 
%In this version of the algorithm, the value of $m_k$ is adapted on-the-fly using the following formula:  
%\begin{equation}
%m_k:= \begin{cases}
%    min(k, m_{\rm{and}}) \; \text{for} \,  k \leq \bar{k} \\
%    \bar{m} \ll \bar{k} \;  \text{for} \,  k > \bar{k}, \\
%\end{cases}
%\label{eq:anderson_with_memory_criterion}
%\end{equation}
%where $\bar{k}$ is a pre-selected parameter.  
%In our implementation, $\bar{k}$ is the first Schwarz iteration $k$ at which the difference   %on the interface errors at iteration $k$ and at iteration $k-1$
We set a threshold of
$\epsilon_{\rm{and}}=10^{-5}$ for the interface error $|e_{\rm{rel}}^{(k)}-e_{\rm{rel}}^{(k-1)}|$ and $\bar{m}=3$.
In Figure \ref{fig:fig:neohook_anderson_with_memory_rho}, we compare the convergence history of Anderson-accelerated Schwarz both with and without memory adaptation, with filled symbols denoting the classical Anderson method with a fixed $m_{\text{and}}$ and empty symbols denoting Anderson with memory adaptation. %We compare Anderson convergence history with and without memory adaptation. 
%The standard choice is represented by empty marks, while Anderson with memory adaptation is represented by solid marks. 
%For the set tolerance $\epsilon_{\text{and}}$, we found a value of $\bar{k}$ which is around $20$ and we set $\bar{m}=3$.
In the sub-panel in Figure \ref{fig:fig:neohook_anderson_with_memory_rho}, we depict also the $m_k$ values utilized by the Anderson with memory adaptation method, calculated according to the chosen criterion \eqref{eq:anderson_with_memory_criterion}. 
%\giulianote{We numerically observed that a regularization of the objective function of \eqref{eq:anderson_constrained} fails to smooth out these oscillations. We need to gain a better understanding of this aspect.}
\begin{figure}[h]
\centering
\subfloat[Unrelaxed Anderson]{
\includegraphics[width=0.38\linewidth]{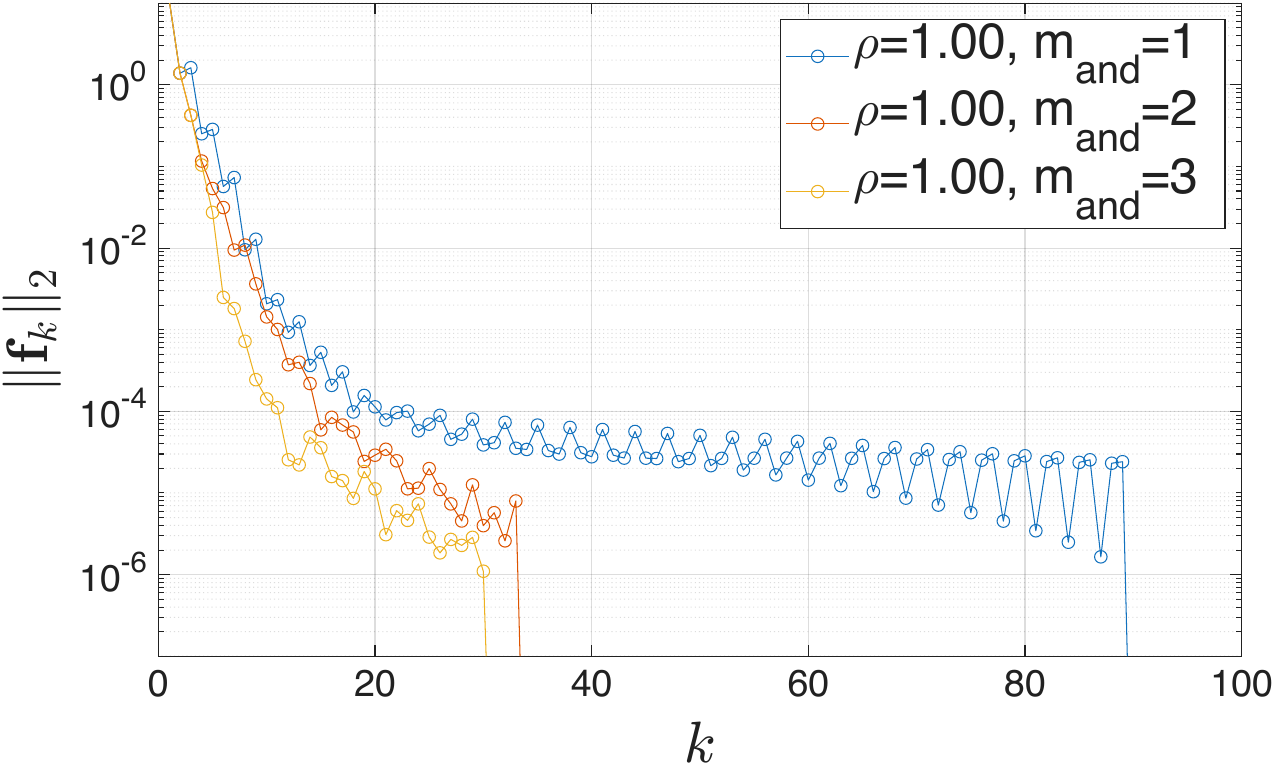}
}
\subfloat[Relaxed Anderson]{
     \includegraphics[width=0.38\linewidth]{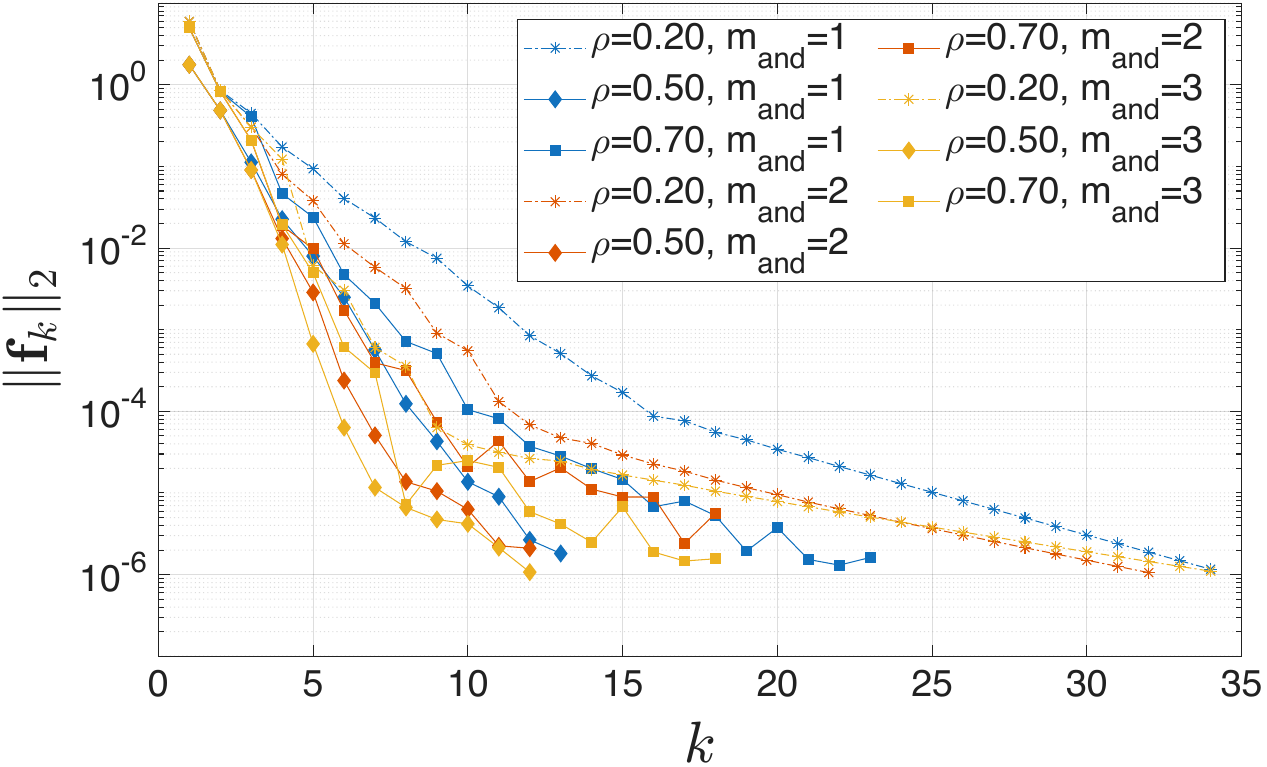}
     }
    \caption{2D nonlinear elasticity with $N_{\text{dd}} = 2$.  Convergence of Anderson-accelerated Schwarz (without memory adaptation) for different values of  $\rho$ and  $m_{\rm{and}}$.}
    \label{fig:neohook_anderson_rho}
\end{figure}
\begin{figure}[h]
\centering
     \includegraphics[scale=0.32]{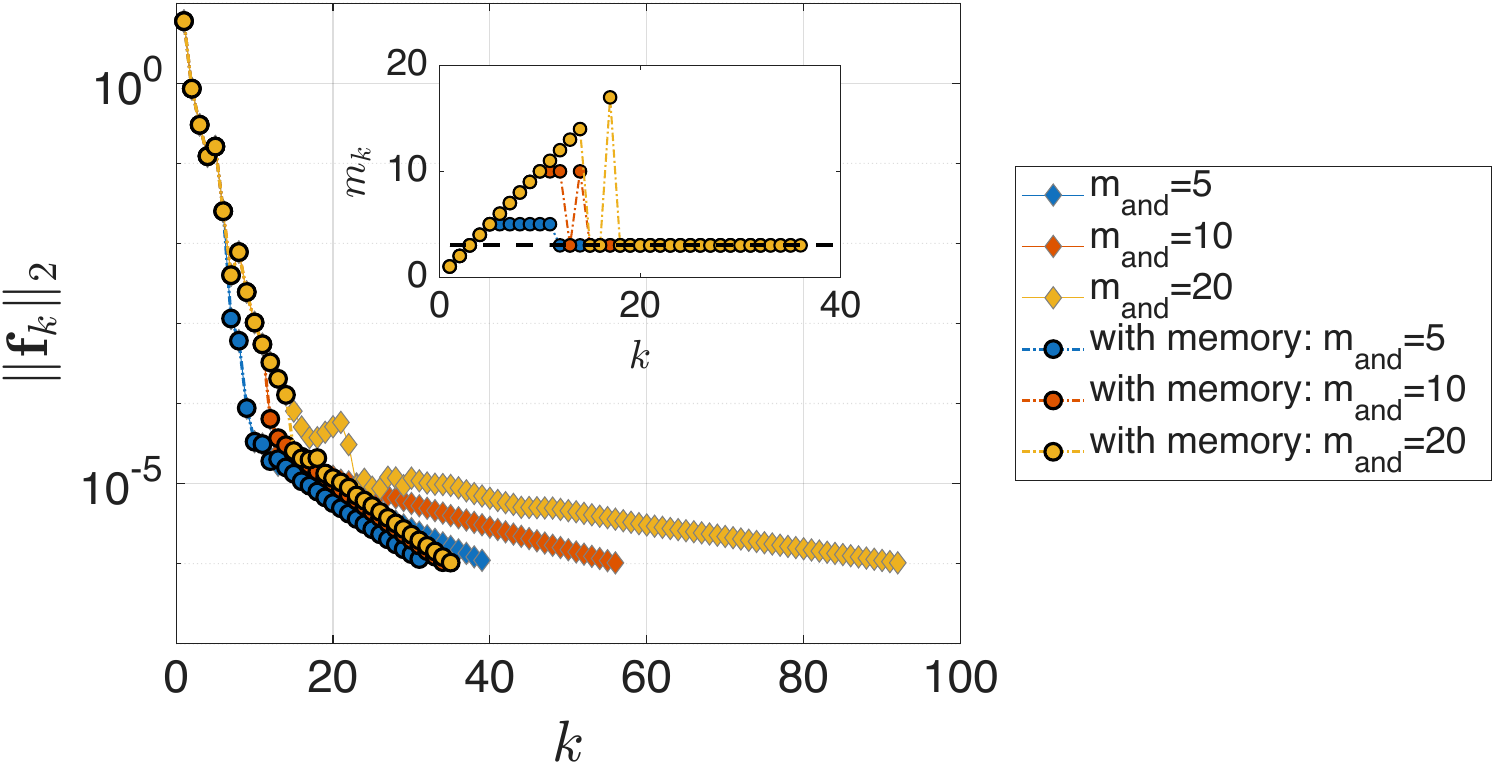}
    \caption{2D nonlinear elasticity with $N_{\text{dd}} = 2$.  Anderson with memory adaptation compared to the standard non-adapted Anderson with $\rho=0.2$. In the sub-plot, the values of $m_k$ chosen by the Anderson with memory technique are reported. The threshold in black represents $\bar{m}=3$ for the condition in \eqref{eq:anderson_with_memory_criterion}.}
    \label{fig:fig:neohook_anderson_with_memory_rho}
\end{figure}

\subsubsection{Methods comparison}
\label{sec:methods_comparison}

In Table \ref{tab:methods_comparison_rho02}, we compare Anderson-accelerated Schwarz with no memory adaptation %\irinanote{check with Giulia} 
for $m_{\rm{and}}=1,2,3$ with the same algorithm accelerated using classical relaxation and Aitken acceleration in terms of errors with respect to the global solution for $\rho = \rho^{(1)} = 0.2$.  Table \ref{tab:methods_comparison_rho05}
 shows analogous results for the case $\rho  = \rho^{(1)}=0.5$. 
The following conclusions from Tables \ref{tab:methods_comparison_rho02}--\ref{tab:methods_comparison_rho05} are noteworthy: (i) Aitken-accelerated Schwarz is far less sensitive to $\rho^{(1)}$ than classically relaxed and Anderson-accelerated Schwarz are to $\rho$, (ii) the fastest convergence is achieved for Anderson-accelerated Schwarz with $\rho = 0.5$ and $m_{\text{and}} \in \{2, 3\}$ (Table \ref{tab:methods_comparison_rho05}), and (iii) while Aitken-accelerated Schwarz takes 1--2 more Schwarz iterations to converge than the ``optimal'' version of Anderson-accelerated Schwarz, it achieves a smaller value of the error $e_{max}(i)$ upon convergence.  
 
\begin{table}[h]
\centering
\resizebox{0.88\textwidth}{!}{%
\begin{tabular}{|l|c|c|ccc|}
\hline
\textbf{} & Classical relaxation& Aitken acceleration& \multicolumn{3}{c|}{Anderson acceleration} \\
\cline{4-6}
          &                  &                   $(N_0 =2)$& $m_{\text{and}}=1$& $m_{\text{and}}=2$& $m_{\text{and}}=3$\\
\hline
$e_{max}(1)$& $5.06 \times 10^{-7}$& $1.23 \times 10^{-9}$& $1.49 \times 10^{-7}$& $2.24 \times 10^{-7}$& $3.96 \times 10^{-7}$\\
$e_{max}(2)$& $2.58 \times 10^{-7}$& $1.94 \times 10^{-9}$& $7.61\times 10^{-8}$& $1.15 \times 10^{-7}$& $2.04 \times 10^{-7}$\\
$k$& $39$& $15$& $35$&$33$& $35$\\
\hline
\end{tabular}
}
\caption{2D nonlinear elasticity with $N_{\text{dd}} = 2$.  Accuracy and convergence comparisons between non-overlapping Schwarz with classical relaxation, and with the two  acceleration methods evaluated.  All methods evaluated used  $\rho = \rho^{(1)}=0.2$.}
\label{tab:methods_comparison_rho02}
\end{table}
In Figure \ref{fig:meth_comparison}, we compare the relaxation/acceleration methods described so far in terms of relative error until convergence (or until the termination criterion on the maximum number of iterations). %; in \ref{fig:meth_comparison}(a) we set $\rho^{(1)}=0.2$, in \ref{fig:meth_comparison}(b) we set $\rho^{(1)}=0.5$.\\
\begin{table}[h]
\centering
\resizebox{0.88\textwidth}{!}{
\begin{tabular}{|l|c|c|ccc|}
\hline
\textbf{} & Classical relaxation& Aitken acceleration& \multicolumn{3}{c|}{Anderson acceleration} \\
\cline{4-6}
          &                  &                   $(N_0 =2)$& $m_{\text{and}}=1$& $m_{\text{and}}=2$& $m_{\text{and}}=3$\\
\hline
$e_{max}(1)$& $5.78 \times 10^{-7}$& $1.21 \times 10^{-9}$& $1.38 \times 10^{-7}$& $2.08 \times 10^{-7}$& $2.34 \times 10^{-7}$\\
$e_{max}(2)$& $1.14 \times 10^{-7}$& $1.14 \times 10^{-9}$& $2.12 \times 10^{-7}$& $2.98 \times 10^{-7}$& $1.83 \times 10^{-7}$\\
$k$& $22$& $14$& $14$&$13$& $13$\\
\hline
\end{tabular}
}
\caption{2D nonlinear elasticity with $N_{\text{dd}} = 2$.  Accuracy and convergence comparisons between Schwarz with classical relaxation, and with the two  acceleration methods evaluated.  All methods evaluated used  $\rho = \rho^{(1)}=0.5$. } 
\label{tab:methods_comparison_rho05}
\end{table}
In Figure \ref{fig:meth_comparison}, we show the Schwarz convergence error in the case of different relaxation/acceleration techniques with $\rho = \rho^{(1)} = 0.2$ (Figure \ref{fig:meth_comparison}(a)) and $\rho = \rho^{(1)} = 0.5$ (Figure \ref{fig:meth_comparison}(b)). We do not report the unrelaxed Schwarz scheme for comparison, because, in this case, the algorithm does not converge. 
%This is not surprising, since convergence is generally not guaranteed for the unrelaxed Dirichlet-Neumann Schwarz algorithm, as discussed earlier in Section \ref{sec:classic_DirNeu}; indeed, this is the main motivation for developing the relaxation and acceleration techniques considered in this paper.  
\begin{figure}[tbhp]
    \centering
    \subfloat[$\rho = \rho^{(1)}=0.2$]{
    \includegraphics[width=0.5\linewidth]{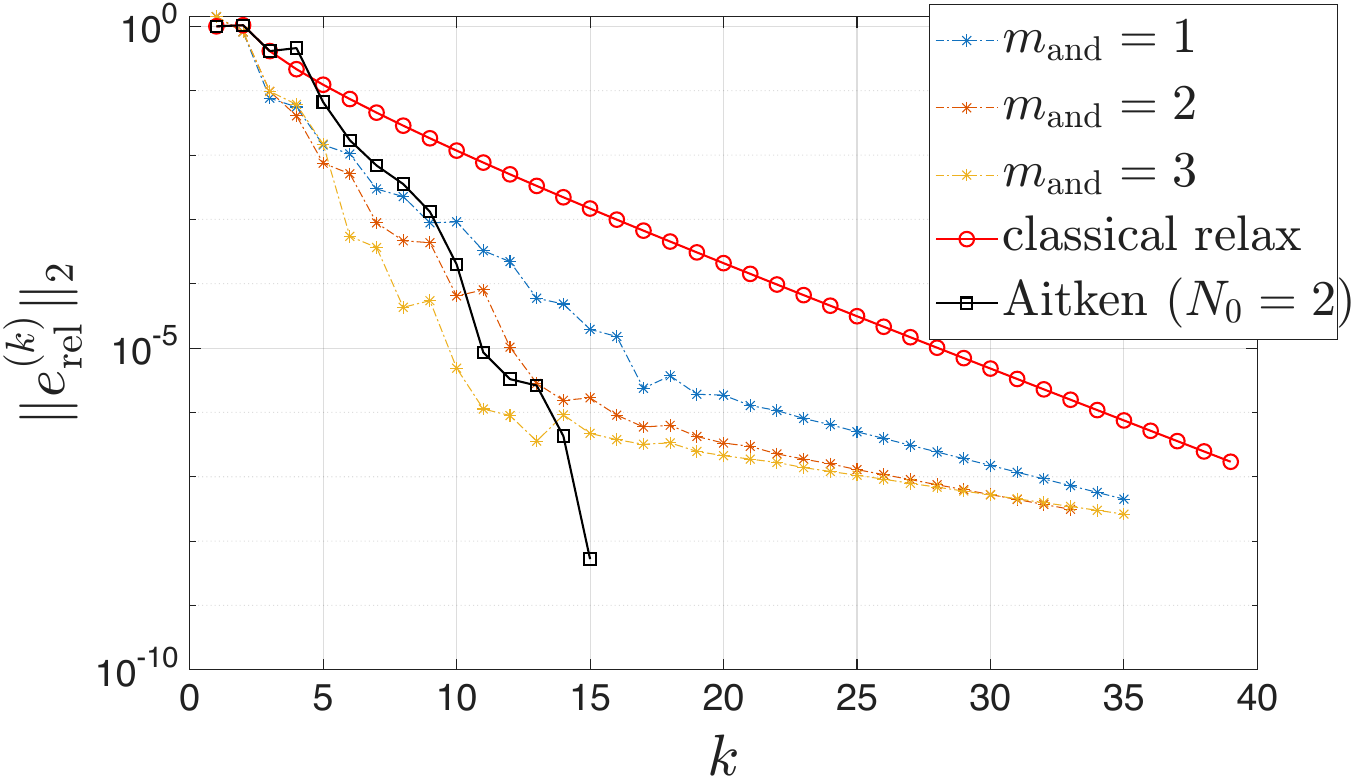}}
    \subfloat[$\rho = \rho^{(1)}=0.5$]{
    \includegraphics[width=0.5\linewidth]{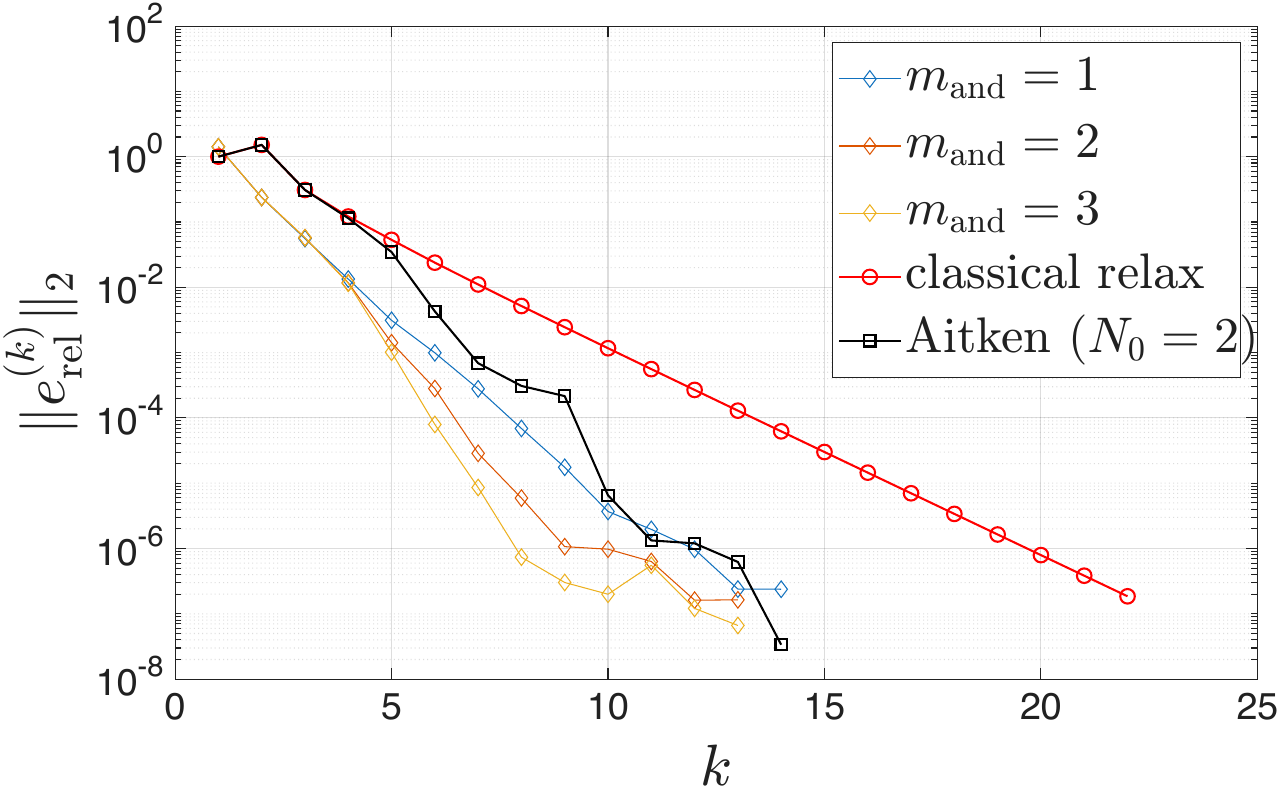}}
    \caption{2D nonlinear elasticity with $N_{\text{dd}} = 2$.  Convergence of Schwarz with classical relaxation, Aitken acceleration and Anderson acceleration,  with (a) $\rho = \rho^{(1)}=0.2$, and (b) $\rho = \rho^{(1)}=0.5$.}
    \label{fig:meth_comparison}
\end{figure}
The plots in Figure \ref{fig:meth_comparison} show that Aitken- and Anderson-accelerated Schwarz is more robust to the choice of $\rho = \rho^{(1)}$ than Schwarz with classical relaxation. This result is not surprising for the Aitken-accelerated version of the method, as the Aitken acceleration approach determines on-the-fly the optimal value of $\rho^{(k)}$ for $k > N_0$, rather than setting it to an \textit{a priori}-specified constant value, as in classical relaxation and Anderson acceleration.
The reader can observe that Anderson-accelerated Schwarz is not particularly sensitive to   $m_{\text{and}}$.  
%A possible implementation of an adaptive procedure to enhance the convergence properties of Anderson is beyond the scope of this work. \irinanote{What sort of adaptive procedure do you have in mind?  It might be best to cut this statement unless you want to explain more what you mean.}

An additional property which can be inferred from Figure \ref{fig:meth_comparison} is the convergence rate.  We numerically computed the convergence rate for the case $\rho=\rho^{(1)}=0.5$ by calculating $C_k=\frac{|e^{(k+1)}_{\rm{rel}}|}{|e^{k}_{\rm{rel}}|^2}$, excluding the pre-asymptotic regime and the last iteration.  Whereas the Schwarz alternating method with classical relaxation is converging at a linear or sub-linear rate  ($C_k \in [0.3,1]$), both Aitken- and Anderson-accelerated Schwarz converge at a rate that is almost quadratic ($C_k \in [0.5, 1.4]$). %Indeed, we computed $C_k=\frac{e^{(k+1)}_{rel}}{(e^{k}_{rel})^2}$ and excluding the pre-asymptotic and the last iterations we obtain $C_k \in [0.3,1] $ in the case of Aitken and $C_k \in [0.5, 1.4]$ in the case of Anderson.
%This result is consistent with our theoretical analysis of the scheme given in Section \ref{sec:aitken_conv_rate}. We cannot say that because the theory is only for d=1
Additional comments on the computational costs associated with all the techniques evaluated are postponed to the next section, where we consider the more general case of $N_{\rm{dd}}\ge2$.
%\irinanote{To complete the comparison, I think you should create Pareto plots that plot the CPU time for all the methods vs. the error reported in the Tables in this subsection.  I am conjecturing that Anderson acceleration will be more expensive.  Then, we can add some discussion saying that, while it's possible to get Anderson acceleration to converge in less or the same number of Schwarz iterations as Aitken, it will be slower, and is hence not preferred.  I would also like to add some commentary about Anderson being less preferred due to having more tuning knobs and requiring saving of more history information.  I can add this discussion once you add the Pareto plots.}

\subsubsection{A multi-domain study}
\label{par:multi_domain}
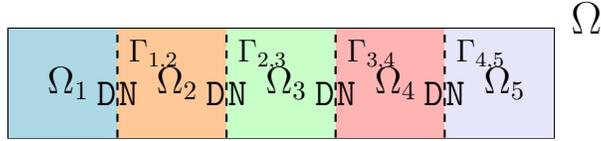
\begin{figure}[h]
    \centering
\begin{tikzpicture}[scale=1.45]
% Global domain
\draw[thick] (0,0) rectangle (5,1);
\node[font=\Large] at (5.3,1.1) {$\Omega$};

% sub-domain colors
\definecolor{col1}{RGB}{173,216,230} % light blue
\definecolor{col2}{RGB}{255,200,150} % light orange
\definecolor{col3}{RGB}{200,255,200} % light green
\definecolor{col4}{RGB}{255,180,180} % light red/pink
\definecolor{col5}{RGB}{230,230,250} % light purple

% sub-domain boundaries and widths
\def\x{0,1,2,3,4,5} % x positions of boundaries

% Draw sub-domains
\fill[col1] (0,0) rectangle (1,1);
\fill[col2] (1,0) rectangle (2,1);
\fill[col3] (2,0) rectangle (3,1);
\fill[col4] (3,0) rectangle (4,1);
\fill[col5] (4,0) rectangle (5,1);

% Draw interfaces
\foreach \i in {1,2,3,4} {
    \pgfmathtruncatemacro{\j}{\i+1} % compute i+1 and store in \j
    \draw[thick, dashed] (\i,0) -- (\i,1) 
        node[below right] {$\Gamma_{\i,\j}$};
}

% Labels for sub-domains
\node[font=\large] at (0.55,0.5) {$\Omega_1$};
\node[font=\large] at (1.55,0.5) {$\Omega_2$};
\node[font=\large] at (2.55,0.5) {$\Omega_3$};
\node[font=\large] at (3.55,0.5) {$\Omega_4$};
\node[font=\large] at (4.55,0.5) {$\Omega_5$};

% Local boundaries (thick lines)
%\draw[very thick, blue] (0,0) -- (0,1); % leftmost boundary
%\draw[very thick, red] (5,0) -- (5,1); % rightmost boundary

% Boundary labels
%\node[rotate=90, blue] at (0.2,0.5) {$\Gamma_{\text{neu}}$};
%\node[rotate=90, red] at (4.8,0.5) {$\Gamma_{\text{dir}}$};

% Interaction arrows across interfaces
%\foreach \i in {1,2,3,4} {
%\pgfmathsetmacro{\j}{int(\i+1)}
%    \draw[->] (\i-0.2,0.7) -- (\i+0.2,0.7) node[pos=1,below] {$\mathbf{n}_{\i\j}$};
    %\draw[->] (\i+0.2,0.3) -- (\i-0.2,0.3);
%}
% labels
\pgfmathsetmacro{\xend}{1}
\node[font=\large] at ({\xend-0.1}, 0.4) {$\texttt{D}$};
\foreach \i in {1,...,3}{
\pgfmathsetmacro{\xstart}{\i}
\pgfmathsetmacro{\xend}{\i+1}
\node[font=\large] at ({\xstart+0.1}, 0.4) {$\texttt{N}$};
\node[font=\large] at ({\xend-0.1}, 0.4) {$\texttt{D}$};
}
\pgfmathsetmacro{\xstart}{4}
\node[font=\large] at ({\xstart+0.1}, 0.4) {$\texttt{N}$};
\end{tikzpicture}
\caption{2D nonlinear elasticity.  Illustration of non-overlapping DD with $N_{\rm{dd}}=5$.}
\label{fig:nonovlp_multidomain}
\end{figure}

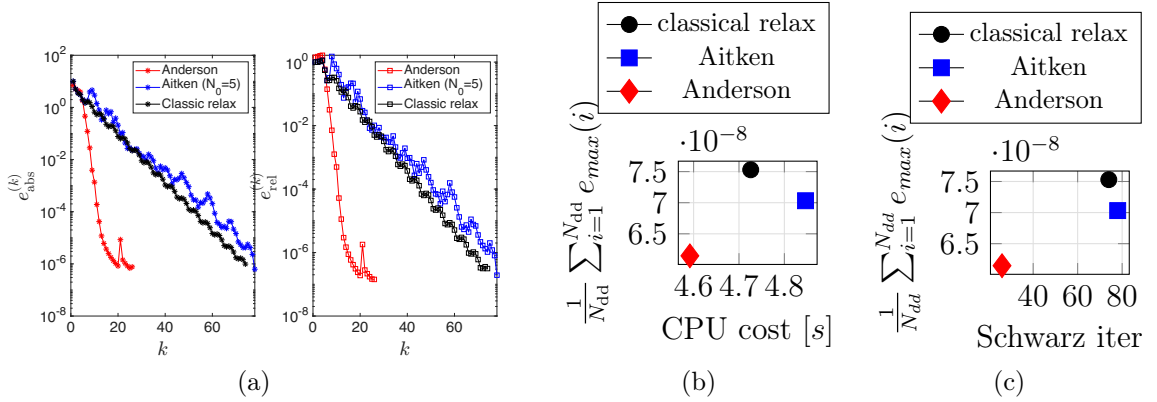
\begin{figure}[h]
\subfloat[]{
 \includegraphics[width=0.48\linewidth]{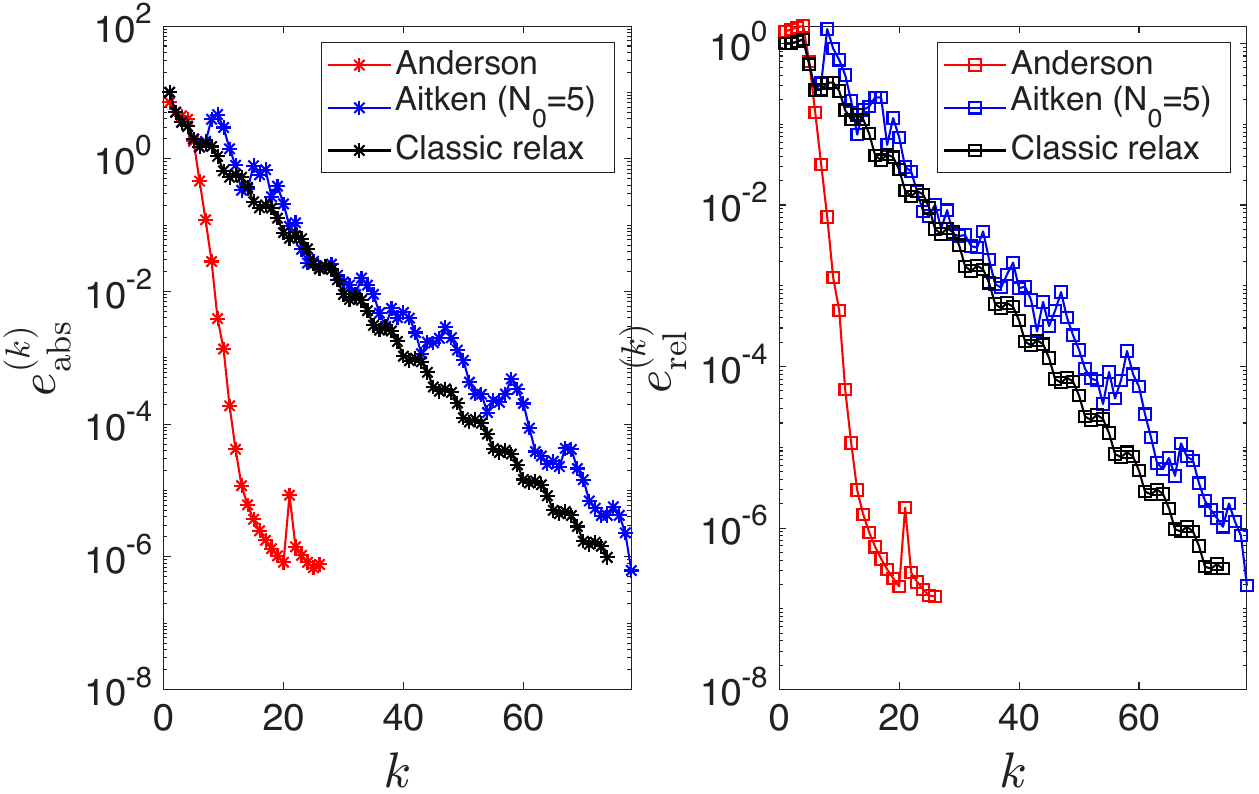}
}
\quad 
\subfloat[]{
  \begin{tikzpicture}
  	\begin{axis}[
   		xmode=linear,
   		ymode=linear,
   		grid=both,
  		minor grid style={gray!25},
   		major grid style={gray!25},
   		title = {},
   		xlabel={CPU cost $[s]$},
   		ylabel={$\frac{1}{N_{\rm{dd}}}\sum_{i=1}^{N_{\rm{dd}}} e_{max}(i)$},
   		width=0.25\textwidth,
   		legend style={at={(1.09,2)},anchor=east},
   		]
   		\addplot[ % <-- plot options
   		black,
   		mark=*,
        mark size=3,    
   		mark options={ color = black},
   		]  %
   		table{dat/cost_Ndd5_relax.dat};
   		\addlegendentry{\footnotesize classical relax};	
        \addplot[ % <-- plot options
   		black,
   		mark=square*,
        mark size=3,    
   		mark options={ color = blue},
   		]  %
   		table{dat/cost_Ndd5_aitk.dat};
   		\addlegendentry{\footnotesize Aitken};	
        \addplot[ % <-- plot options
   		black,
   		mark=diamond*,
        mark size=3,
   		mark options={ color = red,scale=1.5},
   		]  %
   		table{dat/cost_Ndd5_and.dat};
   		\addlegendentry{\footnotesize Anderson};	
   	\end{axis}
\end{tikzpicture}
}
\subfloat[]{
 \begin{tikzpicture}
  	\begin{axis}[
   		xmode=linear,
   		ymode=linear,
   		grid=both,
  		minor grid style={gray!25},
   		major grid style={gray!25},
   		title = {},
   		xlabel={Schwarz iter},
   		ylabel={$\frac{1}{N_{dd}}\sum_{i=1}^{N_{dd}} e_{max}(i)$},
   		width=0.25\textwidth,
   		legend style={at={(1.09,2)},anchor=east},
   		]
   		\addplot[ % <-- plot options
   		black,
   		mark=*,
        mark size=3,    
   		mark options={ color = black},
   		]  %
   		table{dat/iter_Ndd5_relax.dat};
   		\addlegendentry{\footnotesize classical relax};	
        \addplot[ % <-- plot options
   		black,
   		mark=square*,
        mark size=3,    
   		mark options={ color = blue},
   		]  %
   		table{dat/iter_Ndd5_aitk.dat};
   		\addlegendentry{\footnotesize Aitken};	
        \addplot[ % <-- plot options
   		black,
   		mark=diamond*,
        mark size=3,    
   		mark options={ color = red,scale=1.5},
   		]  %
   		table{dat/iter_Ndd5_and.dat};
   		\addlegendentry{\footnotesize Anderson};	
   	\end{axis}
\end{tikzpicture}
}
\caption{%(a), (b):  Methods comparison for $N_{\rm{dd}}=5$ and $\rho^{(1)}=0.5$. (a): average sub-domains error average of the per iteration CPU cost; (b): average sub-domains error vs the total number of Schwarz iteration. (c): geometric configuration. (d): convergence history in terms of absolute and relative interface errors.  
2D nonlinear elasticity with $N_{\text{dd}} = 5$.  Method comparison metrics with $\rho =0.5$.  
(a) Schwarz convergence metrics $\epsilon_{abs}^{(k)}$, and $\epsilon_{rel}^{(k)}$ as a function of the Schwarz iteration $k$. (b) Average  sub-domain error $e_{max}(i)$ as a function of the average per iteration CPU time (in seconds).  (c) Average  sub-domain error $e_{max}(i)$ as a function of the number of Schwarz iterations required to reach convergence. }
\label{fig:meth_comparison_Ndd5}
\end{figure}

In the discussion above, we have limited our attention to the Schwarz alternating method applied to two non-overlapping sub-domains.  Since the method typically requires more iterations to reach convergence when it is used 
to couple more sub-domains, there is a lot of interest in accelerating Schwarz when $N_{\rm{dd}} > 2$.  
Toward this effect, we consider here the same 2D nonlinear elasticity problem studied earlier, but with a linear-in-$x$ decomposition of our domain $\Omega$ into $N_{\text{dd}}$ sub-domains, with $N_{\text{dd}}\in \{2, ..., 5\}$.  An example
decomposition with $N_{\text{dd}} = 5$ is shown in Figure \ref{fig:nonovlp_multidomain}.
The letters $\texttt{D}$ and $\texttt{N}$ in this figure indicate whether a Dirichlet or Neumann BC 
is specified on a given boundary as seen from sub-domain $\Omega_i$ during the Schwarz iteration procedure, with 
%, respectively, the Dirichlet and Neumann boundary conditions that we set the internal boundaries for each sub-domain 
$i\in \{1, ..., 5\}$. %Indeed, the internal sub-domains are solved as mixed Dirichlet-Neumann problems. 
%The setting of external boundary conditions associated with the global domain is the same as in the $2$-domain case.
As before, we assess and compare the performance of the Schwarz alternating method with: (i) classical relaxation, (ii) Aitken acceleration with $N_0 = 5$ and $\rho^{(1)}=0.5$, and (iii) Anderson acceleration with memory adaptation based on $m_{\rm{and}}=20$, $\epsilon_{\text{and}} = 10^{-5}$ and
% not important: (which gives an approximate value of $\bar{k}=20$),  
$\bar{m} = 3$.
The multi-domain implementation of Aitken acceleration is a natural 
extension of the  procedure described in Section \ref{sec:acceleration} to the multi-domain setting, in which 
the Aitken acceleration procedure is applied to each set of neighboring sub-domains independently.  The approach is 
based on the assumption that the inverse Jacobian $\mathbf{G}^{(k)}$ in \eqref{eq:multi-secant} %of the residuals in Remark \ref{remark:secant_eq} 
is constant for all the sub-domains $i=1, \ldots, N_{\rm{dd}}$, 
and corresponds to the Aitken variant termed the ``Diagonal Aitken'' technique in 
\cite{deparis2004numerical}. 
In all our studies, we used a value of $\texttt{maxit}=80$ for the maximum allowable number of Schwarz iterations.
We employed Anderson with memory adaptation, as it reduced the number of iterations required to reach convergence by a factor of three with 
respect to the standard Anderson method with a fixed $m_k$.

Figures \ref{fig:meth_comparison_Ndd5}(a)--(c) study the convergence and performance of the Dirichlet-Neumann Schwarz algorithm with the different acceleration techniques considered for $N_{\text{dd}} = 5$.  The reader can observe that Anderson-accelerated Schwarz is the clear winner in terms of not only the number of Schwarz iterations required to reach convergence (Figures \ref{fig:meth_comparison_Ndd5}(a) and (c)), but also in terms of the CPU time per Schwarz iteration (Figure \ref{fig:meth_comparison_Ndd5}(b)). Generally, Anderson-accelerated Schwarz converges in approximately $3 \times$ fewer iterations than Aitken-accelerated Schwarz.  Curiously, the convergence of Aitken-accelerated Schwarz is actually worse than that of Schwarz with classical relaxation.  These results 
differ from our previously-reported results for the specific case of $N_{\text{dd}} = 2$, which showed that Aitken-accelerated Schwarz often outperforms, or performs comparably to, Anderson-accelerated Schwarz (see Tables \ref{tab:methods_comparison_rho02}--\ref{tab:methods_comparison_rho05} and Figure \ref{fig:meth_comparison}).  The poor convergence of Aitken-accelerated Schwarz is consistent with previously-reported results \cite{deparis2004numerical}, which showed non-convergence of this algorithm for $N_{\text{dd}} > 2$ when applied to a fluid-structure interaction problem.
It is somewhat surprising that Anderson-accelerated Schwarz 
has the lowest CPU time per Schwarz iteration (Figure \ref{fig:meth_comparison_Ndd5}(b)), as it is the only method that requires the online numerical solution of an optimization problem.  This result is likely due to the fast optimized algorithms within the $\texttt{quadrprog}$ {\tt MATLAB} routine, which we use to solve the Anderson minimization problem.   It is interesting to remark that the Schwarz alternating method accelerated using Anderson with memory adaptation exhibited some sensitivity to the $m_{\text{and}}$ parameter.  In particular, setting $m_{\text{and}}$ to a smaller value resulted in more iterations required to reach convergence.  For example, with $m_{\text{and}} = 5$, the method converged in 54 iterations, whereas it converged in just 26 iterations for $m_{\text{and}} = 20$.

We performed additional tests using two other values of the Aitken acceleration parameter: %, specifically, we set 
$\rho^{(1)}=0.2$ and $\rho^{(1)}=0.7$. For these choices, the Aitken acceleration approach negatively affects the performance of the local Newton solver and prevents the Schwarz algorithm from converging. In contrast, Anderson-accelerated Schwarz appears to be more robust with respect to the choice of $\rho$, converging in 
%: when using Anderson technique, the accelerated Schwarz algorithm converges in 
$60$  and $29$ iterations for $\rho=0.2$ and $\rho=0.7$, respectively.  

\begin{remark}
    There are more sophisticated variants of the Aitken algorithm, e.g., the ``Block Aitken'' \cite{deparis2004numerical} and the sparse Aitken-Schwarz methods \cite{Berenguer:2022}.
    %    might work better, like the block relaxation approach in \cite{deparis2004numerical} and the sparse Aitken-Schwarz method in \cite{Berenguer:2022}. 
    These and other variants of Aitken-accelerated Schwarz may be studied in a future work.
\end{remark}

Finally, we analyze how the performance and scalability of Schwarz with our various acceleration techniques depends on the number of sub-domains. In Figure \ref{fig:meth_comparison_all_Ndd}, we present both the average of the per-iteration CPU time (Figure \ref{fig:meth_comparison_all_Ndd}(a)), and the number of total Schwarz iterations needed to achieve convergence (Figure \ref{fig:meth_comparison_all_Ndd}(b)) for $N_{\rm{dd}}\in \{2,...,5\}$. We observe that the number of Schwarz iterations required for convergence for Aitken-accelerated Schwarz scales roughly linearly with $N_{\text{dd}}$; in contrast, Anderson-accelerated Schwarz appears to be much more robust to increases in $N_{\rm{dd}}$ (Figure \ref{fig:meth_comparison_all_Ndd}(b)). Again, we see that Anderson-accelerated Schwarz does not require a significantly higher CPU costs than the other two techniques (Figure \ref{fig:meth_comparison_all_Ndd}(a)).

\begin{figure}[h]
\subfloat[]{
 \begin{tikzpicture}
  	\begin{axis}[
   		xmode=linear,
   		ymode=linear,
   		grid=both,
  		minor grid style={gray!25},
   		major grid style={gray!25},
   		title = {},
   		xlabel={$N_{\rm{dd}}$},
   		ylabel={CPU cost $[s]$},
        ymin=0,
   		width=0.3\textwidth,
   		legend style={at={(2.3,0.28)},anchor=east},
   		]
        \addplot[ % <-- plot options
   		black,
   		mark=*,
        mark size=3,    
   		mark options={ color = black},
   		]  %
   		table{dat/Nddcost_Ndd2_relax.dat};
        \addplot[ % <-- plot options
   		black,
   		mark=square,
        mark size=3,    
   		mark options={ color = blue},
   		]  %
   		table{dat/Nddcost_Ndd2_aitk.dat};
        \addplot[ % <-- plot options
   		black,
   		mark=diamond,
        mark size=3,    
   		mark options={ color = red, scale=1.5},
   		]  %
   		table{dat/Nddcost_Ndd2_and.dat};
   		\addplot[ % <-- plot options
   		black,
   		mark=*,
        mark size=3,    
   		mark options={ color = black},
   		]  %
   		table{dat/Nddcost_Ndd3_relax.dat};
   		\addlegendentry{\footnotesize classical relax};	
        \addplot[ % <-- plot options
   		black,
   		mark=square,
        mark size=3,    
   		mark options={ color = blue},
   		]  %
   		table{dat/Nddcost_Ndd3_aitk.dat};
   		\addlegendentry{\footnotesize Aitken};	
        \addplot[ % <-- plot options
   		black,
   		mark=diamond,
        mark size=3,    
   		mark options={ color = red, scale=1.5},
   		]  %
   		table{dat/Nddcost_Ndd3_and.dat};
   		\addlegendentry{\footnotesize Anderson};	
        \addplot[ % <-- plot options
   		black,
   		mark=*,
        mark size=3,    
   		mark options={ color = black},
   		]  %
   		table{dat/Nddcost_Ndd4_relax.dat};
        \addplot[ % <-- plot options
   		black,
   		mark=square,
        mark size=3,    
   		mark options={ color = blue},
   		]  %
   		table{dat/Nddcost_Ndd4_aitk.dat};
        \addplot[ % <-- plot options
   		black,
   		mark=diamond,
        mark size=3,    
   		mark options={ color = red, scale=1.5},
   		]  %
   		table{dat/Nddcost_Ndd4_and.dat};
        \addplot[ % <-- plot options
   		black,
   		mark=*,
        mark size=3,    
   		mark options={ color = black},
   		]  %
   		table{dat/Nddcost_Ndd5_relax.dat};
        \addplot[ % <-- plot options
   		black,
   		mark=square,
        mark size=3,    
   		mark options={ color = blue},
   		]  %
   		table{dat/Nddcost_Ndd5_aitk.dat};
        \addplot[ % <-- plot options
   		black,
   		mark=diamond,
        mark size=3,    
   		mark options={ color = red, scale=1.5},
   		]  %
   		table{dat/Nddcost_Ndd5_and.dat};
   	\end{axis}
\end{tikzpicture}
}
\hspace{0.2cm}
\subfloat[]{
 \begin{tikzpicture}
  	\begin{axis}[
   		xmode=linear,
   		ymode=linear,
   		grid=both,
  		minor grid style={gray!25},
   		major grid style={gray!25},
   		title = {},
        ymin=0,
        ymax=85,
        axis x line=bottom,
        axis y line=left,
   		xlabel={$N_{\rm{dd}}$},
   		ylabel={Schwarz iter},
   		width=0.3\textwidth,
   		legend style={at={(1.9,0.55)},anchor=east},
   		]
        \addplot[only marks,% <-- plot options
   		black,
   		mark=*,
        mark size=3,    
   		mark options={ color = black},
   		]  %
   		table{dat/iters_allNdd_relax.dat};
      %  \addlegendentry{classical relax};	
         \addplot[ only marks,% <-- plot options
   		black,
   		mark=square,
        mark size=3,    
   		mark options={ color = blue},
   		]  %
   		table{dat/iters_allNdd_aitk.dat};
       % \addlegendentry{Aitken};	
         \addplot[ only marks,% <-- plot options
   		black,
   		mark=diamond,
        mark size=3,    
   		mark options={ color = red, scale=1.5},
   		]  %
   		table{dat/iters_allNdd_and.dat};
      %  \addlegendentry{Anderson};	
        \end{axis}
\end{tikzpicture}
}
    \caption{2D nonlinear elasticity with $N_{\rm{dd}} \in \{2,3,4,5\}$.  (a) $N_{\text{dd}}$ vs. the average CPU time (in seconds) per Schwarz iteration.  (b) $N_{\text{dd}}$ vs. the number of Schwarz iterations required to reach convergence.} % We recall that $\texttt{maxit}=80$. }
 \label{fig:meth_comparison_all_Ndd}
\end{figure}
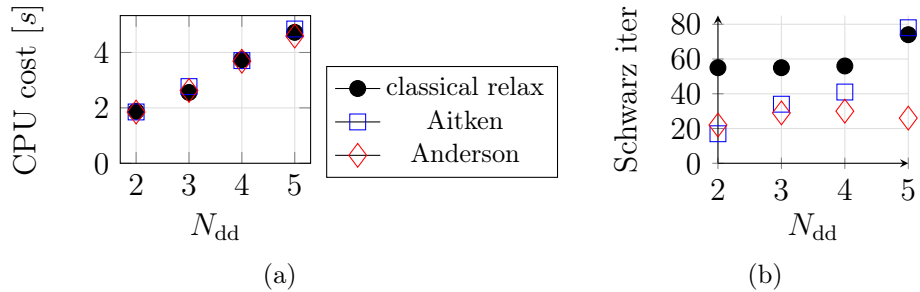

\section{Conclusions and perspectives}
\label{sec:conclusions}
%\irinanote{I will rework this section later once the body of the paper is more finalized.}

In this work, we have presented a theoretical and numerical comparison of the impact of various relaxation and acceleration techniques on the convergence of the multiplicative Schwarz alternating method with Dirichlet-Neumann transmission conditions in the context of DD-based coupling and elliptic BVPs.  For a 1D model problem, we proved quadratic convergence of Aitken-accelerated Schwarz under some conditions, and uncovered a link between
 %We completed the investigation with theoretical insights on the accelerated Schwarz convergence rate and on the link between 
 the spectral properties of the underlying fixed-point operator and the Anderson acceleration coefficients in Anderson-accelerated Schwarz.   We compared both theoretically and numerically the various acceleration techniques on a simple linear PDE, the 1D Laplace equation.  We also studied the methods' accuracy, performance and robustness (to both algorithmic parameters and the number of sub-domains being coupled) on a 2D nonlinear elasticity problem.  For the  two sub-domain setting, we demonstrated that the convergence of Aitken-accelerated Schwarz %Aitken acceleration 
 is not particularly sensitive to the two tuning parameters in this algorithm,  $\rho^{(1)}$ and $N_0$, which makes the algorithm appealing to users as minimal parameter tuning/optimization is required.  
 %depends only loosely on the initial parameter for a steady two-domain setting, Aitken acceleration depends very loosely on the initial parameter $\rho^{(1)}$ --- since the parameter is adaptive and as shown in Table \ref{tab:methods_comparison_rho02}, 
 We also demonstrated that the method can be twice as fast as Anderson-accelerated Schwarz and Schwarz with classical relaxation.  The story is different when considering greater numbers of sub-domains, where Anderson-accelerated Schwarz is the most performant and robust method, requiring $4\times$ fewer iterations to converge than the other methods evaluated.  We discovered that, while Anderson-accelerated Schwarz has a larger memory footprint, its computational cost is approximately equal to Aitken-accelerated Schwarz and Schwarz with classical relaxation.  Our observation that Anderson-accelerated Schwarz can be sensitive to its input parameter $m_{\text{and}}$, which controls the size of the history utilized within the Anderson update formula, led to the development of a new adaptive version of Anderson, termed ``Anderson with memory adaptation'', which leads to a more robust and performant Schwarz scheme. 

The research summarized in this work has informed several directions for future work, which include: 
(i) numerical investigations of the accelerated methods' performance on problems with sub-domains discretized using non-conformal meshes, (ii) explorations of different variants of Aitken-accelerated Schwarz for the multi-domain scenario (e.g., the ``Block Aitken'' method in \cite{deparis2004numerical}), (iii) reformulations of the Anderson update formula such that the relaxation parameter $\rho$ is updated dynamically rather than pre-specified, (iv) extensions to 3D and time-dependent problems, and (v) the development of acceleration strategies for the Robin-Robin version of the non-overlapping Schwarz alternating method.

\appendix
\section{Derivation of the Aitken acceleration formula \eqref{eq:aitken}} 
\label{app:aitken}
To show that the Aitken acceleration formula \eqref{eq:aitken} is obtained when minimizing the function \eqref{eq:aitken_of}, we follow the procedure in Section $4.2.3$ of \cite{deparis2004numerical}.  In Section \ref{sec:fixed_point}, we wrote the Dirichlet-Neumann Schwarz algorithm as a fixed-point iteration (c.f., equation \eqref{eq:fp_general}). To simplify the notation here, we consider conforming interfaces, so that $\gamma=\gamma_1=\gamma_2$ and $\Pi_{12}=\Pi_{21}=\mathbb{I}$.
\begin{subequations}
The sub-domain interface jump \eqref{eq:int_jump} also simplifies to 
$ \mathcal{E}(\gamma(\mathbf{u}_1), \gamma(\mathbf{u}_2))=\gamma(\mathbf{u}_2)- \gamma(\mathbf{u}_1)$.
The fixed-point iteration reads $\mathbf{g}^{\star}=T(\mathbf{g}^{\star})$ if and only if $\mathcal{E}(\gamma(\mathbf{u}^{\star}_1), \gamma(\mathbf{u}^{\star}_2))=\mathbf{0}$, with $\mathbf{u}_1^{\star}$ and $\mathbf{u}_2^{\star}$ the solution to the Dirichlet and Neumann problems once $\mathbf{g}^{\star}$ is given.
We denote by $\tilde{\boldsymbol{\rho}}$ the approximation of the inverse of Jacobian function with respect to the discrepancy. Following \cite{deparis2004numerical}, we approximate the inverse of the Jacobian
by a scalar time the identity, i.e. $\tilde{\boldsymbol{\rho}}=-\rho \mathbb{I}$.  
Now, if we apply Newton's method, we obtain, for $k>0$,
\begin{equation}
\begin{aligned}
\gamma(\mathbf{u}_{1}^{\star,(k-1)})&=\gamma(\mathbf{u}^{(k-1)}_{1})-\tilde{\boldsymbol{\rho}}\mathcal{E}(\gamma(\mathbf{u}_1^{(k-1)}),\gamma(\mathbf{u}_2^{(k-1)}))\\
\gamma(\mathbf{u}^{\star,(k)}_{1})&=\gamma(\mathbf{u}^{(k)}_{1})-\tilde{\boldsymbol{\rho}}\mathcal{E}(\gamma(\mathbf{u}_1^{(k)}),\gamma(\mathbf{u}_2^{(k)}))
\end{aligned}
\label{eq:quasi_newton_aitken}
\end{equation}
at  iteration $k-1$ and $k$. 
Since system \eqref{eq:quasi_newton_aitken} is not well-defined, we employ a least-squares technique to find the optimal $\rho$:
\begin{align*}
\rho^{(k)}&=\arg\min_{\rho \in \mathbb{R}} \big\|\gamma(\mathbf{u}^{\star, (k)}_{1})-\gamma(\mathbf{u}^{\star, (k-1)}_{1})\big\|_2^2\\
&=\arg \min_{\rho \in \mathbb{R}}\Big\|\underbrace{\left(\gamma(\mathbf{u}^{(k)}_{1})-\gamma(\mathbf{u}^{(k-1)}_{1})\right)}_{\mathbf{d}^{(k)}}+\rho \underbrace{\left(\mathcal{E}\Big(\gamma(\mathbf{u}_{1}^{(k)}),\gamma(\mathbf{u}_2^{(k)})\Big)-\mathcal{E}\Big(\gamma(\mathbf{u}_1^{(k-1)}),\gamma(\mathbf{u}_2^{(k-1)})\Big)\right)}_{\boldsymbol{\delta}^{(k)}}\Big\|_2^2\\
&=\arg \min_{\rho \in \mathbb{R}} \big\|\mathbf{d}^{(k)}+\rho \boldsymbol{\delta}^{(k)}\big\|_2^2.
\end{align*}
It is easy to see that the solution is \eqref{eq:aitken}.
%We write the solution as
%\begin{equation}
%\rho=-\frac{\boldsymbol{\delta}^{(k)}\cdot \mathbf{d}^{(k)}}{\|\boldsymbol{\delta}^{(k)}\|_2^2},
%    \label{eq:aitken_compact_formula}
%\end{equation}
% which is equivalent to
%\begin{equation}
%\rho^{(k)}=-\frac{\left(\mathcal{E}\Big(\gamma(\mathbf{u}_1^{(k)}),\gamma(\mathbf{u}_2^{(k)})\Big)-
%\mathcal{E}\Big(\gamma(\mathbf{u}_1^{(k-1)}),\gamma(\mathbf{u}_2^{(k-1)})\Big)
%\right) \cdot \left(\gamma(\mathbf{u}_1^{(k)})-\gamma(\mathbf{u}_1^{(k-1)})\right)}
%{
%\Big\|\left(\mathcal{E}\Big(\gamma(\mathbf{u}_1^{(k)}),\gamma(\%mathbf{u}_2^{(k)})\Big)-
%\mathcal{E}\Big(\gamma(\mathbf{u}_1^{(k-%1)}),\gamma(\mathbf{u}_2^{(k-1)})\Big)
%\right)\Big\|_2^2
%}.
%\end{equation}
\end{subequations}

\section*{Acknowledgments} \label{sec:acknowl}  Support for this work was received through Sandia National Laboratories' (SNL's) Laboratory Directed Research and Development (LDRD) program, as well as through the U.S. Department of Energy (DOE), Office of Science, Office of Advanced Scientific Computing Research Mathematical Multifaceted Integrated Capability Centers (MMICCs) program (under Field Work Proposal 22025291 and the Multifaceted Mathematics for Predictive Digital Twins project) and Applied Mathematics Competitive Portfolios program. Additionally, the writing of this manuscript was funded in part by I. Tezaur’s Presidential Early Career Award for Scientists and Engineers (PECASE).  
The corresponding 
author gratefully acknowledges support from the SIAM Postdoctoral Support Program fellowship.
SNL is a multi-mission laboratory managed and operated by National Technology and Engineering Solutions of Sandia, LLC., a wholly owned subsidiary of Honeywell International, Inc., for the U.S. DOE's National Nuclear Security Administration under contract DE-NA0003525.\\
The authors also thank Dr. Tommaso Taddei for his helpful comments on an earlier version of this work.

%The authors wish to thank Alejandro Mota for creating the {\tt Norma.jl} code in which our numerical examples are implemented, and for assisting with the setup of our numerical experiments within this code.  

%\section*{Acknowledgments}
%We would like to acknowledge the assistance of volunteers in putting
%together this example manuscript and supplement.

\bibliographystyle{elsarticle-num} 
\bibliography{ref}



\end{document}